\documentclass[10pt]{article}

\usepackage{Paper}

\newcommand{\e}{\mathrm{e}}
\renewcommand{\i}{\mathrm{i}}
\newcommand{\st}{\ |\ }

\newcommand{\bfx}{\mathbf{x}}
\newcommand{\bfxi}{\boldsymbol{\xi}}

\DeclareMathOperator*{\maxs}{\max\vphantom{p}}

\begin{document}

\allowdisplaybreaks


\title{On function spaces for radial functions}
\author{Mark D. Groves}
\author{Dan J. Hill}

\affil{\small Fachrichtung Mathematik, Universit\"at des Saarlandes, Postfach 151150, 66041 Saarbr\"ucken, Germany}

\date{}
\maketitle


\begin{abstract}
\noindent
This paper is concerned with complex Banach-space valued functions of the form
$$
\hat{f}_k(r\cos\theta,r\sin\theta,z)=\e^{\i k \theta}f_k(r,z), \qquad r \in [0,\infty),\ \theta \in {\mathbb T}^1,\ z \in {\mathbb R}, 
$$
for some $k \in {\mathbb Z}$. It is demonstrated how classical and Sobolev spaces for the radial function $f_k$ can be
constructed in a natural fashion from the corresponding standard function spaces for $\hat{f}_k$. A
theory of radial distributions is derived in the same spirit. Finally, a new class of \emph{Hankel spaces} for the case $f_k=f_k(r)$
is introduced. These spaces are the radial counterparts of the familiar Bessel-potential spaces for functions defined on ${\mathbb R}^d$. 
The paper concludes with an application of the theory to the Dirichlet boundary-value problem for Poisson's equation in a cylindrical domain.

\end{abstract}

\section{Introduction}

\subsection*{Background}

The selection of suitable function spaces is an integral part of the analysis of partial differential equations.
Concepts such as weak and strong solutions and variational principles, as well as techniques such as contraction-mapping principles, the Fourier transform and Green's functions all rely heavily on these choices. In many physical settings---such as heat conduction, acoustics and planetary motion---it may be convenient to formulate the mathematical problem in polar coordinates, replacing the planar Cartesian coordinates $(x,y)$ by the planar radius $r\geq0$ and angle $\theta\in {\mathbb T}^1$. 
Following the development of the radial spatial dynamics theory of Scheel \cite{Scheel03}, there has been a renewed interest in utilising polar coordinates to study partial differential equations arising in pattern formation,
either through an angular Fourier expansion
\begin{equation*}
    \hat{u}(r\cos \theta,r \sin \theta) = \sum_{k\in\mathbb{Z}} \mathrm{e}^{\mathrm{i}k\,\theta} u_k(r),
\end{equation*}
or by restricting to the special case of axisymmetric functions $\hat{u} = u_0(r)$. A notable area of progress in recent years comes from the study of spatial localisation in higher dimensions; e.g.\ see
Lloyd \& Sandstede \cite{LloydSandstede09}, McCalla \& Sandstede \cite{McCallaSandstede13}, and Byrnes \emph{et al.} \cite{ByrnesCarterDoelmanLiu23} for localised axisymmetric patterns, and Hill, Bramburger \& Lloyd \cite{HillBramburgerLloyd23,HillBramburgerLloyd24} for approximate localised dihedral patterns. The partial differential equations in many of these examples are fairly simple, and do not require some of the more sophisticated tools of functional analysis. However, if one wishes to study more complex partial differential equations, such as those arising in the study of axisymmetric fluid jets (e.g.\ see Alt, Caffarelli \& Friedman \cite{AltCaffarelliFriedman83}, Erhardt, Wahl\'{e}n \& Weber \cite{ErhardtWahlenWeber22}, and Groves \& Nilsson \cite{GrovesNilsson18}) or spikes on the surface of a ferrofluid (Hill, Lloyd \& Turner \cite{HillLloydTurner21}), then one must first establish a convenient framework of radial function spaces in order to utilise the full range of available analytic tools. \pagebreak

Of course, this problem is not novel. Many previous rigorous studies have considered axisymmetric solutions to
partial differential equations posed in polar coordinates, and by necessity deal with the question of radial function spaces. However, the authors typically either define their own \emph{ad hoc} function spaces (e.g.\ see
Ericsson \cite{Ericsson21} and Liu \& Wang \cite{LiuWang09} for the Stokes and Navier--Stokes equations) or first massage their problem to fit in a standard function-space framework (e.g.\ see Varvaruca \& Weiss \cite{VarvarucaWeiss14}
for water waves and Wang \& Yang \cite{WangYang19} for ferrofluid jets).
In fact it appears many of the radial analogues of standard function spaces are not well established, even in the case of the classical spaces of continuously differentiable functions. The aim of the present work is to present a clear and comprehensive framework for radial function spaces that encompasses many of the standard spaces connected with the analysis of partial differential equations. 

The usual practice for symmetric functions of the form $\hat{u}_0(\bfx) := u_0(|\bfx|)$ in ${\mathbb R}^d$ is to find a function space for the radial function $u_0$ which characterises the membership of
the symmetric function $\hat{u}_0$ in a given standard space; for example, de Figueiredo, dos Santos \& Miyagaki \cite{deFigueiredodosSantosMiyagaki11}   introduce
weighted Sobolev spaces for $u_0(r)$ which guarantee that $\hat{u}_0$ belongs to the subspace of real, symmetric functions in the Sobolev space $H^m(B)$, where $B$
is the open unit ball in ${\mathbb R}^d$ (see also Ostermann \cite{Osterman22} for a complete characterisation). Similarly, Bernardi, Dauge \& Maday \cite{BernardiDaugeMaday} and
Costabel, Dauge \& Hu \cite{CostabelDaugeHu23} consider \emph{mode $k$ functions} $\hat{u}_k=\hat{u}_k(\bfx)$ in ${\mathbb R}^2$ which have the property that
$$
\hat{u}_k(r\cos\theta,r\sin\theta)=\e^{\i k \theta}u_k(r), \qquad r \geq 0,\ \theta \in {\mathbb T}^1,
$$
for some \emph{radial coefficient} $u_k=u_k(r)$, where $u_k(0)=0$ if $k \neq 0$; note in particular the distinction between
\emph{axisymmetric functions} (mode $0$ functions, i.e.\ rotation invariant functions defined on ${\mathbb R}^2$) and \emph{radial functions}
(functions defined on the half-line $\{r\in[0,\infty)\}$).
The above authors show that $\hat{u}_k$ belongs to $H^m(B_R(\mathbf{0});{\mathbb C})$, where $B_R(\mathbf{0})$ is the open origin-centred ball of radius $R$ in ${\mathbb R}^2$,
if and only if $u_k$ belongs to a weighted Sobolev space
$H^m_{(k)}$ (which is characterised in different, but equivalent, ways in the two references). 
The characterisations of $H^m_{(k)}$ are however
not particularly intuitive, and do not address in particular the following question directly. Suppose that
$\hat{u}_0(\bfx)=u_0(|\bfx|)$, $\bfx \in {\mathbb R}^2$ is an axisymmetric function in $H^m(B_R(\mathbf{0});{\mathbb C})$ (so that $u_0 \in H^m_{(0)}$); in what space does $u_0^\prime(|\bfx|)$ lie? Consider for example the function $u_0(r):=\mathrm{sech}(r)$, and note that
    \begin{itemize}
        \item $\hat{u}_0(\bfx):=u_0(|\bfx|)$ is axisymmetric and smooth in $\mathbb{R}^{2}$ (see Figure~\ref{fig:rad_smooth}(a));
        \item $\hat{u}_1(\bfx):= u_0^\prime(|\bfx|)$ is axisymmetric but not smooth in $\mathbb{R}^{2}$ (see Figure~\ref{fig:rad_smooth}(b));
        \item $u_2(\bfx):= \Delta u_{0}(|\bfx|) =u_{0}''(|\bfx|) + \frac{1}{|\bfx|}u_{0}'(|\bfx|)$ is axisymmetric and smooth in $\mathbb{R}^{2}$ (see Figure~\ref{fig:rad_smooth}(c));
        \item $\hat{u}_{3}(\bfx) := \cos(\arg(\bfx))u_{0}'(|\bfx|)$ is not axisymmetric but is smooth in $\mathbb{R}^{2}$ (see Figure~\ref{fig:rad_smooth}(d)).
    \end{itemize}
This example illustrates that \textit{the radial derivative of a smooth axisymmetric function is not in general a smooth axisymmetric function}.
It also illustrates the facts (see below) that the Laplacian of a smooth axisymmetric function is, however, a smooth axisymmetric function, and, perhaps more surprisingly,
that the radial derivative of a smooth axisymmetric function is the radial coefficient of a smooth mode $\pm 1$ function.

\begin{figure}[ht!]
    \centering
    \includegraphics[scale=0.45]{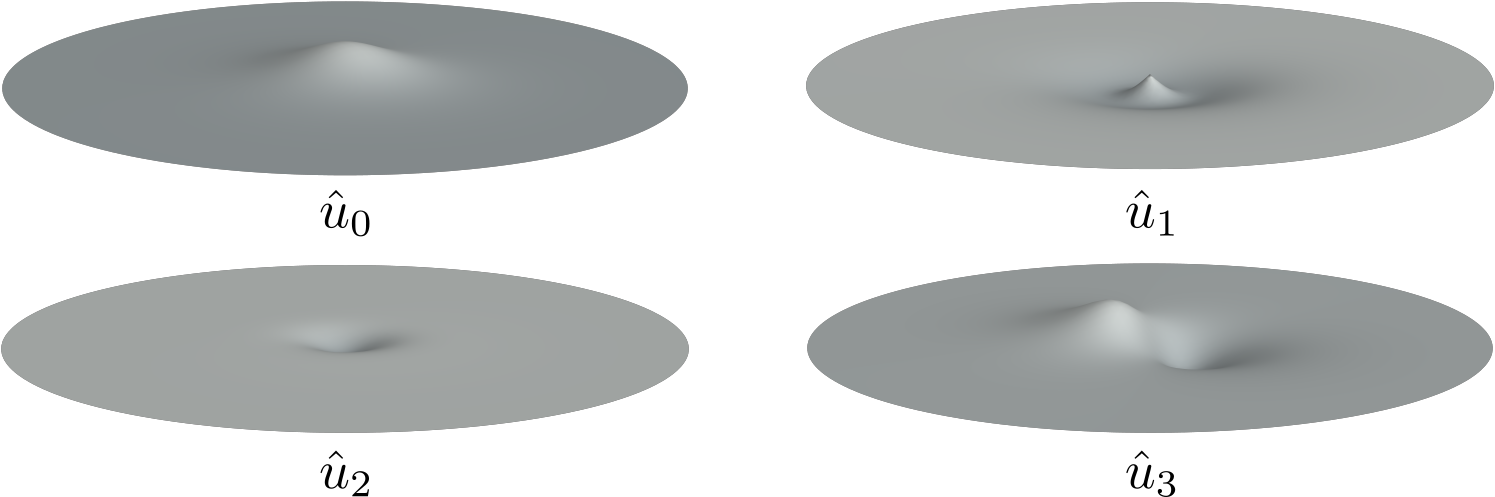}
    \caption{Planar plots of the functions $\hat{u}_0(\bfx),\hat{u}_1(\bfx),\hat{u}_2(\bfx),\hat{u}_3(\bfx)$ defined above. Each function is smooth except for $\hat{u}_1$, which has a cusp at the origin $\bfx=\mathbf{0}$.}
    \label{fig:rad_smooth}
\end{figure}

The above example shows that \emph{it is not sufficient to only consider axisymmetric functions when introducing an analytical framework for radial functions},
and more generally it is not immediately obvious whether derivatives of a radial function in a given space also lie in that space.
This remark is particularly pertinent when utilising radial spatial dynamics, where radial derivatives play a central role. In this paper we introduce a framework of radial function spaces with two aims: the radial spaces should correspond in a natural fashion to (subspaces of) standard
spaces of mode $k$ functions defined on ${\mathbb R}^2$, such that
standard results for the latter spaces are readily transferred to the former; and (appropriately defined) derivatives of a radial function should map one space into
another according to clear rules and without leaving the framework.

\subsection*{Overview}

We begin our task in Section \ref{s:DiffOp} by exploiting the structure inherent to polar coordinates, rather than attempting to employ 
a more Cartesian-minded framework. Following Costabel, Dauge \& Hu \cite{CostabelDaugeHu23}, we consider the effect of applying Wirtinger-type complex differential operators $\partial_{\overline{\zeta}}^{n-i}\partial_{\zeta}^{i}$ with
$$
\partial_{\zeta} := \tfrac{1}{\sqrt{2}}(\partial_{x} - \mathrm{i}\partial_{y}), \qquad \partial_{\overline{\zeta}} := \tfrac{1}{\sqrt{2}}(\partial_{x} + \mathrm{i}\partial_{y})
$$
to mode $k$ functions; explicit calculations indicate
that one should replace the 
differential operator $\partial_r$ with the \emph{$k$-index Bessel operator}
$$
    \mathcal{D}_{k} := r^{-k} \frac{\mathrm{d}}{\mathrm{d}r} r^{k} = \frac{\mathrm{d}}{\mathrm{d}r} + \frac{k}{r}
$$
and the higher-order differential operator $\partial_r^i$ with
$$
    \mathcal{D}^{i}_{k} := r^{-k+i} \left(\frac{1}{r}\frac{\mathrm{d}}{\mathrm{d}r}\right)^{i} r^{k} = \mathcal{D}_{k-(i-1)}\mathcal{D}_{k-(i-2)}\dots\mathcal{D}_{k-1}\mathcal{D}_{k}.
$$

\begin{proposition}
Let $\hat{f}_{k}:\mathbb{R}^{2} \to {\mathbb C}$ be a mode $k$ function with radial coefficient
$f_k:[0,\infty) \rightarrow {\mathbb C}$.
It follows that
$$
       \partial_{\zeta}\hat{f}_{k}= \mathrm{e}^{\mathrm{i}(k-1)\theta}\tfrac{1}{\sqrt{2}}\mathcal{D}_{k}f_k, \qquad \partial_{\overline{\zeta}}\hat{f}_{k} = \mathrm{e}^{\mathrm{i}(k+1)\theta}\tfrac{1}{\sqrt{2}}\mathcal{D}_{-k}f_k,
$$
 and more generally
        \begin{equation}
        \partial_{\overline{\zeta}}^{n-i}\partial^{i}_{\zeta}\hat{f}_{k} = \mathrm{e}^{\mathrm{i}(k+n-2i)\theta}\,2^{-\frac{n}{2}}\mathcal{D}_{-k+i}^{n-i}\mathcal{D}^{i}_{k}f_k
        \label{eq:WirtingervsBessel}
    \end{equation}
for $n \in {\mathbb N}_0$ and $0 \leq i \leq n$.
\end{proposition}

According to the above proposition the operator
$\partial_\zeta$ maps a mode $k$ function with radial coefficient $f_k$
to a mode $k-1$ function with radial coefficient ${\mathcal D}_k f_k$ while $\partial_{\bar\zeta}$ maps a mode $k+1$ function with radial coefficient ${\mathcal D}_{-k} f_k$.
Correspondingly, one finds that ${\mathcal D}_k$ and ${\mathcal D}_{-k}$ map a mode $k$ radial coefficient to a mode $k-1$ and a mode $k+1$ radial coefficient respectively,
as illustrated diagrammatically in Figure \ref{fig:Fkspace}; it follows from the calculation
$$
    \mathcal{D}_{\ell}^{n}\mathcal{D}_{k}^{m} = \mathcal{D}_{k+n}^{m}\mathcal{D}_{\ell-m}^{n}
$$
that the diagram commutes.
Note that it is actually not necessary to distinguish between mode $k$ and mode $-k$
radial coefficients since the radial coefficient of the mode $k$ function $\hat{f}_k(x,y)$ is also the radial coefficient of the mode $-k$ function $\hat{f}_k(x,-y)$
(which explains the apparent ambiguity in this interpretation of ${\mathcal D}_0$.)\pagebreak

\begin{figure}[ht]
    \centering
    \includegraphics[width=0.55\textwidth]{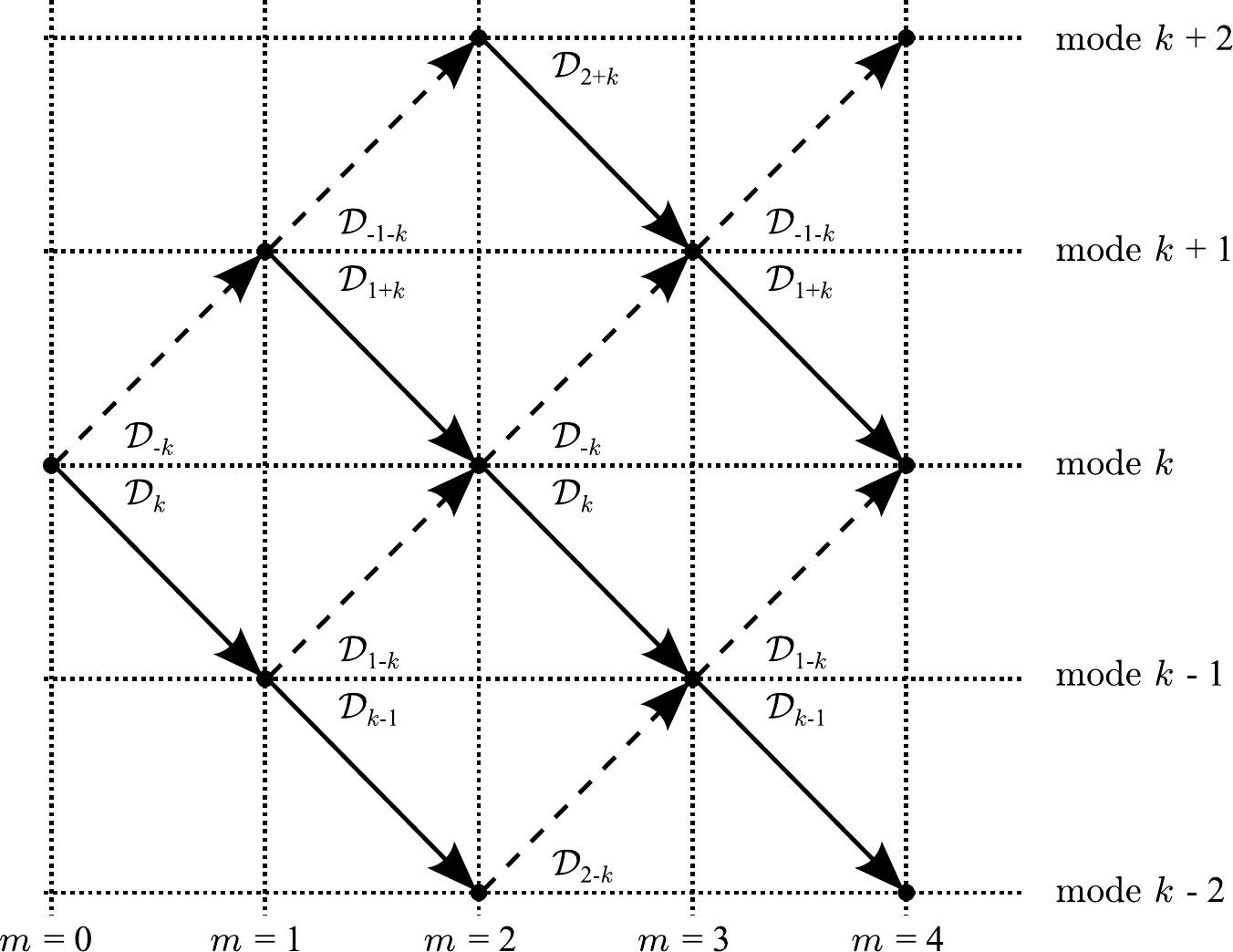}
    \caption{Actions of the Bessel operators. The mode $k+j$ function in column $m$ is $\e^{\i(k+j)\theta}{\mathcal D}_{k_m} {\mathcal D}_{k_{m-1}} \ldots {\mathcal D}_{k_1} f_k$,
    where the indices $\{k_i\}_{i=1}^{m}$ satisfy $k_i=\pm (k_{i-1} - 1)$, with $k_1=\pm k$, and consist of $\frac{1}{2}(m-j)$ positive and $\frac{1}{2}(m+j)$ non-positive terms.}
    \label{fig:Fkspace}
\end{figure}

In Section \ref{Classical FS} we examine spaces of classical functions using the above observations. Fix $R \in (0,\infty]$.
It is readily established that $\hat{f}_k$ belongs to the subset $\hat{C}_{(k)}^0(B_R(\mathbf{0});{\mathbb C})$ of mode $k$ functions in
$C^0(B_R(\mathbf{0});{\mathbb C})$ if and only if its radial coefficient $f_k$ belongs to
$$C_{(k)}^0([0,R);{\mathbb C}) := \{f_k \in C^0([0,R);{\mathbb C}) \st kf_k|_{r=0}=0\}$$
(see Lemma \ref{lem:C0}). Observing that $\hat{f}_k \in C^m(B_R(\mathbf{0});{\mathbb C})$ if and only if $\partial_{\bar{\zeta}}^{n-i}\partial_\zeta^i \hat{f}_k$ belongs to
$C^0(B_R(\mathbf{0});{\mathbb C})$ for $0 \leq i \leq n\leq m$, we thus find from equation \eqref{eq:WirtingervsBessel} that
$\hat{f}_k$ belongs to the subset $\hat{C}_{(k)}^m(B_R(\mathbf{0});{\mathbb C})$ of mode $k$ functions in
$C^m(B_R(\mathbf{0});{\mathbb C})$ if and only if its radial coefficient $f_k$ belongs to
\begin{align*}
C_{(k)}^m&([0,R);{\mathbb C}) \\
&:= \{f_k: [0,R) \rightarrow {\mathbb C} \st \mathcal{D}_{-k+i}^{n-i}\mathcal{D}_k^i f_k \in C^0_{(k+n-2i)}([0,R) ;{\mathbb C}),\ 0\leq i\leq n \leq m\} \\
&\,= \{f_k: [0,R)  \rightarrow {\mathbb C} \st \mathcal{D}_{-k+i}^{n-i}\mathcal{D}_k^i  f_k \in C^0([0,R);{\mathbb C}),
(k+n-2i)\mathcal{D}_{-k+i}^{n-i}\mathcal{D}_k^if_k|_{r=0}=0,\ 0\leq i\leq n \leq m\}.
\end{align*}

We establish analogous results for other classes of continuous functions, in particular for the mode $k$ bounded continuously differentiable functions
$\hat{C}_{\mathrm{b}(k)}^m(B_R(\mathbf{0});{\mathbb C})$, test functions $\hat{\mathscr D}_{(k)}(B_R(\mathbf{0});{\mathbb C})$
and Schwartz-class functions $\hat{\mathscr S}_{(k)}({\mathbb R}^2;{\mathbb C})$. In all cases the mapping $f_k \mapsto \hat{f}_k$ defines a bijection between the
respective spaces and is furthermore an isometry $C_{(k)\mathrm{b}}^m([0,R);{\mathbb C}) \rightarrow \hat{C}_{(k)\mathrm{b}}^m(B_R(\mathbf{0});{\mathbb C})$ and ${\mathscr S}_{(k)}([0,\infty);{\mathbb C}) \rightarrow \hat{\mathscr S}_{(k)}({\mathbb R}^2;{\mathbb C})$
with respect to the natural (semi-)norms (see Section \ref{Classical FS} for details).

After discussing radial distributions in Section \ref{Distributions}, we present a similar approach to Sobolev spaces in Section
\ref{Sobolev FS}. A mode $k$ function $\hat{f}_k$ evidently belongs to $L^2(B_R(\mathbf{0});{\mathbb C})$ if and only if its radial coefficient $f_k$ belongs to
$$L_1^2((0,R);{\mathbb C}) = \left\{f: [0,R) \rightarrow {\mathbb C} \st 2\pi \int_0^R |f(r)|^2 r \dr< \infty\right\},$$
and this observation leads us to define
$\hat{H}_{(k)}^m(B_R(\mathbf{0});{\mathbb C})$ for $m \in {\mathbb N}_0$ as the completion of
$\hat{\mathscr D}_{(k)}(\overline{B}_R(\mathbf{0});{\mathbb C})$ with respect to the norm
$$\| \hat{f}_k\|_{H^m}^2 := \sum_{n=0}^m \sum_{i=0}^n \begin{pmatrix} n \\ i \end{pmatrix}\
\| \partial_{\bar{\zeta}}^{n-i} \partial_\zeta^i \hat{f}_k\|_{L^2}^2$$
and $H_{(k)}^m((0,R);{\mathbb C})$ for $m \in {\mathbb N}_0$
as the completion of ${\mathscr D}_{(k)}([0,R];{\mathbb C})$ with respect to the norm
$$\|f_k\|_{H^{m}_{(k)}}^2:= \sum_{n=0}^m 2^{-n} \sum_{i=0}^n \begin{pmatrix} n \\ i \end{pmatrix} 
\|{\mathcal D}_{-k+i}^{n-i}{\mathcal D}_k^i f_k\|_{L_1^2}^2$$
(\emph{cf.} equation \eqref{eq:WirtingervsBessel}); the spaces $\hat{H}_{(k)}^s(B_R(\mathbf{0});{\mathbb C})$
and $H_{(k)}^s((0,R);{\mathbb C})$ for $s \geq 0$ are defined by (complex) interpolation. Embedding and trace
theorems for these spaces are given in Section \ref{Sobolev FS}, where we exploit the fact that
the mapping $f_k \mapsto \hat{f}_k$ is an isometric isomorphism $H_{(k)}^s((0,R);{\mathbb C}) \to \hat{H}_{(k)}^s(B_R(\mathbf{0});{\mathbb C})$
by construction.

We also give an equivalent construction of $\hat{H}_{(k)}^m(B_R(\mathbf{0});{\mathbb C})$ and $H_{(k)}^m((0,R);{\mathbb C})$ for $m \in {\mathbb N}_0$
using weak derivatives (equations \eqref{eq:wkderiv 1 - intro} and \eqref{eq:wkderiv 2 - intro} below of course reduce
to integration-by-parts formulae in the special case
$\hat{f}_k \in \hat{\mathscr D}_{(k)}(\overline{B}_R(\mathbf{0});{\mathbb C})$, $f_k \in {\mathscr D}_{(k)}([0,R];{\mathbb C})$, such that
$\hat{g}_{k+n-2i} = \partial_{\bar{\zeta}}^{n-i}\partial_\zeta^i\hat{f}_k$, $g_{k+n-2i} = 2^{-\frac{n}{2}}{\mathcal D}_{-k+i}^{n-i}{\mathcal D}_k^i f_k$).

\begin{lemma}
Suppose that $\hat{f}_k \in L^2(B_R(\mathbf{0});{\mathbb C})$. The function $\hat{g}_{k+n-2i}\in L^2(B_R(\mathbf{0});{\mathbb C})$ is the weak derivative
$\partial_{\bar{\zeta}}^{n-i}\partial_\zeta^i\hat{f}_k$ of $\hat{f}_k$ if and only if
\begin{equation}
(-1)^n
\hspace{-2mm}\int\limits_{B_R(\mathbf{0})} \hspace{-2mm}\hat{f}_k \partial_{\zeta}^i \partial_{\bar{\zeta}}^{n-i} \hat{\phi}_{-k-n+2i}
=
\hspace{-2mm} \int\limits_{B_R(\mathbf{0})} \hspace{-2mm} \hat{g}_{k+n-2i}\hat{\phi}_{-k-n+2i} \label{eq:wkderiv 1 - intro}
\end{equation}
for all $\hat{\phi}_{-k-n+2i} \in \hat{\mathscr D}_{(-k-n+2i)}(B_R(\mathbf{0});{\mathbb C})$, and this condition holds if and only if
\begin{equation}
(-1)^n
 \int_0^R f_k(r)2^{-\frac{n}{2}}{\mathcal D}_{k+n-i}^{n-i} {\mathcal D}_{-k-n+2i}^i \phi_{-k-n+2i}(r)r\dr=  \int_0^R g_{k+n-2i}(r)\phi_{-k-n+2i}(r) r\dr
\label{eq:wkderiv 2 - intro}
\end{equation}
for all $\phi_{-k-n+2i} \in {\mathscr D}_{(-k-n+2i)}([0,R);{\mathbb C})$.
\end{lemma}

\begin{definition}
The function $2^\frac{n}{2}g_{k+n-2i} \in L_1^2((0,R);\mathbb{C})$
is the \underline{weak radial derivative} ${\mathcal D}_{-k+i}^{n-i}{\mathcal D}_k^i f_k$ of\linebreak
$f_k \in L_1^2((0,R);\mathbb{C})$
if it satisfies \eqref{eq:wkderiv 2 - intro} for all
$\phi_{-k-n+2i} \in {\mathscr D}_{(-k-n+2i)}([0,R);{\mathbb C})$.
\end{definition}

This construction leads us to the alternative definitions
\begin{align*}
\hat{H}_{(k)}^m(B_R(\mathbf{0});{\mathbb C}) &:=\{\hat{f}_k \in L^2(B_R(\mathbf{0});{\mathbb C})\st\partial_{\bar{\zeta}}^{n-i}\partial_\zeta^i \hat{f}_k \in L^2(B_R(\mathbf{0});{\mathbb C}),\ 0\leq i\leq n\leq m\}, \\
H_{(k)}^m((0,R);{\mathbb C})
&:=\{f_k \in L_1^2((0,R);{\mathbb C})\st {\mathcal D}_{-k+i}^{n-i}{\mathcal D}_k^i f_k \in L_1^2((0,R);{\mathbb C}),\ 0 \leq i \leq n \leq m\}
\end{align*}
for mode $k$ functions and their radial coefficients.
These definitions can also be understood to mean that appropriately defined distributional derivatives can be identified with functions in $L^2(B_R(\mathbf{0});{\mathbb C})$
and $L_1^2((0,R);{\mathbb C})$; this point is discussed in Sections \ref{Distributions} and \ref{Sobolev FS}.

\begin{remark}
Suppose that $\hat{\psi}:B_R(\mathbf{0}) \to {\mathbb C}$ and define a mode $k$ function ${\mathcal Q}_k[\hat{\psi}]$ by
$${\mathcal Q}_k[\hat{\psi}](\mathbf{x}):=\frac{1}{2\pi}\int_0^{2\pi} \hat{\psi}({\mathcal R}_\theta \mathbf{x})\e^{-\i k \theta}\dtheta,$$
where ${\mathcal R}_\theta$ is a rotation through the angle $\theta$.
It was shown by Costabel, Dauge \& Hu \cite[Proposition 2.7(ii)]{CostabelDaugeHu23} that ${\mathcal Q}_k$ is an orthogonal projection on $H^m(B_R(\mathbf{0});{\mathbb C})$ and that the family $\{{\mathcal Q}_k\}_{k \in {\mathbb Z}}$ defines an isometric isomorphism
$H^m(B_R(\mathbf{0});{\mathbb C}) \rightarrow \bigoplus\limits_{k \in {\mathbb Z}} \hat{H}_{(k)}^m(B_R(\mathbf{0});{\mathbb C})$, such that
$$\hat{f} (r\cos\theta, r \sin\theta) = \sum_{k \in {\mathbb Z}} f_k(r) \e^{\i k \theta}, \qquad \|\hat{f}\|_{ H^m}^2 = \sum_{k \in {\mathbb Z}} \|f_k\|_{H^{m}_{(k)}}^2$$
for each $\hat{f} \in H^m(B_R(\mathbf{0});{\mathbb C})$, where $f_k$ is the radial coefficient of ${\mathcal Q}_k[\hat{f}]$.
\end{remark}

In Section \ref{Hankel FS} we introduce a scale of Hilbert spaces $B_{(k)}^s((0,\infty);{\mathbb C})$, $s \geq 0$
for the radial coefficients of mode $k$ functions. These \emph{Hankel spaces} are the radial counterparts of the
Bessel-potential spaces constructed as follows.
The \emph{Fourier transform} is a bijection $\mathscr{S}({\mathbb R}^d;{\mathbb C}) \rightarrow \mathscr{S}({\mathbb R}^d;{\mathbb C})$
defined by
$${\mathcal F}[f](\bfxi) = \frac{1}{(2\pi)^\frac{d}{2}} \int_{\mathbb{R}^d} f(\bfx) \e^{-\i \bfxi\cdot\bfx}$$
with
${\mathcal F}^{-1}[g](\bfx) = {\mathcal F}[g](-\bfx)$; it is isometric with respect to the $L^2({\mathbb R}^d;{\mathbb C})$ norm, and therefore
extends by density to an isometric isomorphism $L^2({\mathbb R}^d;{\mathbb C}) \rightarrow L^2({\mathbb R}^d;{\mathbb C})$ (see Craig \cite[\S5.2]{Craig}). The \emph{Bessel-potential space} $B^s({\mathbb R}^d;{\mathbb C})$ is
defined by
$$B^s({\mathbb R}^d;{\mathbb C}) = \{f \in L^2({\mathbb R}^d;{\mathbb C}): (1+|\bfxi|^2)^\frac{s}{2}{\mathcal F}[f] \in L^2({\mathbb R}^d;{\mathbb C})\}, \qquad s \geq 0;$$
it coincides with the standard Sobolev space $H^s({\mathbb R}^d;{\mathbb C})$ defined by weak derivatives for $s \in {\mathbb N}_0$ and interpolation for $s \geq 0$
(see Adams \cite[Theorem 7.65]{Adams}). Analogously, we define the \emph{$k$-index Hankel transform} of a function $f \in {\mathscr S}_{(k)}([0,\infty);{\mathbb C})$ by
$$\mathcal{H}_{k}[f](\rho) := \int_{0}^{\infty} f(r) J_{k}(\rho \,r) \,r\dr,$$
where $J_k$ is the $k$-index Bessel function of the first kind,
and show that it is a bijection\linebreak ${\mathscr S}_{(k)}([0,\infty);{\mathbb C}) \rightarrow {\mathscr S}_{(k)}([0,\infty);{\mathbb C})$ with
${\mathcal H}_k^{-1}[g]={\mathcal H}_k[g]$ which is isometric
with respect to the $L_1^2((0,\infty);{\mathbb C})$ norm; it therefore
extends by density to an isometric isomorphism
$L_1^2((0,\infty);{\mathbb C}) \rightarrow L_1^2((0,\infty);{\mathbb C})$.
The \emph{Hankel space} $B_{(k)}^s((0,\infty);{\mathbb C})$ is defined by the formula
\begin{equation}
B_{(k)}^s((0,\infty);{\mathbb C}) := \{f \in L_1^2((0,\infty);{\mathbb C})\st(1+\rho^2)^{\frac{s}{2}}{\mathcal H}_k[f] \in L_1^2((0,\infty);{\mathbb C})\}, \qquad s \geq 0; \label{HS definition - intro}
\end{equation}
it coincides with the Sobolev space $H_{(k)}^s((0,\infty);{\mathbb C})$ constructed in Section \ref{Sobolev FS}.

\begin{remarks} $ $
\begin{itemize}
\item[(i)] Just as the Fourier transform works well with standard differential operators, in that
$$
\partial_{\bfxi}^\alpha{\mathcal F}[f](\bfxi)=(-\i)^n{\mathcal F}[\mathbf{x}^\alpha f](\bfxi), \qquad
{\mathcal F}[\partial_{\bfx}^\alpha f](\bfxi) = \i^n \bfxi^\alpha {\mathcal F}[f](\bfxi), 
$$
where $\alpha \in {\mathbb N}_0^d$ for $|\alpha|=n$, $n \in {\mathbb N}_0$, the Hankel transforms work well with the Bessel operators, in that
$$
\tilde{\mathcal D}_{-k+i}^{n-i}\tilde{\mathcal D}_k^i {\mathcal H}_k[f](\rho)  = (-1)^{n-i} {\mathcal H}_{k+n-2i}[r^nf](\rho), \qquad
{\mathcal H}_{k+n-2i} [{\mathcal D}_{-k+i}^{n-i}{\mathcal D}_k^i f](\rho) = (-1)^{n-i}\rho^n {\mathcal H}_k[f](\rho)
$$
for $0 \leq i \leq n$, $n \in {\mathbb N}_0$, where the tilde denotes differentiation with respect to $\rho$ (see Lemma \ref{HT properties}).
\item[(ii)] Just as one can define the Fourier transform of a tempered distribution $T$ in the dual of $\mathscr{S}({\mathbb R}^2;{\mathbb C})$ by the formula $(\mathcal{F}[T],\phi)=(T,\mathcal{F}[\phi])$,
one can define the Hankel transform of a tempered radial distribution $T_k$ in the dual of $\mathscr{S}_{(k)}((0,\infty);{\mathbb C})$ by the formula $(\mathcal{H}_k[T_k],\phi_k)=(T_k,\mathcal{H}_k[\phi_k])$ (see Definitions
\ref{tempered definition} and \ref{Hankel space definition}).
\end{itemize}
\end{remarks}

The $k$-index Hankel transform of $f_k \in {\mathscr S}_{(k)}((0,\infty);{\mathbb C})$ is actually related to the Fourier transform of\linebreak $\hat{f}_k \in \hat{\mathscr S}_{(k)}({\mathbb R}^2;{\mathbb C})$ by the formula
\begin{equation*}
    \mathcal{F}[\hat{f}_k](\rho\cos \omega,\rho\sin\omega) = \mathrm{e}^{\mathrm{i}k(\omega-\frac{\pi}{2})}\mathcal{H}_k[f_k](\rho).
\end{equation*}
Using this observation one can deduce properties of Hankel spaces (embeddings and algebra properties) from the corresponding, known properties of Bessel-potential spaces.
However we find it instructive to work directly with the definition \eqref{HS definition - intro} and give a self-contained theory in Section \ref{Hankel FS}.

\begin{remark}
The more general version of the Hankel transform
$${\mathcal H}_\nu[f](\rho) = \int_0^\infty f(r) J_\nu(\rho r)\sqrt{\rho r}\dr,\qquad \nu \geq - \tfrac{1}{2}$$
was studied by Zemanian \cite{Zemanian66a,Zemanian66b}, who constructed three families ${\mathscr H}_\nu$, ${\mathscr B}_\nu$, ${\mathscr Y}_\nu$ of Frech\'{e}t spaces consisting of smooth `testing' functions $[0,\infty) \rightarrow {\mathbb R}$
which play a similar role to our Schwartz-class spaces ${\mathscr S}_{(k)}((0,\infty); {\mathbb C})$. The Hankel transform is an automorphism on the space ${\mathscr H}_\nu$ of `slow growing' functions and maps the space
${\mathscr B}_\nu$ of `rapidly growing' functions isometrically onto ${\mathscr Y}_\nu$; distributions and distributional Hankel transforms can be constructed in a natural manner, and there are `shift' operators to move between different
spaces. While our work on the Hankel transform has indeed been influenced by the Zemanian spaces, they appear to be often overlooked by the wider analysis community. 

\end{remark}

Our theory is actually more general than indicated in the above summary. We treat functions whose values lie in a
complex Banach space $X$ and---in Sections \ref{Classical FS}--\ref{Sobolev FS} (classical and Sobolev spaces, distributions)---which may depend upon an additional axial coordinate $z$.
In other words, we consider
functions $\hat{f}_k: B_R({\mathbf 0}) \times {\mathbb R}\rightarrow X$ with the property that
$$
\hat{f}_k(r\cos\theta,r\sin\theta,z)=\e^{\i k \theta}f_k(r,z), \qquad r \in [0,R),\ \theta \in {\mathbb T}^1,\ z \in {\mathbb R},
$$
for some $k \in {\mathbb Z}$ and some $f_k:[0,R) \times {\mathbb R}\rightarrow X$ with $f_k(0,z)=0$ for $k \neq 0$; the definitions of the various function spaces are generalised in the obvious fashion.
Furthermore, just as one can define a partial Fourier transform ${\mathcal F}: {\mathscr S}_{(k)}(B_R(\mathbf{0}) \times {\mathbb R};X) \rightarrow  {\mathscr S}_{(k)}(B_R(\mathbf{0}) \times {\mathbb R};X)$ by the formula
$${\mathcal F}[f_k](r,\xi) = \frac{1}{\sqrt{2\pi}} \int_{\mathbb{R}} f_k(r,z) \e^{-\i \xi z}\dz,$$
one can define a partial Hankel transform ${\mathcal H}_k: {\mathscr S}_{(k)}((0,\infty) \times {\mathbb R};X)\rightarrow {\mathscr S}_{(k)}((0,\infty) \times {\mathbb R};X)$ by the formula
$$\mathcal{H}_{k}[f_k](\rho,z) := \int_{0}^{\infty} f_k(r,z) J_{k}(\rho \,r) \,r\dr.$$

In Section \ref{DVP} we present a typical application of our theory by considering the linear boundary-value problem
\begin{align*}
       \Delta \hat{u}_k ={}& \hat{f}_k , \qquad 0<|(x,y)|<1,\\
       \hat{u}_k ={}& \hat{g}_k , \qquad |(x,y)|=1,
\end{align*}
for the mode $k$ function $\hat{u}_k=\e^{\i k \theta}u_k(r,z)$ in the cylindrical domain $B_1({\mathbf 0}) \times {\mathbb R}$,
where $\hat{f}_k=\e^{\i k \theta}f_k(r,z)$ and $\hat{g}_k=\e^{\i k \theta}g_k(z)$ are given mode $k$ functions.
More precisely, we consider the equivalent boundary-value problem
\begin{align}
        (\mathcal{D}_{1-k}\mathcal{D}_{k} + \partial_z^2 )u_k ={}& f_k , \qquad 0<r<1, \label{bvp:rzn1 - intro} \\
        u_k ={}& g_k , \qquad r=1, \label{bvp:rzn2 - intro}
\end{align}
for $u_k$ in the domain $(0,1) \times {\mathbb R}$.

First suppose that $f_k \in {\mathscr S}_{(k)}((0,1) \times {\mathbb R};\mathbb{C})$ and $g_k \in {\mathscr S}({\mathbb R};\mathbb{C})$.
Applying the partial Fourier transform  to \eqref{bvp:rzn1 - intro}, \eqref{bvp:rzn2 - intro} yields the two-point boundary-value problem
\begin{align*}
        (\mathcal{D}_{1-k}\mathcal{D}_k - |\xi|^2 )\check{u}_k ={}& \check{f}_k, \qquad 0<r<1, \\
        \check{u}_k ={}& \check{g}_k , \qquad r=1,
\end{align*}
where $\check{u}_k$, $\check{f}_k$ and $\check{g}_k$ denote the partial Fourier transforms of $u_k$, $f_k$ and $g_k$ respectively. The solution to
\eqref{bvp:rzn1 - intro}, \eqref{bvp:rzn2 - intro} is therefore
\begin{equation}
u_k=\mathcal{G}_k(f_k)  + \mathcal{B}_k(g_k), \label{eq:integral solution - intro}
\end{equation}
in which
$$
    \mathcal{G}_k(f_k) := \mathcal{F}^{-1}\left[\int_0^1 G_k(r,\rho)\,\check{f}_k(\rho)\,\rho\,\mathrm{d}\rho\right],\qquad \mathcal{B}_k(g_k) := \mathcal{F}^{-1}\left[\tilde{\mathcal D}_{k}G_k(r,1)\,\check{g}_k\right]
$$
(and we again denote Bessel operators with respect to $r$ and $\rho$ by respectively ${\mathcal D}_k^i$ and $\tilde{\mathcal D}_k^i$);
the Green's function $G_k(r,\rho)$ is given by
\begin{equation*}
    G_k(r,\rho) = \begin{cases}
        \displaystyle-I_{k}(|\xi|r) \left( K_{k}(|\xi|\rho) - \frac{K_{k}(|\xi|)}{I_{k}(|\xi|)}I_{k}(|\xi|\rho)\right), & 0<r<\rho,\\[2mm]
        \displaystyle-I_{k}(|\xi|\rho)\left(K_{k}(|\xi|r) - \frac{K_{k}(|\xi|)}{I_{k}(|\xi|)}I_{k}(|\xi|r)\right), & \rho<r<1,
    \end{cases}
\end{equation*}
where $I_k$, $K_k$ are modified Bessel functions of the first and second kind respectively. In Section \ref{DVP} we establish the following result for the operators ${\mathcal G}_k$
and ${\mathcal B}_k$; the main ingredient in the proof is a set of integral identities for modified Bessel functions (see equations \eqref{Block1}--\eqref{Block4K}).

\begin{lemma}  \label{props of Gk, Bk}
Suppose that $m \in {\mathbb N}_0$.
\begin{itemize}
\item[(i)]
The function ${\mathcal G}_k(f_k)$ belongs to $H_{(k)}^{m+2}((0,1) \times {\mathbb R};{\mathbb C})$ with
$$\|{\mathcal G}_k(f_k)\|_{H^{m+2}_{(k)}}\lesssim\|f_{k}\|_{H^{m}_{(k)}},$$
such that ${\mathcal G}_k$ extends by density to a bounded linear operator $H_{(k)}^m((0,1) \times {\mathbb R};{\mathbb C}) \rightarrow H_{(k)}^{m+2}((0,1) \times {\mathbb R};{\mathbb C})$.
\item[(ii)]
The function ${\mathcal B}_k(g_k)$ belongs to $H_{(k)}^{m+2}((0,1) \times {\mathbb R};{\mathbb C})$ with
$$\|{\mathcal B}_k(g_k)\|_{H^{m+1}_{(k)}} \lesssim \|g_k\|_{H^{m+\frac{1}{2}}}, \qquad \|{\mathcal B}_k(g_k)\|_{H^{m+2}_{(k)}} \lesssim \|g_k\|_{H^{m+\frac{3}{2}}}$$
such that ${\mathcal B}_k$ extends by density to a bounded linear operator $H^{m+\frac{1}{2}}({\mathbb R};{\mathbb C}) \rightarrow H_{(k)}^{m+1}((0,1) \times {\mathbb R};{\mathbb C})$ and
$H^{m+\frac{3}{2}}({\mathbb R};{\mathbb C}) \rightarrow H_{(k)}^{m+2}((0,1) \times {\mathbb R};{\mathbb C})$.
\end{itemize}
\end{lemma}

Altogether we have established the following result.
\begin{theorem}
Suppose that $m \in {\mathbb N}_0$.
For each $f_k \in H^{m}_{(k)}((0,1) \times {\mathbb R};{\mathbb C})$ and $g_k \in H^{m+\frac{3}{2}}({\mathbb R};{\mathbb C})$ the boundary-value problem \eqref{bvp:rzn1 - intro}, \eqref{bvp:rzn2 - intro} has a unique solution
$u_k \in H_{(k)}^{m+2}((0,1) \times {\mathbb R};{\mathbb C})$ which depends continuously upon $f_k$ and $g_k$.
\end{theorem}

We also develop a theory of weak solutions for \eqref{bvp:rzn1 - intro}, \eqref{bvp:rzn2 - intro}.

\begin{definition}
Suppose that $f_k \in L_1^2((0,1) \times {\mathbb R};{\mathbb C})$ and
$g_k \in H^\frac{1}{2}({\mathbb R}; {\mathbb C})$.
A \underline{weak solution} of \eqref{bvp:rzn1 - intro}, \eqref{bvp:rzn2 - intro}
is a function $u_k \in H^1_{(k)}((0,1) \times {\mathbb R};{\mathbb C})$ which satisfies
\eqref{bvp:rzn2 - intro} in $H^\frac{1}{2}({\mathbb R}; {\mathbb C})$ and
$$
\tfrac{1}{2}\langle {\mathcal D}_k{u}_k, {\mathcal D}_k{\phi}_k \rangle_{L_1^2} + \tfrac{1}{2}\langle{\mathcal D}_{-k} {u}_k, {\mathcal D}_{-k}  {\phi}_k \rangle_{L_1^2}
+ \langle \partial_z{u}_k, \partial_z {\phi}_k \rangle_{L_1^2} + \langle {f}_k, {\phi}_k \rangle_{L_1^2}=0
$$
for all $\phi_k \in H^1_{(k)}((0,1) \times {\mathbb R};{\mathbb C})$ with $\phi_k|_{r=1}=0$ in $H^\frac{1}{2}({\mathbb R}; {\mathbb C})$.
\end{definition}

We show that \eqref{bvp:rzn1 - intro}, \eqref{bvp:rzn2 - intro} has a unique weak solution which is given by \eqref{eq:integral solution - intro}; in view of Lemma \ref{props of Gk, Bk} (with $m=0$) it depends
continuously on $f_k \in L_1^2((0,1) \times {\mathbb R};{\mathbb C})$ and $g_k \in H^\frac{1}{2}({\mathbb R}; {\mathbb C})$.

We note that natural continuations of our work---such as treatment of vector-valued functions, convergence of Fourier series of mode $k$ functions and a theory for Nemytskii operators in radial Sobolev and Hankel spaces---are deferred to later studies in order to keep the present article concise.

\subsection*{Acknowledgements}

This material is based upon work supported by the Swedish Research Council under grant no.\ 2021-06594 while the authors were in residence at Institut Mittag-Leffler in Djursholm, Sweden during October 2023.
DJH gratefully acknowledges support from the Alexander von Humboldt Foundation.

\section{Radial differential operators}\label{s:DiffOp}

In this section we introduce, and discuss the basic properties of, our radial differential operators. Let $X$ be a complex Banach space and
$\hat{f}_k: {\mathbb R}^2 \times {\mathbb R}\rightarrow X$ be a function with the property that
\begin{equation}
\hat{f}_k(r\cos\theta,r\sin\theta,z)=\e^{\i k \theta}f_k(r,z), \qquad r \in [0,\infty),\ \theta \in {\mathbb T}^1,\ z \in {\mathbb R}, \label{eq:mode k defn}
\end{equation}
for some $k \in {\mathbb Z}$ and some $f_k:[0,\infty) \times {\mathbb R}\rightarrow X$ with $f_k(0,z)=0$ for $k \neq 0$.
We refer to such functions as \emph{mode $k$ functions}.

\begin{remarks} \label{rem:radialcoeffs}
$ $
\begin{itemize}
\item[(i)]
The \underline{radial coefficient} $f_0: [0,\infty) \times {\mathbb R} \rightarrow X$ of a mode $0$ function $\hat{f}_0$ obviously satisfies
$f_0(0,z)=\hat{f}_0(\mathbf{0},z)$.
The same is true for $k \neq 0$ since $f_k(0,z)=0$ implies that $\hat{f}_k(\mathbf{0},z)=0$.
\item[(ii)]
The radial coefficient $f_k: [0,\infty) \times {\mathbb R} \rightarrow X$ of the mode $k$ function $\hat{f}_k(x,y,z)$ is the radial coefficient of the
mode $-k$ function $\hat{f}_k(x,-y,z)$.
\end{itemize}
\end{remarks}

Assuming sufficient differentiability, let us consider the effect of applying Wirtinger-type complex differential operators $\partial_{\overline{\zeta}}^{n-i}\partial_{\zeta}^{i}$ with
$$
\partial_{\zeta} := \tfrac{1}{\sqrt{2}}(\partial_{x} - \mathrm{i}\partial_{y}), \qquad \partial_{\overline{\zeta}} := \tfrac{1}{\sqrt{2}}(\partial_{x} + \mathrm{i}\partial_{y})
$$
to mode $k$ functions. For this purpose we introduce the \emph{$k$-index Bessel operator}
$$
    \mathcal{D}_{k} := r^{-k} \frac{\mathrm{d}}{\mathrm{d}r} r^{k} = \frac{\mathrm{d}}{\mathrm{d}r} + \frac{k}{r},
$$
and
$$
    \mathcal{D}^{n}_{k} := r^{-k+n} \left(\frac{1}{r}\frac{\mathrm{d}}{\mathrm{d}r}\right)^{n} r^{k} = \mathcal{D}_{k-(n-1)}\mathcal{D}_{k-(n-2)}\dots\mathcal{D}_{k-1}\mathcal{D}_{k},
$$
for $k\in\mathbb{Z}$ and $n\in\mathbb{N}_0$. The following proposition shows in particular that
the generic radial differential operator is the generalised Bessel operator $\mathcal{D}_{-k+i}^{n-i}\mathcal{D}_{k}^{i}$.

\begin{proposition}\label{Prop:Dk}
Let $\hat{f}_{k}:\mathbb{R}^{2} \times {\mathbb R} \to X$ be a mode $k$ function with radial coefficient
$f_k:[0,\infty) \times {\mathbb R} \rightarrow X$.
It follows that
$$
       \partial_{\zeta}\hat{f}_{k}= \mathrm{e}^{\mathrm{i}(k-1)\theta}\tfrac{1}{\sqrt{2}}\mathcal{D}_{k}f_k, \qquad \partial_{\overline{\zeta}}\hat{f}_{k} = \mathrm{e}^{\mathrm{i}(k+1)\theta}\tfrac{1}{\sqrt{2}}\mathcal{D}_{-k}f_k,
$$
 and more generally
$$
        \partial_{\overline{\zeta}}^{n-i}\partial^{i}_{\zeta}\hat{f}_{k} = \mathrm{e}^{\mathrm{i}(k+n-2i)\theta}\,2^{-\frac{n}{2}}\mathcal{D}_{-k+i}^{n-i}\mathcal{D}^{i}_{k}f_k
$$
for $n \in {\mathbb N}_0$ and $0 \leq i \leq n$.
\end{proposition}

\begin{proof}
Observe that in polar coordinates
$$
  \partial_{\zeta} = \mathrm{e}^{-\mathrm{i}\theta}\tfrac{1}{\sqrt{2}}(\partial_{r} - \tfrac{\mathrm{i}}{r}\partial_{\theta}), \qquad \partial_{\overline{\zeta}} = \mathrm{e}^{\mathrm{i}\theta}\tfrac{1}{\sqrt{2}}(\partial_{r} + \tfrac{\mathrm{i}}{r}\partial_{\theta}),
$$
such that
\begin{equation*}
    \begin{split}
        \partial_{\zeta}\hat{f}_{k} ={}& \mathrm{e}^{-\mathrm{i}\theta}\tfrac{1}{\sqrt{2}}(\partial_{r} - \tfrac{\mathrm{i}}{r}\partial_{\theta})\big(\mathrm{e}^{\mathrm{i}k\theta} f_k\big) = \mathrm{e}^{\mathrm{i}(k-1)\theta}\tfrac{1}{\sqrt{2}}(\tfrac{\mathrm{d}}{\mathrm{d}r} + \tfrac{k}{r}) f_k = \mathrm{e}^{\mathrm{i}(k-1)\theta}\tfrac{1}{\sqrt{2}}\mathcal{D}_{k}f_k,\\
        \partial_{\overline{\zeta}}\hat{f}_{k} ={}& \mathrm{e}^{\mathrm{i}\theta}\tfrac{1}{\sqrt{2}}(\partial_{r} + \tfrac{\mathrm{i}}{r}\partial_{\theta})\big(\mathrm{e}^{\mathrm{i}k\theta} f_k\big) = \mathrm{e}^{\mathrm{i}(k+1)\theta}\tfrac{1}{\sqrt{2}}(\tfrac{\mathrm{d}}{\mathrm{d}r} - \tfrac{k}{r}) f_k = \mathrm{e}^{\mathrm{i}(k+1)\theta}\tfrac{1}{\sqrt{2}}\mathcal{D}_{-k}f_k.\\
    \end{split}
\end{equation*}
Furthermore
\begin{equation*}
    \begin{split}
        \partial_{\zeta}^{n}\big(\mathrm{e}^{\mathrm{i}k\theta}f_k\big) ={}& \partial_{\zeta}^{n-1}\big(\mathrm{e}^{\mathrm{i}(k-1)\theta}\tfrac{1}{\sqrt{2}}\mathcal{D}_{k}f_k\big) = \partial_{\zeta}^{n-2}\big(\mathrm{e}^{\mathrm{i}(k-2)\theta}\tfrac{1}{2}\mathcal{D}^{2}_{k}f_k\big) = \dots = \mathrm{e}^{\mathrm{i}(k-n)\theta}2^{-\frac{n}{2}}\mathcal{D}^{n}_{k}f_k,\\
        \partial^{n}_{\overline{\zeta}}\big(\mathrm{e}^{\mathrm{i}k\theta}f_k\big) ={}& \partial^{n-1}_{\overline{\zeta}}\big(\mathrm{e}^{\mathrm{i}(k+1)\theta}\tfrac{1}{\sqrt{2}}\mathcal{D}_{-k}f_k\big) = \partial^{n-2}_{\overline{\zeta}}\big(\mathrm{e}^{\mathrm{i}(k+2)\theta}\tfrac{1}{2}\mathcal{D}^{2}_{-k}f_k\big) = \dots = \mathrm{e}^{\mathrm{i}(k+n)\theta}2^{-\frac{n}{2}}\mathcal{D}^{n}_{-k}f_k
    \end{split}
\end{equation*}
for $n \in {\mathbb N}_0$,
from which it follows that
$$
        \partial_{\overline{\zeta}}^{n-i}\partial^{i}_{\zeta}\big(\mathrm{e}^{\mathrm{i}k\theta}f\big) = \partial_{\overline{\zeta}}^{n-i}\big(\mathrm{e}^{\mathrm{i}(k-i)\theta}2^{-\frac{i}{2}}\mathcal{D}^{i}_{k}f\big) = \mathrm{e}^{\mathrm{i}(k+n-2i)\theta}2^{-\frac{n}{2}}\mathcal{D}_{-k+i}^{n-i}\mathcal{D}^{i}_{k}f.
$$
for $n \in {\mathbb N}_0$ and $0 \leq i \leq n$.
\end{proof}

\begin{remarks} \label{rem:Lap}$ $
\begin{itemize}
\item[(i)]
Notice that
$$\Delta \hat{f}_k = (2\,\partial_\zeta\partial_{\bar{\zeta}} + \partial_z^2)\hat{f}_k = \e^{\i k \theta}\Delta_k f_k,$$
where the \underline{$k$-index radial Laplacian} is given by
$$\Delta_k:=\mathcal{D}_{1-k}\mathcal{D}_k + \partial_z^2= \mathcal{D}_{1+k}\mathcal{D}_{-k}+\partial_z^2.$$
\item[(ii)]
Let us also record the useful formula
$$\mathcal{D}_{-k+i}^{n-i}\mathcal{D}^{i}_{k}J_{k}(r) = r^{k+n-2i}\left(\frac{1}{r}\frac{\mathrm{d}}{\mathrm{d}r}\right)^{n-i}r^{-2(k-i)}\left(\frac{1}{r}\frac{\mathrm{d}}{\mathrm{d}r}\right)^{i}r^{k} J_{k}(r) = (-1)^{n-i}J_{k+n-2i}(r),$$
where $J_k$ is the $k$-index Bessel function of the first kind (see Watson \cite{Watson} and Abramowitz \& Stegun \cite{AbramowitzStegun} for a comprehensive treatment of Bessel functions).
\end{itemize}
\end{remarks}

Although the Wirtinger operators $\partial_{\zeta},\partial_{\overline{\zeta}}$ commute, the
Bessel operators (whose definition naturally extends to real-valued indices) instead satisfy the following commutation relation.

\begin{proposition}\label{Prop:Comm}
    The commutation relation
$$
    \mathcal{D}_{\nu}^{n}\mathcal{D}_{\mu}^{m} = \mathcal{D}_{\mu+n}^{m}\mathcal{D}_{\nu-m}^{n}
$$
 holds for all $\nu,\mu\in\mathbb{R}$ and $m,n\in\mathbb{N}_{0}$.
\end{proposition}

\begin{proof}
We begin by considering the case when $n=m=1$, such that
$$
    \mathcal{D}_{\nu}\mathcal{D}_{\mu} = \left(\frac{\mathrm{d}}{\mathrm{d}r} + \frac{\nu}{r}\right)\!\!\left(\frac{\mathrm{d}}{\mathrm{d}r} + \frac{\mu}{r}\right) = \frac{\mathrm{d}^2}{\mathrm{d}r^2} + \frac{\nu+\mu}{r}\frac{\mathrm{d}}{\mathrm{d}r} + \frac{\mu(\nu-1)}{r^2}.
$$
It follows that
$\mathcal{D}_{\nu}\mathcal{D}_{\mu} = \mathcal{D}_{\alpha}\mathcal{D}_{\beta}$ for any $\alpha,\beta\in\mathbb{R}$ that satisfy
$$
    \alpha + \beta = \nu + \mu, \qquad \qquad \beta(\alpha-1) = \mu(\nu-1),
$$
and we conclude that
$$
    \mathcal{D}_{\nu}\mathcal{D}_{\mu} = \mathcal{D}_{\mu+1}\mathcal{D}_{\nu-1}
$$
for all $\nu,\mu\in\mathbb{R}$.

We now return to the general case
$$
    \mathcal{D}_{\nu}^{n}\mathcal{D}_{\mu}^{m} = \left(\mathcal{D}_{\nu-(n-1)}\mathcal{D}_{\nu-(n-2)}\dots\mathcal{D}_{\nu-1}\mathcal{D}_{\nu}\right)\left(\mathcal{D}_{\mu-(m-1)}\mathcal{D}_{\mu-(m-2)}\dots\mathcal{D}_{\mu-1}\mathcal{D}_{\mu}\right).
$$
Observe that
\begin{equation*}\begin{split}
    \mathcal{D}_{\nu}^{n}\mathcal{D}_{\mu}^{m} ={}& \mathcal{D}^{n-1}_{\nu-1}\left(\mathcal{D}_{\nu}\mathcal{D}_{\mu-(m-1)}\mathcal{D}_{\mu-(m-2)}\dots\mathcal{D}_{\mu-1}\mathcal{D}_{\mu}\right),\\
    ={}& \mathcal{D}^{n-1}_{\nu-1}\left(\mathcal{D}_{\mu+1-(m-1)}\mathcal{D}_{\nu-1}\mathcal{D}_{\mu-(m-2)}\dots\mathcal{D}_{\mu-1}\mathcal{D}_{\mu}\right),\\
    ={}& \mathcal{D}^{n-1}_{\nu-1}\left(\mathcal{D}_{\mu+1-(m-1)}\mathcal{D}_{\mu+1-(m-2)}\mathcal{D}_{\nu-2}\dots\mathcal{D}_{\mu-1}\mathcal{D}_{\mu}\right),\\
    \vdots{}& \\
    ={}& \mathcal{D}^{n-1}_{\nu-1}\left(\mathcal{D}_{\mu+1-(m-1)}\mathcal{D}_{\mu+1-(m-2)}\dots\mathcal{D}_{\mu+1-1}\mathcal{D}_{\mu+1}\right)\mathcal{D}_{\nu-m},\\
    ={}& \mathcal{D}^{n-1}_{\nu-1}\mathcal{D}_{\mu+1}^{m}\mathcal{D}_{\nu-m}.\\
\end{split}\end{equation*}
Continuing this iteration, we obtain
$$
    \mathcal{D}_{\nu}^{n}\mathcal{D}_{\mu}^{m} = \mathcal{D}^{n-1}_{\nu-1}\mathcal{D}_{\mu+1}^{m}\mathcal{D}_{\nu-m} = \mathcal{D}^{n-2}_{\nu-2}\mathcal{D}_{\mu+2}^{m}\mathcal{D}^{2}_{\nu-m}=\dots = \mathcal{D}_{\mu+n}^{m}\mathcal{D}^{n}_{\nu-m}.
\eqno{\qedhere}
$$
\end{proof}

According to Proposition \ref{Prop:Dk} the operator
$\partial_\zeta$ maps a mode $k$ function with radial coefficient $f_k$
to a mode $k-1$ function with radial coefficient ${\mathcal D}_k f_k$ while $\partial_{\bar\zeta}$ maps a mode $k+1$ function with radial coefficient ${\mathcal D}_{-k} f_k$.
Correspondingly, one finds that ${\mathcal D}_k$ and ${\mathcal D}_{-k}$ map a mode $k$ radial coefficient to a mode $k-1$ and a mode $k+1$ radial coefficient respectively,
as illustrated diagrammatically in Figure \ref{fig:Fkspace}. (The apparent ambiguity in this interpretation of ${\mathcal D}_0$ is resolved by Remark \ref{rem:radialcoeffs}(ii).)
It follows from Proposition~\ref{Prop:Comm} that the diagram commutes,
and we note the particularly helpful case
$$
    \mathcal{D}_{-k+n}^{m}\mathcal{D}_{k}^{n} = \mathcal{D}_{k+m}^{n}\mathcal{D}_{-k}^{m}
$$
for $k \in {\mathbb Z}$ and $m,n \in {\mathbb N}_0$ for later use.

The following proposition lists some elementary properties of the Bessel operators (with real-valued indices), under the assumption that all products, compositions and integrals exist.

\begin{proposition}\label{Prop:DkProperties}
  The Bessel operator $\mathcal{D}_{\mu}$ has the following properties for each $\mu\in\mathbb{R}$.
\begin{enumerate}[label=(\alph*)]
    \item Constant rule:
    \begin{equation*}
        \mathcal{D}_{-\mu}(r^{\mu}) = 0.
    \end{equation*}
    \item Shift rule:
    \begin{equation*}
        \mathcal{D}_{\mu}(r^{\nu}f(r)) = r^{\nu}\mathcal{D}_{\mu+\nu}f(r)
    \end{equation*}
     for any $\nu\in\mathbb{R}$.
    \item Product rule:    \begin{equation*}
        \mathcal{D}_{\mu+\nu}(f g) = \mathcal{D}_{\mu}f g + f\mathcal{D}_{\nu}g
    \end{equation*}
 for any $\nu\in\mathbb{R}$.
  \item Leibniz integral rule:
    \begin{equation*}
        \mathcal{D}_{\mu}\left(\int_{a(r)}^{b(r)} f(r,t)\,\mathrm{d}t\right) = f(r,b(r)) b^\prime(r)- f(r,a(r)) a^\prime(r) + \int_{a(r)}^{b(r)} \mathcal{D}_{\mu}f(r,t)\,\mathrm{d}t.
    \end{equation*}
\end{enumerate}
\end{proposition}

\begin{remark} \label{rem:product}
The Leibniz product rule
\begin{equation*}
    \begin{split}
        \mathcal{D}_{-k-q+i}^{n-i}\mathcal{D}_{k+q}^{i}\partial_z^{p-n}(f_kg_q) &=  \sum_{j=0}^{i}\sum_{\ell=j}^{n-(i-j)}\sum_{s=0}^{p-n}\begin{pmatrix}
              n-i \\ \ell-j
          \end{pmatrix}\!\begin{pmatrix}
              i \\ j
          \end{pmatrix}\! \begin{pmatrix} p-n \\ s \end{pmatrix} \left(\mathcal{D}_{-k+i-j}^{n-(i-j)-\ell}\mathcal{D}_{k}^{i-j}\partial_z^{p-n-s} f_k\right) \!\!\left(\mathcal{D}_{-q+j}^{\ell-j}\mathcal{D}_{q}^{j}\partial_z^sg_q\right)
    \end{split}
\end{equation*}
for mode $k$ and $q$ radial coefficients $f_k$ and $q_k$ is readily obtained by induction over $n$ and $p$.
\end{remark}

Finally, we give a useful radial version of Fa\'{a} di Bruno's formula (Lemma \ref{Radial FDB} below), again assuming that all functions are sufficiently well behaved to justify any formal calculations. The following proposition is used in the proof of this lemma.

\begin{proposition} \label{prop:proj}
Suppose that $\hat{\psi}:\mathbb{R}^{2} \times {\mathbb R} \to X$ and define a mode $k$ radial coefficient ${\mathcal P}_k[\hat{\psi}]$ by
$${\mathcal P}_k[\hat{\psi}](r,z):=\frac{1}{2\pi}\int_0^{2\pi} \hat{\psi}(r\cos\theta,r\sin\theta,z)\e^{-\i k \theta}\dtheta.$$
The formula
\begin{equation*}
    \begin{split}
        \mathcal{D}^{n-i}_{-k+i}\mathcal{D}^i_{k}\partial_z^{p-n}\mathcal{P}_{k}[\hat{\psi}] ={}& \mathcal{P}_{k+n-2i}[2^{\frac{n}{2}}\partial^{n-i}_{\overline{\zeta}}\partial^i_\zeta \partial_z^{p-n}\hat{\psi}]
    \end{split}
\end{equation*} 
holds for $0 \leq i \leq n \leq p$, $p \in {\mathbb N}_0$.
\end{proposition}
\begin{proof}
Observe that
\begin{equation*}
    \begin{split}
        \mathcal{D}_{k}\mathcal{P}_{k}[\hat{\psi}](r,z) ={}& \frac{1}{2\pi}\int_0^{2\pi} \left(\partial_r + \frac{k}{r}\right)\hat{\psi}(r\cos\theta, r\sin\theta,z)\,\e^{-\i k\theta}\dtheta\\
        ={}& \frac{1}{2\pi}\int_0^{2\pi} \partial_r(\hat{\psi}(r\cos\theta, r\sin\theta,z))\,\e^{-\i k\theta}\dtheta + \frac{1}{2\pi}\int_0^{2\pi} \frac{k}{r}\hat{\psi}(r\cos\theta, r\sin\theta)\,\e^{-\i k\theta}\dtheta \\
={}& \frac{1}{2\pi}\int_0^{2\pi} \partial_r(\hat{\psi}(r\cos\theta, r\sin\theta,z))\,\e^{-\i k\theta}\dtheta - \frac{1}{2\pi}\int_0^{2\pi} \frac{\i}{r}\partial_\theta(\hat{\psi}(r\cos\theta, r\sin\theta,z))\,\e^{-\i k\theta}\dtheta,\\
        ={}& \frac{1}{2\pi}\int_0^{2\pi} \e^{-\i\theta}\left(\partial_r - \frac{\i}{r}\partial_{\theta}\right)\hat{\psi}(r\cos\theta, r\sin\theta,z)\,\e^{-\i(k-1)\theta}\dtheta\\
  ={}& \frac{1}{2\pi}\int_0^{2\pi}\sqrt{2}\,\partial_\zeta \hat{\psi}(r\cos\theta,r\sin\theta,z)\,\e^{-\i(k-1)\theta}\dtheta \\
  ={}& \mathcal{P}_{k-1}[\sqrt{2}\,\partial_\zeta \hat{\psi}](r,z),
    \end{split}
    \end{equation*}
where we have used the relation $\partial_{\zeta} = \mathrm{e}^{-\mathrm{i}\theta}\tfrac{1}{\sqrt{2}}(\partial_{r} - \tfrac{\mathrm{i}}{r}\partial_{\theta})$, and similarly
$$\mathcal{D}_{-k}\mathcal{P}_{k}[\hat{\psi}] =\mathcal{P}_{k+1}[\sqrt{2}\,\partial_{\overline{\zeta}} \hat{\psi}], \qquad \partial_z \mathcal{P}_{k}[\hat{\psi}]= \mathcal{P}_{k}[\partial_z\hat{\psi}].$$
Iterating these results yields
$$
      \mathcal{D}^n_{k}\mathcal{P}_{k}[\hat{\psi}] = \mathcal{P}_{k-n}[2^{\frac{n}{2}}\partial^n_\zeta \hat{\psi}], \qquad  \mathcal{D}^n_{-k}\mathcal{P}_{k}[\hat{\psi}] = \mathcal{P}_{k+n}[2^{\frac{n}{2}}\partial^n_{\overline{\zeta}} \hat{\psi}],
        \qquad \partial_z^p \mathcal{P}_{k}[\hat{\psi}]= \mathcal{P}_{k}[\partial_z^p\hat{\psi}]
$$
and hence
\begin{equation*}
    \begin{split}
        \mathcal{D}^{n-i}_{-k+i}\mathcal{D}^i_{k}\partial_z^{p-n}\mathcal{P}_{k}[\hat{\psi}] ={}& \mathcal{P}_{k+n-2i}[2^{\frac{n}{2}}\partial^{n-i}_{\overline{\zeta}}\partial^i_\zeta \partial_z^{p-n}\hat{\psi}].\qedhere
    \end{split}
\end{equation*} 
\end{proof}

\begin{remark} \label{propprojrig}
Suppose that $\hat{\phi} \in C^m({\mathbb R}^2 \times {\mathbb R};X)$. It follows from the uniform continuity of the defining integrand and its derivatives over compact subsets of $[0,\infty) \times {\mathbb T}^1 \times {\mathbb R}$ that
$P_k[\hat{\phi}] \in C^0([0,\infty) \times {\mathbb R};X)$, and furthermore that differentiation under the integral sign is justified, so that $P_k[\hat{\phi}] \in C^m([0,\infty) \times {\mathbb R};X)$.
\end{remark}

\begin{lemma} \label{Radial FDB}
Suppose that $u: X \rightarrow X$, let $\hat{f}_{k}:\mathbb{R}^{2} \times {\mathbb R} \to X$ be a mode $k$ function with radial coefficient
$f_k:[0,\infty) \times {\mathbb R} \rightarrow X$ and define a mode $\ell$ radial coefficient $u_\ell(f_k): [0,\infty) \times {\mathbb R} \rightarrow X$ by
$$
    u_{\ell}(f_{k}) := \mathcal{P}_{\ell}[u \circ \hat{f}_{k}].
$$
The formula
$$
    \mathcal{D}_{-\ell+i}^{n-i}\mathcal{D}_{\ell}^{i}\partial_z^{p-n}u_{\ell}(f_{k})=\sum_{\pi=\{B_j\}\in\Pi_{p}} \hspace{-2mm}{\mathcal P}_{\ell-|\pi|k}\big[\mathrm{d}^{|\pi|}f[\hat{u}_{k}]\big({\mathcal D}^{B_1}_{k}f_k, \cdots, {\mathcal D}^{B_{|\pi|}}_{k}f_k\big)\big]
    $$
holds for $0 \leq i \leq n \leq p$, $p \in {\mathbb N}_0$. Here $\Pi_p$ is the set of partitions of $\{1,\ldots,p\}$, each partition $\pi  \in \Pi_p$ consists of $|\pi|$ blocks $B_1,\ldots,B_{|\pi|}$, and
$${\mathcal D}^{B}_{k} f_k={\mathcal D}_{-k+m_{B}^i}^{m_{B}^n-m_{B}^i}\mathcal{D}_{k}^{m_{B}^i} \partial_z^{|B|-m_B^n} f_{k},$$
in which $m_B^j$ is the number of elements in $B$ which are less than or equal to $j$.
\end{lemma}
\begin{proof}
Using the multivariate version of Fa\'{a} di Bruno's formula given by Hardy \cite[Proposition 1]{Hardy06} in the present context, we find that
\begin{equation*}
    \partial_{\overline{\zeta}}^{n-i}\partial_{\zeta}^{i} \partial_z^{p-n}(u \circ \hat{f}_{k}) = \hspace{-2mm}\sum_{\pi=\{B_j\}\in\Pi_{p}} \hspace{-2mm}\mathrm{d}^{|\pi|}u[\hat{f}_{k}]\left(\partial^{B_1}\hat{f}_k, \cdots, \partial^{B_{|\pi|}}\hat{f}_k\right),
\end{equation*}
where
$$\partial^{B}\hat{f}_k = \partial_{\overline{\zeta}}^{m_B^n-m_B^i}\partial_{\zeta}^{m_B^i} \partial_z^{|B|-m_B^n}\hat{f}_{k}.$$
It follows from Proposition \ref{prop:proj} that
\begin{equation*}
\begin{split}
    \mathcal{D}_{-\ell+i}^{n-i}\mathcal{D}_{\ell}^{i} &\partial_z^{p-n} P_{\ell}[u \circ \hat{f}_k] \\
    ={}& \frac{1}{2\pi}\int_{0}^{2\pi} 2^\frac{n}{2}\partial_{\overline{\zeta}}^{n-i}\partial_{\zeta}^{i}\partial_z^{p-n}(u \circ \hat{f}_k)\mathrm{e}^{-\mathrm{i}(\ell+n-2i)\theta}\,\mathrm{d}\theta\\
    ={}& \frac{1}{2\pi}\int_{0}^{2\pi} 2^\frac{n}{2}\hspace{-3mm}\sum_{\pi=\{B_j\}\in\Pi_{p}} \hspace{-2mm}\mathrm{d}^{|\pi|}u[\hat{f}_k]\left(\partial^{B_1}\hat{f}_k, \cdots, \partial^{B_{|\pi|}}\hat{f}_k\right)\mathrm{e}^{-\mathrm{i}(\ell+n-2i)\theta}\,\mathrm{d}\theta\\
    ={}& \frac{1}{2\pi}\int_{0}^{2\pi} \hspace{-2mm}\sum_{\pi=\{B_j\}\in\Pi_{p}} \hspace{-2mm}\mathrm{e}^{\mathrm{i}(|\pi|k + n - 2i)\theta}\mathrm{d}^{|\pi|}u[\hat{f}_k]\big({\mathcal D}^{B_1}_{k}f_k, \cdots, {\mathcal D}^{B_{|\pi|}}_{k}f_k\big)\mathrm{e}^{-\mathrm{i}(\ell+n-2i)\theta}\,\mathrm{d}\theta, \\
    ={}& \frac{1}{2\pi}\int_{0}^{2\pi} \hspace{-2mm}\sum_{\pi=\{B_j\}\in\Pi_{p}} \hspace{-2mm}\mathrm{e}^{-\mathrm{i}(\ell - |\pi|k)\theta}\mathrm{d}^{|\pi|}u[\hat{f}_k]\big({\mathcal D}^{B_1}_{k}f_k, \cdots, {\mathcal D}^{B_{|\pi|}}_{k}f_k\big)\, \mathrm{d}\theta, \\
    ={}&\sum_{\pi=\{B_j\}\in\Pi_{p}} \hspace{-2mm}{\mathcal P}_{\ell-|\pi|k}[\mathrm{d}^{|\pi|}u\big[\hat{f}_k]\big({\mathcal D}^{B_1}_{k}f_k, \cdots, {\mathcal D}^{B_{|\pi|}}_{k}f_k\big)\big],
\end{split}
\end{equation*}
where the third line follows from the second because
$$\partial^B\hat{f}_k = \mathrm{e}^{\mathrm{i}(k+m_B^n - 2m_{B}^i)\theta}2^{-\frac{1}{2}m_B^n}\mathcal{D}^B_{k} f_{k}$$
(see Proposition \ref{Prop:Dk}) and
\begin{equation*}
    \begin{split}
        \sum_{B_j\in\pi} k =|\pi| k, \qquad\ \sum_{B_j\in\pi} m_{B_j}^n = n,\qquad\sum_{B_j\in\pi} m_{B_j}^i = i.\qedhere
    \end{split}
\end{equation*}

\end{proof}

\section{Function spaces}
\subsection{Classical function spaces}  \label{Classical FS}

Let $R>0$ and denote the open and closed origin-centred balls of radius $R$ in ${\mathbb R}^2$ by $B_R(\mathbf{0})$ and $\overline{B}_R(\mathbf{0})$ respectively.
In this section we consider classes of mode $k$ continuous functions $B_R(\mathbf{0}) \times {\mathbb R} \rightarrow X$, in particular characterising
the mode $k$ subspaces
\begin{itemize}
\item[(i)]
 $\hat{C}_{(k)}^m(B_R(\mathbf{0}) \times {\mathbb R};X)$ of the space $C^m(B_R(\mathbf{0}) \times {\mathbb R};X)$ of $m$-times continuously differentiable functions;
 \item[(ii)]
$\hat{C}_{(k)\mathrm{b}}^m(B_R(\mathbf{0}) \times {\mathbb R};X)$ of the space $C_\mathrm{b}^m(B_R(\mathbf{0}) \times {\mathbb R};X)$ of $m$-times boundedly continuously differentiable functions;
\item[(iii)]
$\hat{\mathscr D}_{(k)}(B_R(\mathbf{0}) \times {\mathbb R};X)$ of the space ${\mathscr D}(B_R(\mathbf{0}) \times {\mathbb R};X)$ of test functions with compact support in $B_R(\mathbf{0}) \times {\mathbb R}$;
\item[(iv)]
$\hat{\mathscr S}_{(k)}({\mathbb R}^2 \times {\mathbb R};X)$ of the space ${\mathscr S}({\mathbb R}^2 \times {\mathbb R};X)$ of Schwartz-class functions.
\end{itemize}
Note that $B_R(\mathbf{0})$ can be replaced by ${\mathbb R}^2$ in (i)--(iii) and by $\overline{B}_R(\mathbf{0})$ in (i), (ii) (with the obvious modifications to definitions
and results), and we also consider the subspace
$$
\hat{\mathscr D}_{(k)}(\overline{B}_R(\mathbf{0})\times\mathbb{R};X)=\{\hat{f}_k|_{\overline{B}_R(\mathbf{0})\times{\mathbb R}} \st \hat{f}_k \in \hat{\mathscr{D}}_{(k)}({\mathbb R}^2 \times \mathbb{R};X)\}$$
of
$$
{\mathscr D}(\overline{B}_R(\mathbf{0})\times\mathbb{R};X)=\{f|_{\overline{B}_R(\mathbf{0})\times{\mathbb R}} \st f \in {\mathscr D}({\mathbb R}^2 \times \mathbb{R};X)\}.$$
Functions which do not depend upon the axial coordinate are included as special cases.

\begin{lemma} \label{lem:C0}
A mode $k$ function $\hat{f}_k$ belongs to $C^0(B_R(\mathbf{0}) \times {\mathbb R};X)$ if and only if its radial coefficient $f_k$ belongs to
$$C_{(k)}^0([0,R) \times {\mathbb R};X) := \{f_k \in C^0([0,R)\times{\mathbb R};X) \st kf_k|_{r=0}=0\}.$$
\end{lemma}
\begin{proof}
Note that $\hat{f}_k \in C^0(B_R(\mathbf{0}) \times \mathbb{R};X)$ if and only if
$$\hat{f}_k \in C^0(B_R^\prime(\mathbf{0}) \times \mathbb{R};X), \qquad \lim\limits_{(\bfx,z) \rightarrow (\mathbf{0},z_0)} \hat{f}_k(\bfx,z)=\hat{f}_k(\mathbf{0},z_0)$$
for each $z_0 \in {\mathbb R}$, where $B_R^\prime(\mathbf{0}) = B_R(\mathbf{0}) \setminus \{\mathbf{0}\}$,
while
$f_k \in C^0([0,R)\times \mathbb{R};X)$ if and only if
$$f_k \in C^0((0,R) \times \mathbb{R};X), \qquad \lim\limits_{(r,z) \rightarrow (0,z_0)} f_k(r,z)=f_k(0,z_0)$$
for each $z_0 \in {\mathbb R}.$
However $\hat{f}_k \in C^0(B_R^\prime(\mathbf{0}) \times \mathbb{R};X)$ if and only if $f_k \in C^0((0,R) \times \mathbb{R};X)$ (because the polar-coordinate transformation
is a homeomorphism $B_R^\prime(\mathbf{0}) \rightarrow (0,R) \times {\mathbb T}^1$), while
$$\lim_{(\bfx,z) \rightarrow (\mathbf{0},z_0)} \|\hat{f}_0(\bfx,z)-\!\!\!\underbrace{\hat{f}_0(\mathbf{0},z_0)}_{\displaystyle = f_0(0,z_0)}\!\!\!\|=0$$
if and only if
$$\lim_{(r,z) \rightarrow (0,z_0)} \|f_0(r,z)-f_0(0,z_0)\|=0,$$
and for $k \neq 0$
$$\lim_{(\bfx,z) \rightarrow (\mathbf{0},z_0)} \|\hat{f}_k(\bfx,z_0)\|=0$$
if and only if
$$\lim_{(r,z) \rightarrow (0,z_0)}\underbrace{\|\e^{\i k \theta}f_k(r,z)\|}_{\displaystyle =  \|f_k(r,z)\|} =0.\eqno{\qedhere}$$
\end{proof}

\begin{corollary} \label{cor:Cm}
A mode $k$ function $\hat{f}_k$ belongs to $C^m(B_R(\mathbf{0}) \times \mathbb{R};X)$ if and only if its radial coefficient $f_k$ belongs to
$$C_{(k)}^m([0,R) \times \mathbb{R};X) := \{f_k: [0,R) \times {\mathbb R}\rightarrow X \st \mathcal{D}_{-k+i}^{n-i}\mathcal{D}_k^i \partial_z^{p-n} f_k \in C^0_{(k+n-2i)}([0,R) \times \mathbb{R};X),\ 0\leq i\leq n \leq p \leq m\}.$$
\end{corollary}
\begin{proof}
Observe that $\hat{f}_k \in C^m(B_R(\mathbf{0}) \times \mathbb{R};X)$ if and only if $\partial_{\bar{\zeta}}^{n-i}\partial_\zeta^i \partial_z^{p-n} \hat{f}_k$ belongs to
$C^0(B_R(\mathbf{0}) \times \mathbb{R};X)$ for $0 \leq i \leq n \leq p \leq m$. However
$$\partial_{\bar{\zeta}}^{n-i}\partial_\zeta^i \partial_z^{p-n}\hat{f}_k = \e^{\i(k+n-2i)\theta} 2^{-\frac{n}{2}}{\mathcal D}_{-k+i}^{n-i}{\mathcal D}_k^i\partial_z^{p-n} f_k,$$
such that the result follows from Lemma \ref{lem:C0}.
\end{proof}

\begin{remark}
Note that
\begin{align*}
C_{(k)}^m([0,R) \times \mathbb{R};X) = \{f_k: [0,R) \times &{\mathbb R} \rightarrow X \st \mathcal{D}_{-k+i}^{n-i}\mathcal{D}_k^i \partial_z^{p-n} f_k \in C^0([0,R)\times\mathbb{R};X), \\
&(k+n-2i)\mathcal{D}_{-k+i}^{n-i}\mathcal{D}_k^i\partial_z^pf_k|_{r=0}=0,\ 0\leq i\leq n\leq p \leq m\}.
\end{align*}
\end{remark}

The above arguments also show that a mode $k$ function $\hat{f}_k$ belongs to $C_\mathrm{b}^m(B_R(\mathbf{0}) \times \mathbb{R};X)$
if and only if its radial coefficient belongs to
$$C^m_{(k)\mathrm{b}}([0,R) \times \mathbb{R};X):=\{f_k \in C^m_{(k)}([0,R) \times \mathbb{R};X)\st \sup_{r < R \atop z \in {\mathbb R}} \|\mathcal{D}_{-k+i}^{n-i}\mathcal{D}_k^i\partial_z^{p-n}f_k(r,z)\|< \infty,
\ 0\leq i\leq n\leq p \leq m\},$$
to ${\mathscr D}(B_R(\mathbf{0}) \times \mathbb{R};X)$ if and only if
its radial coefficient belongs to
$$
{\mathscr D}_{(k)}([0,R) \times \mathbb{R};X) := \{f_k \in C^\infty_{(k)}([0,R) \times \mathbb{R};X)\st \mathrm{supp}\,f_k \Subset [0,R)\times{ \mathbb R}\},
$$
and to ${\mathscr S}({\mathbb R}^2 \times \mathbb{R};X)$ if and only if its radial coefficient belongs to
\begin{align*}
{\mathscr S}_{(k)}&([0,\infty) \times \mathbb{R};X) \\
&:=\! \{f_k \in C^\infty_{(k)}([0,\infty) \times \mathbb{R};X)\st \sup_{r \geq 0 \atop z \in {\mathbb R}} |(r,z)|^{m_1} \|\mathcal{D}_{-k+i}^{n-i}\mathcal{D}_k^i\partial_z^{p-n}f_k(r,z)\|^{m_2}\! <\! \infty,
\ 0\leq i\leq n \leq p,\ m_1,m_2, p \in {\mathbb N}_0\}.
\end{align*}
We refer to functions in ${\mathscr D}_{(k)}([0,R) \times \mathbb{R};X)$ and ${\mathscr S}_{(k)}([0,\infty) \times \mathbb{R};X)$ as \emph{radial test functions} and \emph{radial Schwartz-class functions}. In this context we also note that a mode $k$ function belongs to $\hat{\mathscr D}_{(k)}(\overline{B}_R(\mathbf{0})\times\mathbb{R};X)$ if and only if its radial coefficient belongs to
$$
{\mathscr D}_{(k)}([0,R] \times \mathbb{R};X) := \{f_k|_{[0,R] \times {\mathbb R}} \st f_k \in {\mathscr D}_{(k)}([0,\infty) \times \mathbb{R};X)\}.$$

\begin{lemma} $ $ \label{lem:Topologies of Cb and S}
\begin{itemize}
\item[(i)]
The space $C^m_{(k)\mathrm{b}}([0,R) \times \mathbb{R};X)$ is a Banach space with respect to the norm
$$\|f_k\|_{C_{(k)\mathrm{b}}^m} :=\hspace{-2mm} \maxs_{0 \leq i \leq n \leq p \leq m \atop } \sup_{r < R \atop z \in {\mathbb R}} \|\mathcal{D}_{-k+i}^{n-i}\mathcal{D}_k^i\partial_z^{p-n}f_k(r,z)\|.$$
\item[(ii)]
The space ${\mathscr S}_{(k)}([0,\infty) \times \mathbb{R};X)$ is a Fr\'{e}chet space with respect to the family 
$$\|f_k\|_{m_1,m_2,p}:=\hspace{-2mm} \maxs_{0 \leq i \leq n \leq p \atop } \sup_{r \geq 0 \atop z \in {\mathbb R}}
|(r,z)|^{m_1} \|\mathcal{D}_{-k+i}^{n-i}\mathcal{D}_k^i \partial_z^{p-n} f_k(r,z)\|^{m_2} , \qquad m_1,m_2, p \in {\mathbb N}_0\
$$
of seminorms.
\end{itemize}
\end{lemma}
\begin{proof}
This result follows from the facts that
$\hat{C}_{(k)\mathrm{b}}^m(B_R(\mathbf{0}) \times {\mathbb R};X)$
and $\hat{\mathscr S}_{(k)}(\mathbb{R}^{2}\times \mathbb{R};X)$ are closed subspaces of the Banach space
$$C_\mathrm{b}^m(B_R(\mathbf{0}) \times \mathbb{R};X)=\{\hat{f} \in C^m(B_R(\mathbf{0}) \times \mathbb{R};X)\st \maxs_{0 \leq i \leq n \leq p \leq m \atop } \sup_{|\bfx| \leq R \atop z \in {\mathbb R}}
\|\partial_{\bar{\zeta}}^{n-i}\partial_\zeta^i \partial_z^{p-n} \hat{f}(\bfx,z)\|<\infty\}$$
and the Fr\'{e}chet space
$${\mathscr S}({\mathbb R}^2 \times \mathbb{R};X)=\{\hat{f} \in C^\infty({\mathbb R}^2 \times \mathbb{R};X)\st \maxs_{0 \leq i \leq n \leq p \atop }   \sup_{\bfx \in {\mathbb R}^2 \atop z \in {\mathbb R}}|(\bfx,z)|^{m_1}\|\partial_{\bar{\zeta}}^{n-i}\partial_\zeta^i \partial_z^{p-n}\hat{f}({\mathbf x},z)\|^{m_2}<\infty,\ m_1,m_2, p \in {\mathbb N}_0\}$$
respectively,
and that the mapping $f_k \mapsto \hat{f}_k$ defines an isometric isomorphism\linebreak $C_{(k)\mathrm{b}}^m([0,R) \times \mathbb{R};X) \rightarrow \hat{C}_{(k)\mathrm{b}}^m(B_R(\mathbf{0}) \times \mathbb{R};X)$ and ${\mathscr S}_{(k)}([0,\infty) \times \mathbb{R};X) \rightarrow \hat{\mathscr S}_{(k)}({\mathbb R}^2 \times \mathbb{R};X)$.
\end{proof}

The following technical lemma shows how a test function can be written as a sum of radial test functions.

\begin{lemma} \label{lem:FS superconvergence}
Suppose that $\hat{\phi} \in C^\infty(B_R(\mathbf{0}) \times \mathbb{R};X)$, such that $(r,\theta,z) \mapsto \hat{\phi}(r\cos\theta,r\sin\theta,z)$ lies in\linebreak
$C^\infty([0,R) \times {\mathbb T}^1\times\mathbb{R};X)$. The respective functions
$$\phi_k(r,z):=\frac{1}{2\pi}\int_0^{2\pi} \hat{\phi}(r\cos\theta,r\sin\theta,z)\e^{-\i k \theta} \dtheta, \qquad k \in {\mathbb Z},$$
belong to $C^\infty_{(k)}([0,R) \times \mathbb{R};X)$ (with $\phi_k \in {\mathscr D}_{(k)}([0,R) \times \mathbb{R};X)$ if $\hat{\phi} \in {\mathscr D}(B_R(\mathbf{0}) \times \mathbb{R};X)$)
and the series $\sum\limits_{k \in {\mathbb Z}} \hat{\phi}_k$ converges compactly to $\hat{\phi}$
(and uniformly if $\hat{\phi} \in {\mathscr D}(B_R(\mathbf{0}) \times \mathbb{R};X)$).
\end{lemma}

\begin{proof}

The compact convergence of the series follows by familiar methods from the classical theory of Fourier series.
To show that $\phi_k \in C^\infty_{(k)}([0,R) \times \mathbb{R};X)$, observe that $\phi_k = {\mathcal P}_k[\hat{\phi}]$, where
$${\mathcal P}_k[\hat{\psi}](r,z)=\frac{1}{2\pi}\int_0^{2\pi} \hat{\psi}(r\cos\theta,r\sin\theta,z)\e^{-\i k \theta}\dtheta$$
(interchanging the sum and integral is justified due to the uniform convergence of the series for each fixed $(r,z)$).
It follows from Proposition \ref{prop:proj} and Remark \ref{propprojrig} that ${\mathcal P}_k[\hat{\phi}] \in C^\infty([0,R) \times {\mathbb R}; X)$
with
\begin{equation*}
    \begin{split}
        \mathcal{D}^{n-i}_{-k+i}\mathcal{D}^i_{k}\partial_z^{p-n}\mathcal{P}_{k}[\hat{\phi}] ={}& \mathcal{P}_{k+n-2i}[2^{\frac{n}{2}}\partial^{n-i}_{\overline{\zeta}}\partial^i_\zeta \partial_z^{p-n}\hat{\phi}].
    \end{split}
\end{equation*} 
Evidently
$$(k+n-2i)\mathcal{P}_{k+n-2i}[2^{\frac{n}{2}}\partial^{n-i}_{\overline{\zeta}}\partial^i_\zeta \partial_z^{p-n}\hat{\phi}]|_{r=0}=0,$$
such that $\phi_k \in C^\infty_{(k)}([0,R) \times \mathbb{R};X)$.
\end{proof}

\begin{remark}\label{rem:FS superconvergence - diff}
In the notation of the previous lemma, the series $\sum\limits_{k \in {\mathbb Z}} \partial_{\bar{\zeta}}^{n-i}\partial_\zeta^i \partial_z^{p-n}\hat{\phi}_k$
converges compactly to $\partial_{\bar{\zeta}}^{n-i}\partial_\zeta^i\partial_z^{p-n}\hat{\phi}$ for $0 \leq i \leq n \leq p$, $p \in {\mathbb N}_0$
(and uniformly if $\hat{\phi} \in {\mathscr D}(B_R(\mathbf{0}) \times \mathbb{R};X)$).
\end{remark}

\begin{corollary} \label{cor:FS superconvergence}
Suppose that $\hat{\phi} \in {\mathscr D}(\overline{B}_R(\mathbf{0}) \times \mathbb{R};X)$, such that $(r,\theta,z) \mapsto \hat{\phi}(r\cos\theta,r\sin\theta,z)$ lies in\linebreak
${\mathscr D}([0,R] \times {\mathbb T}^1\times\mathbb{R};X)$. The functions
$$\phi_k(r,z):=\frac{1}{2\pi}\int_0^{2\pi} \hat{\phi}(r\cos\theta,r\sin\theta,z)\e^{-\i k \theta} \dtheta, \qquad k \in {\mathbb Z},$$
belong to ${\mathscr D}_{(k)}([0,R] \times \mathbb{R};X)$
and the series $\sum\limits_{k \in {\mathbb Z}} \hat{\phi}_k$ converges uniformly to $\phi$.
\end{corollary}

The corresponding result for Schwartz-class functions is proved in the same fashion.

\begin{lemma} \label{lem:FS Schwartz superconvergence}
Suppose that $\hat{\phi} \in {\mathscr S}(\mathbb{R}^{2}\times \mathbb{R};X)$, such that $(r,\theta,z) \mapsto \hat{\phi}(r\cos\theta,r\sin\theta,z)$ lies in\linebreak
${\mathscr S}([0,\infty) \times {\mathbb T}^1\times\mathbb{R};X)$. The functions
$$\phi_k(r,z):=\frac{1}{2\pi}\int_0^{2\pi} \hat{\phi}(r\cos\theta,r\sin\theta,z)\e^{-\i k \theta} \dtheta, \qquad k \in {\mathbb Z},$$
belong to ${\mathscr S}_{(k)}([0,\infty) \times \mathbb{R};X)$
and the series $\sum\limits_{k \in {\mathbb Z}} \hat{\phi}_k$ converges uniformly to $\phi$.
\end{lemma}

The next proposition, which follows from
Corollary \ref{cor:Cm} and the observation that
$\hat{f}_k \in \hat{C}_k^m(B_R(\mathbf{0}) \times \mathbb{R};X)$, $\hat{g}_q \in \hat{C}_q^m(B_R(\mathbf{0}) \times \mathbb{R};X)$ implies
$\hat{f}_k\hat{g}_q \in \hat{C}_{k+q}^m(B_R(\mathbf{0}) \times \mathbb{R};X)$ (and similarly for the other
function spaces), gives information on the product of two continuous functions. A formula for $\mathcal{D}_{-k-q+i}^{n-i}\mathcal{D}_{k+q}^{i}\partial_z^{p-n}(f_kg_q)$
is given in Remark \ref{rem:product}.

\begin{proposition} \label{prop:products}
Let $X$ be a Banach algebra.
\begin{itemize}
\item[(i)]
Suppose that $f_k \in C_{(k)}^m([0,R) \times \mathbb{R};X)$ (or $C_{(k)\mathrm{b}}^m([0,R) \times \mathbb{R};X)$) and $g_q \in C_{(q)}^m([0,R) \times \mathbb{R};X)$ (or $C_{(q)\mathrm{b}}^m([0,R) \times \mathbb{R};X)$). It follows that $f_k  g_q \in C_{(k+q)}^m([0,R) \times \mathbb{R};X)$ (or $C_{(k+q)\mathrm{b}}^m([0,R)\times \mathbb{R};X)$).
\item[(ii)]
Suppose that $f_k \in C^\infty_{(k)}([0,\infty)\times \mathbb{R};X)$ and $g_q \in {\mathscr D}_{(q)}([0,\infty) \times \mathbb{R};X)$. It follows that\linebreak
$f_k  g_q \in \mathscr{D}_{(k+q)}([0,\infty) \times \mathbb{R};X)$.
\item[(iii)]
Suppose that $f_k \in C^\infty_{(k)}([0,\infty)\times \mathbb{R};X)$, $g_q \in \mathscr{S}_{(q)}([0,\infty) \times \mathbb{R};X)$ and
that $\|\mathcal{D}_{-k+i}^{n-i}\mathcal{D}_k^i\partial_z^{p-n}f_k(r,z)\|$ grows at most polynomially as $|(r,z)| \rightarrow \infty$
for all $0\leq i\leq n \leq p$ and $p \in {\mathbb N}_0$.
It follows that
$f_k  g_q \in \mathscr{S}_{(k+q)}([0,\infty) \times \mathbb{R};X)$.
\end{itemize}
\end{proposition}

Similar results are available for compositions of functions (see Remark \ref{propprojrig}, noting that\linebreak
$\ell{\mathcal P}_\ell[\hat{\psi}](r,z)|_{r=0}=0$); a formula for 
$\mathcal{D}_{-k-q+i}^{n-i}\mathcal{D}_{k+q}^{i}\partial_z^{p-n} u_\ell(f_k)$ is given in Remark \ref{prop:proj}.

\begin{proposition} Let $u \in C^0(X;X)$, $f_k \in C^0_{(k)}([0,R) \times {\mathbb R};X)$ and define
$u_\ell(f_k) \in C^0_{(\ell)}([0,R) \times {\mathbb R};X)$
by $u_\ell(f_k) = {\mathcal P}_\ell[u \circ \hat{f}_k]$, where
$${\mathcal P}_\ell[\hat{\psi}](r,z):=\frac{1}{2\pi}\int_0^{2\pi} \hat{\psi}(r\cos\theta,r\sin\theta,z)\e^{-\i \ell \theta}\dtheta$$
for $\hat{\psi} \in C^0[(0,R) \times \mathbb{R};X)$.
\begin{itemize}
\item[(i)]
Suppose that $f_k \in C_{(k)}^m([0,R) \times \mathbb{R};X)$ (or $C_{(k)\mathrm{b}}^m([0,R)\times \mathbb{R};X)$) and $u \in C^m(X;X)$ (or $u \in C_\mathrm{b}^m(X;X)$). It follows that $u_\ell(f_k) \in C_{(\ell)}^m([0,R) \times \mathbb{R};X)$
(or $C_{(\ell)\mathrm{b}}^m([0,R) \times \mathbb{R};X)$).
\item[(ii)]
Suppose that $f_k \in {\mathscr D}_{(k)}([0,R) \times \mathbb{R};X)$ and $u \in C^\infty(X;X)$. It follows that $u_\ell(f_k) \in {\mathscr D}_{(\ell)}([0,R) \times \mathbb{R};X)$.
\item[(iii)]
Suppose that $f_k \in {\mathscr S}_{(k)}([0,\infty) \times \mathbb{R};X)$,  $u \in C^\infty(X;X)$ and that $\|\mathrm{d}^iu[x]\|$ grows at most polynomially as $\|x\| \rightarrow \infty$ for all $i \in {\mathbb N}_0$.
 It follows that $u_\ell(f_k)\in {\mathscr S}_{(\ell)}([0,\infty) \times \mathbb{R};X)$.
\end{itemize}
\end{proposition}

Finally, we record some useful results concerning standard radial functions.
\begin{lemma} \label{lem:power}
Suppose that $\ell\in\mathbb{N}_{0}$.
\begin{itemize}
\item[(i)]
The function $f(r)=r^\ell$ belongs to $C_{(k)}^m([0,R);\mathbb{R})$ for all $k\in\mathbb{Z}$ if $\ell>m$;
\item[(ii)]
The function $f(r)=r^\ell$ belongs to $C^\infty_{(k)}([0,R); \mathbb{R})$ for $k=\ell,\ell-2,\ldots,-\ell$.
\end{itemize}
\end{lemma}

      \begin{proof}
      Since
$$
            \mathcal{D}_{k}(r^\ell) = (k + \ell)r^{\ell -1}
$$
we find that
$$
\mathcal{D}_{-k+i}^{n-i}\mathcal{D}_{k}^{i}(r^\ell)= \prod_{j=0}^{i-1}(k+\ell -2j)\mathcal{D}_{-k+i}^{n-i}(r^{\ell-i})
 = \prod_{j=0}^{n-i-1}(-k+\ell -2j)\prod_{j=0}^{i-1}(k+\ell -2j) r^{\ell-n}.
$$
Evidently
$\mathcal{D}_{-k+i}^{n-i}\mathcal{D}_{k}^{i}(r^\ell)$ belongs to $C^0([0,R);{\mathbb R})$ and
  vanishes at the origin for all $0 \leq i \leq n < \ell$ ; in particular, this observation establishes part (i).
  
To establish part (ii) we examine $\mathcal{D}_{-k+i}^{\ell-i}\mathcal{D}_{k}^{i}(r^\ell)$. This derivative is constant
and hence belongs to $C^0([0,\infty);{\mathbb R})$ for all $0 \leq i \leq \ell$; furthermore
$$
            (k+\ell-2i)\mathcal{D}_{-k+i}^{\ell-i}\mathcal{D}_{k}^{i}(r^\ell)|_{r=0}=(-1)^{\ell-i}\prod_{j=0}^{\ell}(k+\ell - 2j)
$$
vanishes for all $0 \leq i \leq \ell$ if and only if $k+\ell-2j=0$ for some $j$ with $0 \leq j \leq \ell$. In this case $\mathcal{D}_{-k+i}^{\ell-i}\mathcal{D}_{k}^{i}(r^\ell)$ is identically
zero, such that all higher-order radial derivatives exist and are also identically zero.
\end{proof}

\begin{lemma} \label{lem:Sobfactor}
The function $f(r)=(1+r^2)^{\frac{s}{2}}$ belongs to $C_{(0)}^\infty([0,R);{\mathbb R})$ for every $s \in {\mathbb R}$.
\end{lemma}
\begin{proof}
Let $f_i(r):= (1 + r^2)^{\frac{s}{2}-i}$ for $i \in {\mathbb N}_0$. By induction over $m$ we find that
\begin{equation*}
    \mathcal{D}^{m}_{0}f_i(r) = \prod_{j=i}^{i+m-1}(s-2j) \, r^{m}f_{i+m}(r),
    \qquad
        \mathcal{D}^{m}_{i}(r^{i})=\prod_{j=0}^{m-1}2(i-j) r^{i-m}
\end{equation*}
for all $m \in {\mathbb N}_0$, such that
\begin{equation*}\begin{split}
    \mathcal{D}^{n-i}_{i}\mathcal{D}_{0}^{i}f_{0}(r) ={}& \prod_{j=0}^{i-1}(s-2j)\, \mathcal{D}^{n-i}_{i}(r^{i}f_{i}(r)) \\
    ={}& \prod_{j=0}^{i-1}(s-2j)\, \sum_{\ell = 0}^{n-i} \begin{pmatrix} n-i \\ \ell \end{pmatrix} \mathcal{D}^{n-i-\ell}_{i}(r^{i})\mathcal{D}^{\ell}_{0}(f_{i}(r))\\
    ={}& \prod_{j=0}^{i-1}(s-2j)\, \sum_{\ell = i}^{n} \begin{pmatrix} n-i \\ \ell-i \end{pmatrix} \mathcal{D}^{n-\ell}_{i}(r^{i})\mathcal{D}^{\ell-i}_{0}(f_{i}(r))\\
    ={}& \prod_{j=0}^{i-1}(s-2j)\, \sum_{\ell = i}^{n} \begin{pmatrix} n-i \\ \ell-i \end{pmatrix} \prod_{j=0}^{n-\ell-1}2(i-j) r^{i-n+\ell}\prod_{j=i}^{\ell-1}(s-2j) r^{\ell-i}f_{\ell}(r)\\
    ={}& \sum_{\ell = i}^{n} \begin{pmatrix} n-i \\ \ell-i \end{pmatrix} \prod_{j=0}^{n-\ell-1}2(i-j)\, \prod_{j=0}^{\ell-1}(s-2j) \, r^{2\ell-n} \,f_{\ell}(r)
\end{split}
\end{equation*}
belongs to $C^0([0,\infty);\mathbb{R})$ for all $n \in {\mathbb N}_0$.

To verify that
\begin{equation*}
    (n-2i)\mathcal{D}^{n-i}_{i}\mathcal{D}_{0}^{i}f_{0}(r)|_{r=0} = 0
\end{equation*}
for $0\leq i\leq n$, $n\in\mathbb{N}_{0}$ we examine each term in the sum over $\ell$ separately. Obviously $r^{2\ell-n}|_{r=0}=0$ for $2\ell>n$, while $\prod\limits_{j=0}^{n-\ell-1}2(i-j) = 0$
for $i +\ell \leq n-1$, such that $\ell \leq i\leq n-\ell-1$; the remaining case $i+\ell > n-1$ and $2\ell\leq n$ arises only when
$n=2i = 2\ell$, such that $n-2i=0$.
\end{proof}

\subsection{Spaces of distributions} \label{Distributions}

In this section we study distributions acting on the spaces $\hat{\mathscr D}_{(k)}(B_R(\mathbf{0}) \times {\mathbb R};{\mathbb C})$
and $\hat{\mathscr S}_{(k)}(\mathbb{R}^{2} \times {\mathbb R};{\mathbb C})$ of
mode $k$ test functions and Schwartz-class functions;
again $B_R(\mathbf{0})$ can be replaced by ${\mathbb R}^2$ and functions which do not depend upon the axial coordinate
are included as special cases.

Recall that a \emph{distribution} $T: {\mathscr D}(B_R(\mathbf{0}) \times {\mathbb R};\mathbb{C}) \rightarrow X$ is a sequentially continuous linear functional, where
convergence of the sequence $\{\hat{\phi}^j\} \subseteq {\mathscr D}(B_R(\mathbf{0}) \times {\mathbb R};{\mathbb C})$ to $\hat{\phi} \in {\mathscr D}(B_R(\mathbf{0}) \times {\mathbb R};{\mathbb C})$
means that there is a compact subset $K$ of $B_R(\mathbf{0}) \times {\mathbb R}$ such that
\begin{itemize}\vspace{-0.5\baselineskip}
\item[(i)]
$\mathrm{supp}\, \hat{\phi}^j \subseteq K$ for all $j \in {\mathbb N}$,
\item[(ii)]
$\partial_{\bar{\zeta}}^{n-i}\partial_\zeta^i \partial_z^{p-n} \hat{\phi}^j \rightarrow \partial_{\bar{\zeta}}^{n-i}\partial_\zeta^i \partial_z^{p-n} \hat{\phi}$
uniformly over $K$ for $0 \leq i \leq n \leq p$ and $p \in {\mathbb N}_0$.
\end{itemize}\vspace{-0.5\baselineskip}
We denote the set of all such distributions by ${\mathscr D}^\prime(B_R(\mathbf{0})  \times {\mathbb R};X)$.
Without loss of generality we can take $K=\overline{B}_{r_0}(\mathbf{0}) \times [-z_0,z_0]$ for some $0<r_0<R$ and $z_0>0$ in the definition of convergence,
and hence define the class $\hat{\mathscr D}_{(k)}^\prime(B_R(\mathbf{0}) \times {\mathbb R};X)$ of \emph{mode $k$ distributions} by replacing
${\mathscr D}(B_R(\mathbf{0}) \times {\mathbb R};{\mathbb C})$ by $\hat{\mathscr D}_{(k)}(B_R(\mathbf{0}) \times {\mathbb R};{\mathbb C})$ in the above construction.
Furthermore, a mode $k$ distribution $\hat{T}_k \in \hat{\mathscr D}_{(k)}^\prime(B_R(\mathbf{0}) \times {\mathbb R};X)$
can be extended to a distribution in ${\mathscr D}^\prime(B_R(\mathbf{0}) \times {\mathbb R};X)$ by
writing a test function $\hat{\phi} \in {\mathscr D}(B_R(\mathbf{0}) \times \mathbb{R};{\mathbb C})$ as a sum $\hat{\phi} = \sum\limits_{\ell \in {\mathbb Z}} \hat{\phi}_\ell$ of
radial test functions $\hat{\phi}_\ell \in \hat{\mathscr D}_{(\ell)}(B_R(\mathbf{0}) \times \mathbb{R};\mathbb{C})$ (see Lemma \ref{lem:FS superconvergence}) and defining
$$( \hat{T}_k, \hat{\phi} ) := ( \hat{T}_k, \hat{\phi}_k ).$$

\begin{definition}
A \underline{$k$-index radial distribution} is a sequentially continuous linear functional\linebreak
$T_k: {\mathscr D}_{(k)}([0,R) \times {\mathbb R};{\mathbb C}) \rightarrow X$, where
convergence of the sequence $\{\phi_k^j\} \subseteq {\mathscr D}_{(k)}([0,R) \times {\mathbb R};{\mathbb C})$ to\linebreak $\phi_k \in {\mathscr D}_{(k)}([0,R) \times {\mathbb R};{\mathbb C})$
means that there exists $0<r_0<R$ and $z_0>0$ such that
\begin{itemize}\vspace{-0.5\baselineskip}
\item[(i)]
$\mathrm{supp}\, \phi^j_k \subseteq [0,r_0] \times [-z_0,z_0]$ for all $j \in {\mathbb N}$,
\item[(ii)]
$\mathcal{D}_{-k+i}^{n-i}\mathcal{D}_k^i \partial_z^{p-n} \phi^j_k \rightarrow \mathcal{D}_{-k+i}^{n-i}\mathcal{D}_k^i \partial_z^{p-n} \phi_k$
uniformly over $[0,r_0] \times [-z_0,z_0]$ for $0 \leq i \leq n \leq p$ and $p \in {\mathbb N}_0$.
\end{itemize}\vspace{-0.5\baselineskip}
We denote the set of all $k$-index radial distributions by ${\mathscr D}_{(k)}^\prime([0,R) \times {\mathbb R};X)$.
\end{definition}

Note that the linear mapping $\hat{T}_k: \hat{\mathscr D}_{(k)}(B_R(\mathbf{0}) \times {\mathbb R};{\mathbb C}) \rightarrow X$ 
belongs to $\hat{\mathscr D}_{(k)}^\prime(B_R(\mathbf{0}) \times {\mathbb R};X)$ if and only if $T_k: {\mathscr D}_{(k)}([0,R) \times {\mathbb R};{\mathbb C}) \rightarrow X$ given by
$$(T_k,\phi_k) = (\hat{T}_k,\hat{\phi}_k)$$
belongs to ${\mathscr D}_{(k)}^\prime([0,R) \times {\mathbb R};X)$.

Our next result characterises the distributional derivative of a mode $k$ distribution.

\begin{lemma} \label{lem:distderiv construction}
Suppose that $\hat{T}_k \in \hat{\mathscr D}_{(k)}^\prime(B_R(\mathbf{0}) \times {\mathbb R};X)$. The distribution $\hat{S}_{k+n-2i}\in \hat{\mathscr D}_{(k+n-2i)}^\prime(B_R(\mathbf{0}) \times {\mathbb R};X)$ is the distributional derivative
$\partial_{\bar{\zeta}}^{n-i}\partial_\zeta^i \partial_z^{p-n}\hat{T}_k$ of $\hat{T}_k$ if and only if
\begin{equation}
(-1)^p( \hat{T}_k ,\partial_{\bar{\zeta}}^i \partial_{{\zeta}}^{n-i} \partial_z^{p-n}\hat{\phi}_{k+n-2i})
=
(\hat{S}_{k+n-2i}, \hat{\phi}_{k+n-2i})
\label{eq:distderiv 1}
\end{equation}
for all $\hat{\phi}_{k+n-2i} \in \hat{\mathscr D}_{(k+n-2i)}(B_R(\mathbf{0}) \times \mathbb{R};{\mathbb C})$, and this condition holds if and only if
\begin{equation}
(-1)^p( T_k,{2^{-\frac{n}{2}}\mathcal D}_{-k+i}^i \mathcal{D}_{k+n-2i}^{n-i} \partial_z^{p-n}\phi_{k+n-2i} )
= ( S_{k+n-2i}, \phi_{k+n-2i})
\label{eq:distderiv 2}
\end{equation}
for all $\phi_{k+n-2i} \in {\mathscr D}_{(k+n-2i)}([0,R) \times \mathbb{R};{\mathbb C})$.
\end{lemma}
\begin{proof}
We proceed by writing a test function $\hat{\phi} \in {\mathscr D}(B_R(\mathbf{0}) \times \mathbb{R};{\mathbb C})$ as a sum $\hat{\phi} = \sum\limits_{\ell \in {\mathbb Z}} \hat{\phi}_\ell$ of
radial test functions $\hat{\phi}_\ell \in \hat{\mathscr D}_{(\ell)}(B_R(\mathbf{0}) \times \mathbb{R};{\mathbb C})$ (see Lemma \ref{lem:FS superconvergence}); the general derivative $\partial_{\bar{\zeta}}^i \partial_{{\zeta}}^{n-i} \partial_z^{p-n}\hat{\phi}$ is likewise written as $\partial_{\bar{\zeta}}^i \partial_{{\zeta}}^{n-i} \partial_z^{p-n}\hat{\phi} = \sum\limits_{\ell \in {\mathbb Z}} \partial_{\bar{\zeta}}^i \partial_{{\zeta}}^{n-i} \partial_z^{p-n}\hat{\phi}_\ell$ (see Remark~\ref{rem:FS superconvergence - diff}) where $\partial_{\bar{\zeta}}^i \partial_{{\zeta}}^{n-i} \partial_z^{p-n}\hat{\phi}_{\ell+n-2i} \in \hat{{\mathscr D}}_{(\ell)}(B_R(\mathbf{0}) \times \mathbb{R};{\mathbb C})$. \linebreak We thus find that
\begin{align*}
( \hat{T}_k ,\partial_{\bar{\zeta}}^i \partial_{{\zeta}}^{n-i} \partial_z^{p-n}\hat{\phi} )
& = ( \hat{T}_k ,\partial_{\bar{\zeta}}^i \partial_{{\zeta}}^{n-i} \partial_z^{p-n}\hat{\phi}_{k+n-2i})\\
& =  ( T_k,{2^{-\frac{n}{2}}\mathcal D}_{-k+i}^i \mathcal{D}_{k+n-2i}^{n-i} \partial_z^{p-n}\phi_{k+n-2i} ),\\
\\
( \hat{S}_{k+n-2i},\hat{\phi} ) & = ( \hat{S}_{k+n-2i},\hat{\phi}_{k+n-2i} )\\
& = ( S_{k+n-2i},\phi_{k+n-2i} ),
\end{align*}
such that
$$(-1)^p( \hat{T}_k ,\partial_{\bar{\zeta}}^i \partial_{{\zeta}}^{n-i} \partial_z^{p-n}\hat{\phi})
=
(  \hat{S}_{k+n-2i},\hat{\phi})$$
for all $\hat{\phi} \in {\mathscr D}(B_R(\mathbf{0}) \times \mathbb{R};{\mathbb C})$ if and only if \eqref{eq:distderiv 1} holds for all $\hat{\phi}_{k+n-2i} \in \hat{\mathscr D}_{(k+n-2i)}(B_R(\mathbf{0}) \times \mathbb{R};{\mathbb C})$ and if and only if
\eqref{eq:distderiv 2} holds for all $\phi_{k+n-2i} \in {\mathscr D}_{(k+n-2i)}([0,R)\times\mathbb{R};{\mathbb C})$.
\end{proof}

\begin{definition} \label{Defn of dist rad deriv}
The element $2^\frac{n}{2}S_{k+n-2i} \in 
{\mathscr D}_{(k+n-2i)}^\prime((0,R)\times \mathbb{R};X)$
is the \underline{distributional radial derivative} ${\mathcal D}_{-k+i}^{n-i}{\mathcal D}_k^i \partial_z^{p-n}T_k$
of $T_k \in {\mathscr D}_{(k)}^\prime((0,R)\times \mathbb{R};X)$
if it satisfies \eqref{eq:distderiv 2} for all
$\phi_{k+n-2i} \in {\mathscr D}_{(k+n-2i)}([0,R)\times\mathbb{R};{\mathbb C})$.
\end{definition}

Recall that a distribution $T \in {\mathscr D}^\prime(B_R(\mathbf{0})  \times {\mathbb R};X)$ is identified with a (necessarily unique)
function\linebreak
$f \in L^2(B_R(\mathbf{0}) \times {\mathbb R};X)$ if
$$(T,\hat{\phi}) = \hspace{-4mm}\int\limits_{B_R(\mathbf{0})\times {\mathbb R}} \hspace{-4mm}
f \hat{\phi}$$
for all $\hat{\phi} \in {\mathscr D}(B_R(\mathbf{0}) \times {\mathbb R};{\mathbb C})$. Similarly, we identify a mode $k$
distribution $\hat{T}_k \in \hat{\mathscr D}_{(k)}^\prime(B_R(\mathbf{0}) \times {\mathbb R};X)$ with a mode $-k$
function $\hat{f}_{-k} \in L^2(B_R(\mathbf{0}) \times {\mathbb R};X)$ if
$$(\hat{T}_k,\hat{\phi}_k) = \hspace{-4mm}\int\limits_{B_R(\mathbf{0})\times {\mathbb R}} \hspace{-4mm}
\hat{f}_{-k} \hat{\phi}_k
$$
for all $\hat{\phi}_k \in \hat{\mathscr D}_{(k)}(B_R(\mathbf{0}) \times {\mathbb R};{\mathbb C})$; the calculation
\begin{align*}
(\hat{T}_k,\hat{\phi}) &= (\hat{T}_k,\hat{\phi}_k)= \hspace{-4mm}\int\limits_{B_R(\mathbf{0})\times {\mathbb R}} \hspace{-4mm}
\hat{f}_{-k} \hat{\phi}_k = 2\pi\int_{\mathbb R}\int_0^R f_{-k}(r,z)\phi_k(r,z)r\dr\dz, \\
\int\limits_{B_R(\mathbf{0})\times {\mathbb R}} \hspace{-4mm} \hat{f}_{-k}\hat{\phi}
&=\sum_{\ell \in {\mathbb Z}}
\int\limits_{B_R(\mathbf{0})\times {\mathbb R}} \hspace{-4mm}\hat{f}_{-k}\hat{\phi}_\ell
= 2\pi\sum_{\ell \in {\mathbb Z}}
\delta_{-k, \ell} \int_{\mathbb R} \int_0^R f_{-k}(r,z)\phi_\ell(r,z)r\dr\dz
= 2\pi\int_{\mathbb R}\int_0^R f_{-k}(r,z)\phi_k(r,z)r\dr\dz,
\end{align*}
where $\hat{\phi} = \sum\limits_{\ell \in {\mathbb Z}} \hat{\phi}_\ell$ for $\hat{\phi} \in {\mathscr D}(B_R(\mathbf{0}) \times \mathbb{R};{\mathbb C})$
and $\hat{\phi}_\ell \in \hat{\mathscr D}_{(\ell)}(B_R(\mathbf{0}) \times \mathbb{R};{\mathbb C})$ (see Lemma \ref{lem:FS superconvergence}),
shows that $\hat{T}_k = \hat{f}_{-k}$ in the above sense. Finally, noting that $(\hat{T}_k,\hat{\phi}_k)=(T_k,\phi_k)$ and
$\hat{f}_{-k}$ belongs to $L^2(B_R(\mathbf{0}) \times {\mathbb R};X)$ if and only if its radial coefficient $f_{-k}$ belongs to
$$L_1^2((0,R)\times{\mathbb R};X) = \left\{f: [0,R)\times{\mathbb R} \rightarrow X\st 2\pi\int_{{\mathbb R}} \int_0^R \|f(r,z)\|^2 r \dr \dz< \infty\right\},$$
we identify a $k$-index radial distribution $T_k \in {\mathscr D}_{(k)}^\prime([0,R) \times {\mathbb R};X)$ with
$f_{-k} \in L_1^2((0,R)\times{\mathbb R};X)$ if
$$(T_k,\phi_k) = 2\pi\int_{\mathbb R}\int_0^R f_{-k}(r,z)\phi_k(r,z)r\dr\dz$$
for all $\phi_k \in {\mathscr D}_{(k)}([0,R) \times {\mathbb R};{\mathbb C})$.

The next definition deals with products of distributions with smooth functions. Note that\linebreak
$\hat{f}_q \in \hat{C}^\infty_{(q)}(B_R(\mathbf{0}) \times \mathbb{R};{\mathbb C})$, $\hat{\phi}_k \in \hat{\mathscr D}_{(k)}(B_R(\mathbf{0}) \times \mathbb{R};{\mathbb C})$ implies that
$\hat{f}_q\hat{\phi}_k \in \hat{\mathscr D}_{(k+q)}(B_R(\mathbf{0}) \times \mathbb{R};{\mathbb C})$,
and $\hat{\phi}_k^\ell \rightarrow \hat{\phi}_k$ in $\hat{\mathscr D}_{(k)}(B_R(\mathbf{0})\times {\mathbb R}; {\mathbb C})$
implies that $\hat{f}_q \hat{\phi}_k^\ell \rightarrow \hat{f}_q\hat{\phi}_k$ in $\hat{\mathscr D}_{(k+q)}(B_R(\mathbf{0}) \times {\mathbb R}; {\mathbb C})$,
so that the products are well defined.

\begin{definition} \label{Defn of products with functions}
$ $
\begin{itemize}
\item[(i)]
The product $\hat{f}_q \hat{T}_k$ of the function $\hat{f}_q \in \hat{C}_{(q)}^\infty(B_R(\mathbf{0}) \times {\mathbb R}; {\mathbb C})$ and the mode $k$ distribution \linebreak
$\hat{T}_k \in \hat{\mathscr D}^\prime_{(k)}(B_R(\mathbf{0}) \times {\mathbb R}; X)$ is the element of
$\hat{\mathscr D}^\prime_{(k-q)}(B_R(\mathbf{0}) \times {\mathbb R}; X)$ defined by
$$(\hat{f}_q\hat{T}_k,\hat{\phi}_{k-q}) := (\hat{T}_{k}, \hat{f}_q \hat{\phi}_{k-q}).$$
\item[(ii)]
The product
$f_q T_k$ of the function $f_q \in C^\infty_{(q)}([0,R) \times {\mathbb R}; {\mathbb C})$ and the radial distribution\linebreak
$T_k \in {\mathscr D}^\prime_{(k)}([0,R) \times {\mathbb R}; X)$ is the element of
${\mathscr D}^\prime_{(k-q)}([0,R) \times {\mathbb R}; X)$ defined by
$$(f_qT_k,\phi_{k-q}) := (T_{k}, f_q \phi_{k-q}).$$
\end{itemize}
\end{definition}

A similar theory for tempered distributions is readily constructed. A \emph{tempered distribution} is a continuous linear functional 
${\mathscr S}({\mathbb R}^2 \times {\mathbb R};{\mathbb C}) \rightarrow X$ 
(where continuity is defined by the Fr\'{e}chet-space structure of\linebreak ${\mathscr S}({\mathbb R}^2 \times {\mathbb R};{\mathbb C})$).
We denote the set of all such distributions by ${\mathscr S}^\prime({\mathbb R}^2  \times {\mathbb R};X)$.\pagebreak

\begin{definition}$ $ \label{tempered definition}
\begin{itemize}
\item[(i)]
A \underline{mode $k$ tempered distribution} is a continuous linear functional 
$\hat{\mathscr S}_{(k)}({\mathbb R}^2 \times {\mathbb R};{\mathbb C}) \rightarrow X$ 
(where continuity is defined by the Fr\'{e}chet-space structure of $\hat{\mathscr S}_{(k)}({\mathbb R}^2 \times {\mathbb R};{\mathbb C})$).
We denote the set of all such distributions by $\hat{\mathscr S}_{(k)}^\prime({\mathbb R}^2  \times {\mathbb R};X)$.\item[(ii)]
A \underline{$k$-index tempered radial distribution}  is a continuous linear functional 
${\mathscr S}_{(k)}([0,\infty) \times {\mathbb R};{\mathbb C}) \rightarrow X$ (where continuity is defined by the Fr\'{e}chet-space structure of ${\mathscr S}_{(k)}([0,\infty) \times {\mathbb R};{\mathbb C})$).
We denote the set of all such distributions by ${\mathscr S}_{(k)}^\prime([0,\infty)  \times {\mathbb R};X)$.
\end{itemize}
\end{definition}

The theory for tempered distributions is completely analogous to the above treatment of regular\linebreak
distributions. In particular,
mode $k$ tempered distributions extend to tempered distributions in\linebreak
${\mathscr S}^\prime([0,\infty) \times {\mathbb R};X)$ (see Lemma \ref{lem:FS Schwartz superconvergence}), the mapping $T_k \mapsto \hat{T}_k$, where $( \hat{T}_k, \hat{\phi}_k ) := ( T_k, \phi_k )$, is a bijection
${\mathscr S}_{(k)}^\prime([0,\infty) \times {\mathbb R};X) \rightarrow \hat{\mathscr S}_{(k)}^\prime([0,\infty) \times {\mathbb R};X)$,
and distributional derivatives and products with slowly growing functions are constructed along the lines described in Lemma \ref{lem:distderiv construction}, Definition \ref{Defn of dist rad deriv} and Definition \ref{Defn of products with functions}.

\subsection{Sobolev spaces} \label{Sobolev FS}

In this section we study the subspace $\hat{H}_{(k)}^m(B_R(\mathbf{0}) \times {\mathbb R};X)$
of the Sobolev space $H^m(B_R(\mathbf{0}) \times {\mathbb R}; X)$ consisting of  mode $k$
functions; once again $B_R(\mathbf{0})$ can be replaced by ${\mathbb R}^2$ and functions which do not depend upon the axial coordinate
are included as special cases.

A mode $k$ function
$\hat{f}_k$ evidently belongs to $L^2(B_R(\mathbf{0}) \times {\mathbb R};X)$ if and only if its radial coefficient $f_k$ belongs to
$$L_1^2((0,R)\times{\mathbb R};X) = \left\{f: [0,R)\times{\mathbb R} \rightarrow X\st 2\pi\int_{{\mathbb R}} \int_0^R \|f(r,z)\|^2 r \dr \dz< \infty\right\},$$
and this observation leads us to define
$\hat{H}_{(k)}^m(B_R(\mathbf{0}) \times {\mathbb R};X)$ as the completion of
$\hat{\mathscr D}_{(k)}(\overline{B}_R(\mathbf{0}) \times {\mathbb R};X)$ with respect to the norm
\begin{equation}
\| \hat{f}_k\|_{H^m}^2 := \sum_{p=0}^{m}\sum_{n=0}^p \sum_{i=0}^n \begin{pmatrix} n \\ i \end{pmatrix}\
\| \partial_{\bar{\zeta}}^{n-i} \partial_\zeta^i \partial_z^{p-n}\hat{f}_k\|_{L^2}^2 \label{Hat norm}
\end{equation}
and $H_{(k)}^m((0,R) \times {\mathbb R};X)$
as the completion of ${\mathscr D}_{(k)}([0,R] \times {\mathbb R};X)$ with respect to the norm
\begin{equation}
\|f_k\|_{H^{m}_{(k)}}^2:= \sum_{p=0}^{m}\sum_{n=0}^p 2^{-n} \sum_{i=0}^n \begin{pmatrix} n \\ i \end{pmatrix} 
\|{\mathcal D}_{-k+i}^{n-i}{\mathcal D}_k^i \partial_z^{p-n}f_k\|_{L_1^2}^2. \label{No hat norm}
\end{equation}
The following result follows directly from these definitions.

\begin{proposition}
$ $
\begin{itemize}
\item[(i)]
The space $\hat{H}_{(k)}^m(B_R(\mathbf{0}) \times {\mathbb R};X)$ is a closed subspace of $H^m(B_R(\mathbf{0}) \times {\mathbb R}; X)$.
\item[(ii)]
The mode $k$ function $\hat{f}_k$ belongs to $\hat{H}_{(k)}^m(B_R(\mathbf{0}) \times {\mathbb R};X)$ if and only if
its radial coefficient $f_k$ belongs to $H_{(k)}^m((0,R) \times {\mathbb R};X)$, and the mapping 
$f_k \mapsto \hat{f}_k$ is an isometric isomorphism.
\item[(iii)]
The space $\hat{\mathscr D}_{(k)}(B_R(\mathbf{0})\times{\mathbb R};X) \cap \hat{H}_{(k)}^m(B_R(\mathbf{0})\times{\mathbb R};X)$
is a dense subspace of $\hat{H}_{(k)}^m(B_R(\mathbf{0})\times{\mathbb R};X)$ and the space
${\mathscr D}_{(k)}((0,R)\times{\mathbb R};X) \cap H_{(k)}^m((0,R)\times{\mathbb R};X)$
is a dense subspace of $H_{(k)}^m((0,R)\times{\mathbb R};X)$.
\end{itemize}
\end{proposition}

\begin{definition}
The closure of
$\hat{\mathscr D}_{(k)}(B_R(\mathbf{0}) \times {\mathbb R};X)$ in
$\hat{H}_{(k)}^m(B_R(\mathbf{0}) \times {\mathbb R};X)$
and
${\mathscr D}_{(k)}(B_R(\mathbf{0}) \times {\mathbb R};X)$ in
$H_{(k)}^m((0,R) \times {\mathbb R};X)$ is denoted respectively by
$\hat{H}_{(k)0}^m(B_R(\mathbf{0}) \times {\mathbb R};X)$
and
$H_{(k)0}^m((0,R) \times {\mathbb R};X)$.
\end{definition}

We now turn to fractional-order spaces.
\begin{definition} \label{fractional Sob}
Write $s \geq 0$ as $s=m+\theta$, where $m=\lfloor s \rfloor$. We define
\begin{align*}
\hat{H}_{(k)}^s(B_R(\mathbf{0}) \times {\mathbb R};X)&:=[\hat{H}_{(k)}^{m} (B_R(\mathbf{0}) \times {\mathbb R};X),\hat{H}_{(k)}^{m+1}(B_R(\mathbf{0}) \times {\mathbb R};X)]_\theta, \\
H_{(k)}^s((0,R) \times {\mathbb R};X)&:=[H_{(k)}^{m}((0,R) \times {\mathbb R};X),H_{(k)}^{m+1}((0,R) \times {\mathbb R};X)]_\theta,
\end{align*}
where the square bracket denotes complex interpolation in the sense of Calder\'{o}n \cite{Calderon64} and
Lions \& Magenes \cite[vol.\ I, ch.\ 1]{LionsMagenes}.
\end{definition}

\begin{remarks}
$ $
\begin{itemize}
\item[(i)]
The mode $k$ function $\hat{f}_k$ belongs to $\hat{H}_{(k)}^s(B_R(\mathbf{0}) \times {\mathbb R};X)$ if and only if
its radial coefficient $f_k$ belongs to $H_{(k)}^s((0,R) \times {\mathbb R};X)$, and the mapping 
$f_k \mapsto \hat{f}_k$ is an isometric isomorphism.
\item[(ii)]
Using a result by Triebel \cite[\S1.17.1, Theorem 1]{Triebel}, we find that $\hat{H}_{(k)}^s(B_R(\mathbf{0}) \times {\mathbb R};X)$ coincides with the subspace of
$$H^s(B_R(\mathbf{0}) \times {\mathbb R};X):=[H^{m}(B_R(\mathbf{0}) \times {\mathbb R};X),H^{m+1}(B_R(\mathbf{0}) \times {\mathbb R};X)]_\theta$$
consisting of mode $k$ functions, such that $H_{(k)}^s((0,R) \times {\mathbb R};X)$ coincides with the space of their radial coefficients.
\item[(iii)]
The space $\hat{\mathscr D}_{(k)}(\overline{B}_R(\mathbf{0}) \times {\mathbb R};X)$ is dense in $\hat{H}_{(k)}^s(B_R(\mathbf{0}) \times {\mathbb R};X)$ because
$\hat{H}_{(k)}^m(B_R(\mathbf{0}) \times {\mathbb R};X)$ is dense in $\hat{H}_{(k)}^s(B_R(\mathbf{0}) \times {\mathbb R};X)$
(see Lions \& Magenes \cite[vol.\ I, ch.\ 1, Remark 8.2]{LionsMagenes} and Lemma \ref{lem:density}(ii)), such that ${\mathscr D}_{(k)}([0,R] \times {\mathbb R};X)$ is dense in $H_{(k)}^s(B_R(\mathbf{0}) \times {\mathbb R};X)$.
\end{itemize}
\end{remarks}

Finally, we state the radial version of the Sobolev embedding theorem, now distinguishing between functions with and without
dependence on the axial coordinate.

\begin{lemma} \label{lem:CSobolev}
$ $
\begin{itemize}
\item[(i)]
Suppose that $n \in {\mathbb N}_0$ and $s > n+1$. The radial Sobolev space
$H_{(k)}^s((0,R);X)$ is a Banach algebra and is continuously embedded in $C_{(k)\mathrm{b}}^n([0,R];X)$.
\item[(ii)]
Suppose that $n \in {\mathbb N}_0$ and $s > n+\frac{3}{2}$. The radial Sobolev space
$H_{(k)}^s((0,R)\times {\mathbb R};X)$ is a Banach algebra and is continuously embedded in $C_{(k)\mathrm{b}}^n([0,R] \times {\mathbb R};X)$.
\end{itemize}
\end{lemma}
\begin{proof}
Part (i) follows from the facts that $\hat{H}_{(k)}^s(B_R(\mathbf{0});X)$ is a Banach algebra which is continuously embedded in $\hat{C}_{(k)\mathrm{b}}^n(\overline{B}_R(\mathbf{0});X)$ if $s>n+1$,
and that the spaces $\hat{H}_{(k)}^s(B_R(\mathbf{0});X)$, $H_{(k)}^s((0,R);X)$ and
$\hat{C}_{(k)\mathrm{b}}^n(\overline{B}_R(\mathbf{0});X)$, $C_{(k)\mathrm{b}}^n([0,R];X)$ are isometrically isomorphic. Part (ii) is obtained by the same argument.
\end{proof}

\begin{remark} \label{rem:betterouter}
Noting that $H_{(k)}^s((0,R);X) \subseteq H^s((\delta,R);X)$ and $H_{(k)}^s((0,R)\times {\mathbb R};X) \subseteq H^s((\delta,R)\times {\mathbb R};X)$ for each $\delta \in (0,R)$,
we find that $H_{(k)}^s((0,R);X)$ is continuously embedded in $C_\mathrm{b}^n([\delta,R];X)$ for $s>n+\frac{1}{2}$ and
$H_{(k)}^s((0,R)\times {\mathbb R};X)$ is continuously embedded in $C_\mathrm{b}^n([\delta,R] \times {\mathbb R};X)$ for $s>n+1$.
\end{remark}

The argument given in Remark \ref{rem:betterouter} also leads to a trace theorem for $H_{(k)}^s((0,R) \times {\mathbb R};X)$.

\begin{lemma} \label{lem:trace}
Suppose that $s>\frac{1}{2}$. The mapping $f_k \mapsto f_k|_{r=R}$ defined on ${\mathscr D}_{(k)}([0,R] \times {\mathbb R};X)$ extends to a continuous linear mapping $H_{(k)}^s((0,R) \times {\mathbb R};X)
\rightarrow H^{s-\frac{1}{2}}({\mathbb R};X)$ with continuous right inverse.
\end{lemma}

We now work with the alternative definition of $\hat{H}_{(k)}^m(B_R(\mathbf{0}) \times \mathbb{R};X)$ as the subspace
$$\{\hat{f}_k \in L^2(B_R(\mathbf{0}) \times \mathbb{R};X)\st\partial_{\bar{\zeta}}^{n-i}\partial_\zeta^i \partial_z^{p-n} \hat{f}_k \in L^2(B_R(\mathbf{0}) \times \mathbb{R};X),\ 0\leq i\leq n\leq p \leq m\}$$
of $H^m(B_R(\mathbf{0}) \times \mathbb{R};X)$ consisting of mode $k$ functions (where $\partial_{\bar{\zeta}}^{n-i}\partial_\zeta^i \partial_z^{p-n} \hat{f}_k$ now denotes a distributional derivative of $\hat{f}_k$),
equipping it with the norm \eqref{Hat norm}.
In this context we refer to $\partial_{\bar{\zeta}}^{n-i}\partial_\zeta^i \partial_z^{p-n} \hat{f}_k$ as a \emph{weak derivative}.
In view of Lemma \ref{lem:distderiv construction} and Definition
\ref{Defn of dist rad deriv} we define
\begin{align*}
H_{(k)}^m((0,R) \times \mathbb{R};X)
&:=\{f_k \in L_1^2((0,R)\times\mathbb{R};X)\st {\mathcal D}_{-k+i}^{n-i}{\mathcal D}_k^i \partial_z^{p-n} f_k \in L_1^2((0,R)\times\mathbb{R};X),\ 0 \leq i \leq n \leq p \leq m\},
\end{align*}
(where ${\mathcal D}_{-k+i}^{n-i}{\mathcal D}_k^i \partial_z^{p-n} f_k$ now denotes a distributional radial derivative); in this context we
refer to\linebreak
${\mathcal D}_{-k+i}^{n-i}{\mathcal D}_k^i \partial_z^{p-n} f_k$ as a \emph{weak radial derivative}.
Observing that a mode $k$ function $\hat{f}_k$ belongs to\linebreak $\hat{H}_{(k)}^m(B_R(\mathbf{0}) \times \mathbb{R};X)$ if and only if its radial coefficient $f_k$ belongs to $H_{(k)}^m((0,R) \times \mathbb{R};X)$,
one finds that
$H_{(k)}^m((0,R) \times \mathbb{R};X)$ is a Banach space with respect to the norm \eqref{No hat norm}
because $f_k \mapsto \hat{f}_k$ is an isometric isomorphism $H_{(k)}^m((0,R) \times \mathbb{R};X) \rightarrow \hat{H}_{(k)}^m(B_R(\mathbf{0}) \times \mathbb{R};X)$.

Using Lemma \ref{lem:distderiv construction} we can directly characterise the weak derivative of a mode $k$ function and the
weak radial derivative of its radial coefficient.

\begin{lemma}
Suppose that $\hat{f}_k \in L^2(B_R(\mathbf{0}) \times \mathbb{R};X)$. The function 
$\hat{g}_{k+n-2i}\in L^2(B_R(\mathbf{0}) \times \mathbb{R};X)$
is the weak derivative
$\partial_{\bar{\zeta}}^{n-i}\partial_\zeta^i \partial_z^{p-n}\hat{f}_k$ of $\hat{f}_k$ if and only if
\begin{equation}
(-1)^p
\hspace{-4mm}\int\limits_{B_R(\mathbf{0})\times {\mathbb R}} \hspace{-4mm}\hat{f}_k \partial_{\zeta}^i \partial_{\bar{\zeta}}^{n-i} \partial_z^{p-n}\hat{\phi}_{-k-n+2i}
=
\hspace{-4mm} \int\limits_{B_R(\mathbf{0})\times {\mathbb R}} \hspace{-4mm} \hat{g}_{k+n-2i}\hat{\phi}_{-k-n+2i}
\label{eq:wkderiv 1}
\end{equation}
for all $\hat{\phi}_{-k-n+2i} \in \hat{\mathscr D}_{(-k-n+2i)}(B_R(\mathbf{0}) \times \mathbb{R};X)$, and this condition holds if and only if
\begin{align}
(-1)^p
 \int_{\mathbb R} \int_0^R & f_k(r,z)2^{-\frac{n}{2}}{\mathcal D}_{k+n-i}^{n-i} {\mathcal D}_{-k-n+2i}^i \partial_z^{p-n}\phi_{-k-n+2i}(r,z)r\dr\dz \nonumber\\
&= \int_{\mathbb R} \int_0^R g_{k+n-2i}(r,z)\phi_{-k-n+2i}(r,z) r\dr\dz
\label{eq:wkderiv 2}
\end{align}
for all $\phi_{-k-n+2i} \in {\mathscr D}_{(-k-n+2i)}([0,R)\times\mathbb{R};X)$.
\end{lemma}

The following result confirms that the present constructions of $\hat{H}_{(k)}^m(B_R(\mathbf{0}) \times \mathbb{R};X)$
and $H_{(k)}^m((0,R) \times \mathbb{R};X)$\linebreak by weak derivatives coincide with the spaces obtained as the completions of
$\hat{\mathscr D}_{(k)}(\overline{B}_R(\mathbf{0}) \times {\mathbb R};X)$
and \linebreak ${\mathscr D}_{(k)}([0,R] \times {\mathbb R};X)$.

\begin{lemma} \label{lem:density}
Suppose that $\hat{f}_k \in \hat{H}_{(k)}^m(B_R(\mathbf{0})\times {\mathbb R};X)$ (such that $f_k \in H_{(k)}^m((0,R)\times {\mathbb R};X)$).
There exists\linebreak a sequence $\{\hat{\phi}_k^j\} \subseteq \hat{\mathscr D}_{(k)}(\overline{B}_R(\mathbf{0}) \times {\mathbb R};X)$ such that
$\hat{\phi}_k^j \rightarrow \hat{f}_k$ in $\hat{H}_{(k)}^m(B_R(\mathbf{0})\times {\mathbb R};X)$ and $\phi_k^j \rightarrow f_k$ in\linebreak
$H_{(k)}^m((0,R) \times {\mathbb R};X)$ as $j \rightarrow \infty$.
\end{lemma}
\begin{proof}
Since ${\mathscr D}(\overline{B}_R(\mathbf{0}) \times {\mathbb R};X)$ is dense in $H^m(B_R(\mathbf{0})\times{\mathbb R};X)$ there exists a sequence
$\{\hat{\phi}^j\} \subseteq {\mathscr D}(\overline{B}_R(\mathbf{0}) \times {\mathbb R};X)$ with
$\hat{\phi}^j \rightarrow \hat{f}_k$ in $H^m(B_R(\mathbf{0})\times{\mathbb R};X)$ as $j \rightarrow \infty$, and by
Corollary \ref{cor:FS superconvergence} we can write each of these
functions as a sum $\hat{\phi}^j = \sum\limits_{\ell \in {\mathbb Z}} \hat{\phi}_k^j$ of mode $\ell$ functions $\hat{\phi}_\ell^j \in \hat{\mathscr D}_{(\ell)}(\overline{B}_R(\mathbf{0}) \times {\mathbb R};X)$.
Observe that $\phi_k^j-f_k = {\mathcal P}_k[\hat{\phi}^j-\hat{f}_k]$, where
$${\mathcal P}_k[\hat{\psi}](r,z):=\frac{1}{2\pi}\int_0^{2\pi} \hat{\psi}(r\cos\theta,r\sin\theta,z)\e^{-\i k \theta}\dtheta$$
(interchanging the sum and integral is justified due to the uniform convergence of the series for \linebreak each fixed $(r,z)$).
Furthermore ${\mathcal P}_k$ extends from a mapping ${\mathscr D}(\overline{B}_R(\mathbf{0}) \times {\mathbb R};X) \rightarrow {\mathscr D}_{(k)}([0,R] \times {\mathbb R}; X)$
(see\linebreak Remark \ref{propprojrig}) to a continuous mapping $H^m(B_R(\mathbf{0})\times{\mathbb R};X) \rightarrow H_{(k)}^m([0,R) \times {\mathbb R};X)$. It follows that\linebreak
$\|\phi_k^j-f_k\|_{H_{(k)}^m} = \|{\mathcal P}_k[\hat{\phi}^j-\hat{f}_k]\|_{H_{(k)}^m}  \rightarrow 0$ and hence $\|\hat{\phi}^j-\hat{f}_k]\|_{H^m} \rightarrow 0$ as $j \rightarrow \infty$.
\end{proof}

\subsection{Hankel spaces} \label{Hankel FS}
In this section we introduce a scale of Hilbert spaces $B_{(k)}^s((0,\infty);X)$, $s \geq 0$
for the radial coefficients of mode $k$ functions. These \emph{Hankel spaces}
are the radial counterparts of the familiar Bessel-potential spaces constructed using the Fourier transform,
which is here replaced by the following radial transform.

\begin{definition}
The \underline{$k$-index Hankel transform} of a function $f \in {\mathscr S}_{(k)}([0,\infty);X)$ is defined by
the integral
\begin{equation}
\mathcal{H}_{k}[f](\rho) := \int_{0}^{\infty} f(r) J_{k}(\rho \,r) \,r\dr. \label{eq:Hankdefn}
\end{equation}
\end{definition}

We begin by establishing some basic properties of the Hankel transform; for notational clarity
we denote Bessel operators with respect to $r$ and $\rho$ by respectively ${\mathcal D}_k^i$ and $\tilde{\mathcal D}_k^i$.

\begin{lemma} \label{HT properties}
The Hankel transform $\mathcal{H}_k$ maps ${\mathscr S}_{(k)}([0,\infty);X)$ into itself, and furthermore the identities
\begin{align*}
\tilde{\mathcal D}_{-k+i}^{n-i}\tilde{\mathcal D}_k^i {\mathcal H}_k[f](\rho) & = (-1)^{n-i} {\mathcal H}_{k+n-2i}[r^nf](\rho), \\
{\mathcal H}_{k+n-2i} [{\mathcal D}_{-k+i}^{n-i}{\mathcal D}_k^i f](\rho) & = (-1)^{n-i}\rho^n {\mathcal H}_k[f](\rho)
\end{align*}
hold for each $f \in {\mathscr S}_{(k)}([0,\infty);X)$ and $0 \leq i \leq n$, $n \in {\mathbb N}_0$.
\end{lemma}
\begin{proof}
First observe that $r^nf \in {\mathscr S}_{(k+n-2i)}([0,\infty);X)$ for all $0\leq i\leq n$ by Proposition \ref{prop:products}(ii) and Lemma \ref{lem:power}(ii).
Noting that formally
\begin{align*}
\tilde{\mathcal D}_{-k+i}^{n-i}\tilde{\mathcal D}_k^i {\mathcal H}_k[f](\rho)
& = \int_0^\infty f(r) \tilde{\mathcal D}_{-k+i}^{n-i}\tilde{\mathcal D}_k^i  J_k(\rho r) \,r\dr \\
& = \int_0^\infty r^{i} f(r) \tilde{\mathcal D}_{-k+i}^{n-i} J_{k-i}(\rho r) \,r\dr \\
& = (-1)^{k-i} \int_0^\infty r^{n}f(r) J_{-k-n+2i}(\rho r) \,r\dr \\
& = (-1)^{n-i} {\mathcal H}_{k+n-2i}[r^nf](\rho),
\end{align*}
where we have used Remark \ref{rem:Lap}(ii), and using a familiar dominated convergence argument with dominating function $rf$ establishes the existence of the derivatives
$\tilde{\mathcal D}_{-k+i}^{n-i}\tilde{\mathcal D}_k^i {\mathcal H}_k[f]$ for $0 \leq i \leq n$, $n \in {\mathbb N}_0$ and validity of the first identity.

The second identity follows from the calculation
\begin{align*}
{\mathcal H}_{k+n-2i} [{\mathcal D}_{-k+i}^{n-i}{\mathcal D}_k^i f](\rho)
& = \int_0^\infty {\mathcal D}_{-k+i}^{n-i}{\mathcal D}_k^i f(r)J_{k+n-2i}(\rho r) \,r \dr \\
& = (-1)^n \int_0^\infty f(r) {\mathcal D}_{k+n-i}^{n-i}{\mathcal D}_{-k-n+2i}^i J_{k+n-2i}(\rho r)  \,r\dr \\
& = (-1)^k \int_0^\infty \rho^i f(r) {\mathcal D}_{k+n-i}^{n-i}J_{-k-n+i}(\rho r)  \,r\dr \\
& = (-1)^{n-i}\int_0^\infty  \rho^n f(r) J_k(\rho r) \,r\dr \\
& = (-1)^{n-i}\rho^n {\mathcal H}_k[f](\rho),
\end{align*}
where we have integrated by parts. Combining the two identities yields
\begin{align*}
\rho^m \tilde{\mathcal D}_{-k+i}^{n-i}\tilde{\mathcal D}_k^i {\mathcal H}_k[f](\rho)
& = (-1)^{n-i}\rho^m {\mathcal H}_{k+n-2i}[r^nf](\rho)\\
& = (-1)^{n+m-i-j}{\mathcal H}_{k+n+m-2i-2j}[{\mathcal D}_{-k-n+2i+j}^{m-j}{\mathcal D}_{k+n-2i}^j(r^nf)](\rho)
\end{align*}
for any $j$ with $0 \leq j \leq m$ and $m \in {\mathbb N}_0$. We conclude that
$$\sup_{\rho\geq0} \rho^m |\tilde{\mathcal D}_{-k+i}^{n-i}\tilde{\mathcal D}_k^i {\mathcal H}_k[f](\rho)| < \infty,$$
and furthermore
$$(k+n-2i)\tilde{\mathcal D}_{-k+i}^{n-i}\tilde{\mathcal D}_k^i {\mathcal H}_k[f](0)
=(-1)^{n-i}(k+n-2i)J_{k+n-2i}(0)\int_0^\infty r^{n}f(r)\,r\dr =0$$
because $J_{k+n-2i}(0)=0$ if $k+n-2i \neq 0$.
\end{proof}

\begin{proposition}
The Hankel transform $\mathcal{H}_k: {\mathscr S}_{(k)}([0,\infty);X) \rightarrow {\mathscr S}_{(k)}([0,\infty);X)$ is an involution and isometric with respect to the $L_1^2((0,\infty);X)$ norm.
\end{proposition}

\begin{proof}
Hankel's theorem (see Titchmarsh \cite[\S8.18]{Titchmarsha}) states that
$$\int_0^\infty J_k(\rho r) \sqrt{\rho r} \int_0^\infty J_k(\rho s)\sqrt{\rho s} g(s) \ds\drho=\frac{1}{2}
\left( \lim_{\delta \uparrow 0} g(r+\delta) +  \lim_{\delta \downarrow 0} g(r+\delta) \right)$$
for $g \in L^1((0,\infty);X) \cap \mathrm{BV}_\mathrm{loc}([0,\infty);X)$. 
Writing $g(r)=\sqrt{r}f(r)$, we find that
$$
{\mathcal H}_k[{\mathcal H}_kf](r) =
\int_0^\infty  J_k(\rho r) \int_0^\infty f(s) J_k(\rho s)\,s \ds \,\rho\drho
= f(r)
$$
for $f \in {\mathscr S}_{(k)}([0,\infty);X)$, such that ${\mathcal H}_k^{-1}={\mathcal H}_k$. The fact that $\mathcal{H}_k: {\mathscr S}_{(k)}([0,\infty);X) \rightarrow {\mathscr S}_{(k)}([0,\infty);X)$
is isometric with respect to the $L_1^2((0,\infty);X)$ norm follows from the calculation
\begin{align*}
\int_0^\infty f_1(r) f_2(r) \,r\dr
& = \int_0^\infty f_1(r) \int_0^\infty {\mathcal H}_k[f_2](\rho) J_k(\rho r)\,\rho \drho\; \,r\dr \\
& = \int_0^\infty \int_0^\infty f_1(r)J_k(\rho r) \,r\dr\; {\mathcal H}_k[f_2](\rho) \,\rho\drho \\
& = \int_0^ \infty {\mathcal H}_k[f_1](\rho) {\mathcal H}_k[f_2](\rho) \,\rho\drho
\end{align*}
for $f_1, f_2 \in {\mathscr S}_{(k)}([0,\infty);X)$, where the inversion of the repeated integral is justified since\linebreak $(r,\rho) \mapsto r\rho f_1(r){\mathcal H}_k[f_2](\rho)$ belongs to $L^1((0,\infty);X)^2$.
\end{proof}

\begin{corollary} \label{HT iso}
The Hankel transform $\mathcal{H}_k: {\mathscr S}_{(k)}([0,\infty);X) \rightarrow {\mathscr S}_{(k)}([0,\infty);X)$ extends to an isometric
isomorphism $L_1^2((0,\infty);X) \rightarrow L_1^2((0,\infty);X)$.
\end{corollary}

\begin{remark}
A useful consequence of the previous corollary is that
$$\|\hat{f}_k\|_{H^m}^2= \sum_{p=0}^m \sum_{n=0}^{p}
2^{n}\|\partial_{\bar{\zeta}}^{n-i_n} \partial_\zeta^{i_n} \partial_z^{p-n} \hat{f}_k\|_{L^2}^2, \qquad
\|f_k\|_{H_{(k)}^m}^2= \sum_{p=0}^m \sum_{n=0}^{p}
\| \mathcal{D}_{-k+i_n}^{n-i_n}\mathcal{D}_{k}^{i_n} \partial_z^{p-n} f_k\|_{L_1^2}^2
$$
for $\hat{f}_k \in \hat{H}_{(k)0}^m(B_R(\mathbf{0}) \times {\mathbb R};X)$ 
and any choice of $i_n$ with $0 \leq i_n \leq n$.
It suffices to prove this result for $\hat{f}_k \in \hat{\mathscr D}_{(k)}(B_R(\mathbf{0}) \times {\mathbb R};X)
\subseteq \hat{\mathscr S}_{(k)}({\mathbb R}^2 \times {\mathbb R};X)$, such that $\partial_z^{p-n}\hat{f}_k \in \hat{\mathscr S}_{(k)}({\mathbb R}^2;X)$ for each fixed $z \in {\mathbb R}$. Using Corollary \ref{HT iso} and Lemma \ref{HT properties} and integrating over $z$, we find that
$$\|{\mathcal D}_{-k+i}^{n-i}{\mathcal D}_k^i \partial_z^{p-n}f_k\|_{L_1^2}^2=\|{\mathcal H}_{k+n-2i} [{\mathcal D}_{-k+i}^{n-i}{\mathcal D}_k^i\partial_z^{p-n} f_k]\|_{L_1^2}^2=\|\rho^n {\mathcal H}_k[\partial_z^{p-n}f_k]\|_{L_1^2}^2$$
for each $i$ with $0 \leq i \leq n$. The assertion now follows from the definitions \eqref{Hat norm}, \eqref{No hat norm} of the norms and the fact that
$\|{\mathcal D}_{-k+i}^{n-i}{\mathcal D}_k^i \partial_z^{p-n}f_k\|_{L_1^2}$ and hence $\|\partial_{\bar{\zeta}}^{n-i} \partial_\zeta^{i} \partial_z^{p-n} \hat{f}_k\|_{L^2}$ is independent of $i$.
\end{remark}

We now define the Hankel spaces and state their main properties; although Lemmata \ref{lem:CHankel1} and \ref{lem:CHankel2} follow
from Lemma \ref{lem:basic Hankel properties}(iii) and Lemma \ref{lem:CSobolev} we here give an instructive direct proof.

\begin{definition} $ $ \label{Hankel space definition}
\begin{itemize}
\item[(i)]
The Hankel transform of a tempered radial distribution $T_k \in {\mathscr S}_{(k)}^\prime((0,\infty);X)$ is the tempered radial distribution ${\mathcal H}_k[T_k] \in {\mathscr S}_{(k)}^\prime((0,\infty);X)$ defined by
$$({\mathcal H}_k[T_k],\phi_k) = (T_k, \mathcal{H}_k[\phi_k]), \qquad \phi_k \in {\mathscr S}_{(k)}((0,\infty);X).$$
\item[(ii)]
For $s \geq 0$ the Hankel space $B_{(k)}^s((0,\infty);X)$ is defined by the formula
$$B_{(k)}^s((0,\infty);X) := \{f \in L_1^2((0,\infty);X)\st(1+\rho^2)^{\frac{s}{2}}{\mathcal H}_k[f] \in L_1^2((0,\infty);X)\},$$
where $f$ and ${\mathcal H}_k[f]$ can be understood as tempered radial distributions or as functions in $L_1^2((0,\infty);X)$.
\end{itemize}
\end{definition}

\begin{lemma} \label{lem:basic Hankel properties}$ $
\begin{itemize}
\item[(i)]
The Hankel space $B_{(k)}^s((0,\infty);X)$ is a Banach space with respect to the norm
$$\|f\|_{B_{(k)}^s} = \left( \int_0^\infty (1+\rho^2)^s \|{\mathcal H}_k[f](\rho)\|^2\,\rho\drho\right)^{\!\!1/2}.$$
\item[(ii)]
The Schwartz space ${\mathscr S}_{(k)}([0,\infty);X)$ is dense in $B_{(k)}^s((0,\infty);X)$.
\item[(iii)]
The Hankel space $B_{(k)}^s((0,\infty);X)$ coincides with the radial Sobolev space $H_{(k)}^s((0,\infty);X)$ (see Definition \ref{fractional Sob}).
\end{itemize}
\end{lemma}

\begin{proof} $ $
Parts (i) and (ii) follow from the fact that 
$g \mapsto {\mathcal H}_k[(1+(\cdot)^2)^{-\frac{s}{2}}g]$ is an isometric isomorphism
$L_1^2((0,\infty);X) \rightarrow B_{(k)}^s((0,\infty);X)$ which maps the dense subset 
${\mathscr S}_{(k)}([0,\infty);X)$ of $L_1^2((0,\infty);X)$ onto itself.

Turning to part (iii), we first show that $B_{(k)}^m((0,\infty);X))=H_{(k)}^m((0,\infty);X)$ for $m \in {\mathbb N}_0$
by demonstrating that
$\|f\|_{H_{(k)}^m} \cong \|f\|_{B_{(k)}^m}$ for $f \in {\mathscr S}_{(k)}([0,\infty);X)$. To this end we note that
\begin{align*}
\sum_{n=0}^m 2^{-n} \sum_{i=0}^n \begin{pmatrix} n \\ i \end{pmatrix} \|{\mathcal D}_{-k+i}^{n-i}{\mathcal D}_k^i f\|_{L_1^2}^2
& = \sum_{n=0}^m 2^{-n} \sum_{i=0}^n \begin{pmatrix} n \\ i \end{pmatrix} \|{\mathcal H}_{k+n-2i}[{\mathcal D}_{-k+i}^{n-i}{\mathcal D}_k^i f]\|_{L_1^2}^2 \\
& = \sum_{n=0}^m 2^{-n} \sum_{i=0}^n \begin{pmatrix} n \\ i \end{pmatrix} \|\rho^n {\mathcal H}_k[f]\|_{L_1^2}^2 \\
& = \sum_{n=0}^m  \|\rho^n {\mathcal H}_k[f]\|_{L_1^2}^2 \\
& \cong \|(1+\rho^2)^{m/2}{\mathcal H}_k[f]\|_{L_1^2}^2.
\end{align*}
Now write $s=m+\theta$ with $m=\lfloor s\rfloor$ and let $X_s$ be the Banach lattice of (equivalence classes of) Lebesgue-measurable functions
$f: (0,\infty) \rightarrow {\mathbb R}$ such that $(1+|\cdot|)^{s/2} f \in L_1^2((0,\infty); {\mathbb R})$.
Observe that $X_s$ has the dominated convergence property for all $s \geq 0$, such that
$$[X_m(X),X_{m+1}(X)]_\theta = X_m^{1-\theta}X_{m+1}^\theta(X)$$
(see Calder\'{o}n \cite[p.\ 125]{Calderon64}). Evidently $X_m^{1-\theta}X_{m+1}^\theta = X_s$ and
since ${\mathcal H}_k$ is an isometric isomorphism\linebreak $B_{(k)}^s((0,\infty); X) \rightarrow X_s(X)$ we conclude that
\begin{align*}
H_{(k)}^s((0,\infty);X)&=[H_{(k)}^m((0,\infty);X),H_{(k)}^{m+1}((0,\infty);X)]_\theta \\
& = [B_{(k)}^m((0,\infty);X),B_{(k)}^{m+1}((0,\infty);X)]_\theta \\
& = B_{(k)}^s((0,\infty);X).\qedhere
\end{align*}

\end{proof}

\begin{lemma} \label{lem:CHankel1}
The Hankel space $B_{(k)}^s((0,\infty);X)$ is continuously embedded in $C_{(k)\mathrm{b}}^0([0,\infty);X)$ for $s>1$.
\end{lemma}
\begin{proof}
Note that
\begin{equation}
f(r) = \int_0^\infty {\mathcal H}_k[f](\rho)J_k(\rho r)\,\rho\drho \label{eq:rep of f}
\end{equation}
for $f \in B_{(k)}^s((0,\infty);X)$. It follows from this formula that
\begin{align*}
\|f\|_\infty & = \sup_{r \geq 0} \left\| \int_0^\infty {\mathcal H}_k[f](\rho)J_k(\rho r)\,\rho\drho\right\| \\
& \leq \int_0^\infty \|{\mathcal H}_k[f](\rho)\|\,\rho\drho\\
& = \int_0^\infty \rho^{1/2}(1+\rho^2)^{-s/2}\rho^{1/2}(1+\rho^2)^{\frac{s}{2}}\|{\mathcal H}_k[f](\rho)\|\drho\\
& \leq \left( \int_0^ \infty (1+\rho^2)^{-s}\,\rho\drho\right)^{\!\!1/2} \left(\int_0^\infty (1+\rho^2)^s \|{\mathcal H}_k[f](\rho)\|^2\,\rho\drho\right)^{\!\!1/2} \\
& \lesssim \|f\|_{B_{(k)}^s}
\end{align*}
because $\rho(1+\rho^2)^{-s} \in L^1((0,\infty);X)$ for $s>1$. This calculation shows
that $B_{(k)}^s((0,\infty);X)$ is continuously embedded in $L^\infty((0,\infty);X)$ and also that  $\rho {\mathcal H}_k[f] \in L^1((0,\infty);X)$.

The continuity of $f$ follows by applying a familiar dominated convergence argument to the formula \eqref{eq:rep of f}
with $\rho {\mathcal H}_{k}[f]$ as dominating function; this formula also shows that $f(0)=0$ if $k \neq 0$.
\end{proof}

\begin{lemma} \label{lem:CHankel2}
Suppose that $m \in {\mathbb N}_0$. The Hankel space $B_{(k)}^s((0,\infty);X)$ is continuously embedded in\linebreak
$C_{(k)\mathrm{b}}^m([0,\infty);X)$ for $s>m+1$.
\end{lemma}
\begin{proof}
Take $f \in B_{(k)}^s((0,\infty);X)$ and consider the functions
$$g_{k,n,i}(r) = (-1)^{n-i} \int_0^\infty \rho^{n} {\mathcal H}_k[f](\rho)J_{k+n-2i}(\rho r)\,\rho\drho$$
for $0 \leq i \leq n \leq m$. Arguing as in the previous lemma, we find that
$$
\|g_{k,n,i}\|_\infty \leq \int_0^\infty \!\rho^{n}\|{\mathcal H}_k[f](\rho)\|\,\rho\drho
\leq \left( \int_0^ \infty \!\rho^{2n}(1+\rho^2)^{-s}\,\rho\drho\right)^{\!\!1/2}\!\! \left(\int_0^\infty \!(1+\rho^2)^s \|{\mathcal H}_k[f](\rho)\|^2\,\rho\drho\right)^{\!\!1/2}\!\!\lesssim \|f\|_{B_k^n}
$$
if $s>n+1$ because $\rho^{2n+1}(1+\rho^2)^{-s} \in L^1((0,\infty);X)$ if $s>n+1$.

Applying the derivative ${\mathcal D}_{-k+i}^{n-i}{\mathcal D}_k^i$, $0 \leq i \leq n \leq m$, systematically to the formula
$$f(r) = \int_0^\infty {\mathcal H}_k[f](\rho)J_k(\rho r)\,\rho\drho,$$
noting that formally
$${\mathcal D}_{-k+i}^{n-i}{\mathcal D}_k^i f(r) = (-1)^{n-i}\int_0^\infty \rho^{n}{\mathcal H}_k[f](\rho)J_{k+n-2i}(\rho r)\,\rho\drho$$
and using a familiar dominated convergence argument with dominating function $\rho^{n+1}{\mathcal H}_k[f]$ shows that
${\mathcal D}_{-k+i}^{n-i}{\mathcal D}_k^i f$ exists and equals $g_{k,n,i}$. The above formula also shows that
${\mathcal D}_{-k+i}^{n-i}{\mathcal D}_k^i f(0)=0$ if $k+n-2i \neq 0$.

Altogether we conclude that $f \in C_{(k)\mathrm{b}}^n([0,\infty);X)$ and
$$\|f\|_{C_{(k)\mathrm{b}}^n} \lesssim \|f\|_{B_{(k)}^s}, \qquad 0\leq n\leq m.\eqno{\qedhere}$$
\end{proof}

Finally, we examine the multiplication operator $(f_k,g_\ell) \mapsto f_k g_\ell$, beginning with the observation in Lemma \ref{multconv} below that the $(k+\ell)$ index Hankel transform
of $f_kg_\ell$ is a radial convolution of the $k$ index Hankel transform of $f_k$ and the $\ell$ index Hankel transform of $g_\ell$.
\begin{definition}
Suppose that $X$ is a Banach algebra.
The \underline{$(k,\ell)$-index radial convolution} $(f_k \ast g_\ell)$ of a mode $k$ radial coefficient $f_k:[0,\infty) \rightarrow X$ and a mode $\ell$ radial coefficient $g_\ell:[0,\infty) \rightarrow X$ is
the mode $(k+\ell)$ radial cofficient $f_k\ast g_\ell: [0,\infty) \rightarrow X$ given by
$$(f_k\ast g_\ell)(r) = \int_0^\infty f_k(u)T_{k,\ell}[g_\ell](r,u)u\du,$$
where
$$T_{k,\ell}[g_\ell](r,u) = 2\pi\int_0^\infty \int_0^\infty J_k( u t) J_\ell(wt) J_{k+\ell}(r t)t\dt\ g_\ell(w) w\dw.$$
\end{definition}

By definition
$$f_k \ast g_\ell = 2\pi\,{\mathcal H}_{k+\ell}[{\mathcal H}_k[f_k] {\mathcal H}_\ell[g_\ell]],$$
so that in particular $f_k \in {\mathscr S}_{(k)}( [0,\infty);X)$, $g_\ell \in {\mathscr S}_{(\ell)}( [0,\infty);X)$ implies that
$f_k \ast g_\ell \in {\mathscr S}_{(k+\ell)}( [0,\infty);X)$. The following result is obtained by noting that Hankel transforms are involutions.

\begin{lemma} \label{multconv}
Suppose that $X$ is a Banach algebra and $f_k \in {\mathscr S}_{(k)}( [0,\infty);X)$, $g_\ell \in {\mathscr S}_{(\ell)}( [0,\infty);X)$. It follows that
${\mathcal H}_{k+\ell}[f_k g_\ell] = 2\pi\,{\mathcal H}_k[f_k] \ast {\mathcal H}_\ell[g_\ell]$.
\end{lemma}

\begin{remarks} $ $
\begin{itemize}
\item[(i)]
Note the alternative representations
$$
T_{k,\ell}[g_\ell](r,u) = \int_0^{2\pi}   g_\ell(\sqrt{r^2 + u^2 -2ru \cos \phi})\Phi_{k,\ell}(r,u,\phi)\,\mathrm{d}\phi
$$
where
$$
\Phi_{k,\ell}(r,u,\phi) = \exp \i\left( k \phi - \ell \sin^{-1}\left(\frac{u \sin \phi}{\sqrt{r^2+u^2-2ur\cos\phi}}\right)\right)
$$
(see Watson \cite[p.\ 360]{Watson}) and
$$
T_{k,\ell}[g_\ell](r,u) = 2\pi\int_0^\infty \cos(k\phi_w-\ell\phi_r)D(r,u,w)g_\ell(w)w\dw,
$$
where $\phi_r$, $\phi_u$, $\phi_w$ are the angles of a triangle with sides of length $r$, $u$, $w$ and
$D(r,u,w) = (\pi uw \sin \phi_r)^{-1}$ if this triangle exists and zero otherwise
(see Jackson \& Maximon \cite{JacksonMaximon72} and Gervois \& Navelet \cite{GervoisNavelet84}).
\item[(ii)]
Radial convolutions satisfy an appropriate version of Young's inequality which is given in the Appendix.
\end{itemize}
\end{remarks}
%

\begin{lemma}
Suppose that $X$ is a Banach algebra and that $s, s_1, s_2 \geq 0$, $s_1, s_2 \geq s$ and $s_1+s_2 > s+1$. The inequality
$$\|f_{k} g_{\ell} \|_{B_{(k + \ell)}^s} \lesssim \|f_{k}\|_{B_{(k)}^{s_1}} \|g_{\ell}\|_{B_{(\ell)}^{s_2}}$$
holds for each $f_{k} \in B_{(k)}^{s_1}((0,\infty);X)$ and $g_{\ell} \in B_{(\ell)}^{s_2}((0,\infty);X)$.
\end{lemma}

\begin{proof}
We begin by noting by elementary geometry that
$$D(\rho,u,w) =
\begin{cases}
\dfrac{2}{\pi}\big(u^2-(\rho-w)^2)\big)^{-1/2}\big((\rho+w)^2-u^2\big)^{-1/2}, & |\rho-w|<u<\rho+w, \\[2mm]
0, & \mbox{otherwise,}
\end{cases}$$
and we calculate
\begin{equation}
\int_0^\infty D(\rho,u,w)u \du = \frac{2}{\pi}\int_{|\rho-w|}^{\rho+w} \big(u^2-(\rho-w)^2)\big)^{-1/2}\big((\rho+w)^2-u^2\big)^{-1/2} u \du =1
\label{eq:unit integrals1}
\end{equation}
and
\begin{equation}
\int_0^\infty D(\rho,u,w)w \dw = \int_0^\infty D(\rho,u,w)\rho \drho = 1 \label{eq:unit integrals2}
\end{equation}
because $2\pi D(\rho,u,w)$ is the reciprocal of the area of a triangle of sides with length $\rho$, $u$ and $w$ and is therefore 
symmetric in its arguments.

Observe that
\begin{align*}
\|  f_{k} g_{\ell}\|_{B_{(k+\ell)}^s}^2
& = \int_0^\infty (1+\rho^2)^s  \|(\mathcal{H}_{k}[f_k] \ast \mathcal{H}_{\ell}[g_\ell])(\rho)\|^2\rho\drho \\
& =  (2\pi)^2\int_0^\infty (1+\rho^2)^s \left\|\int_0^\infty \int_0^\infty \cos (k \phi_w - \ell \phi_u) D(\rho,u,w)  {\mathcal H}_k[f_{k}](u){\mathcal H}_\ell[g_{\ell}](w)uw\du\dw\right\|^2\rho\drho \\
& \lesssim \int_0^\infty \left(\int_0^\infty \int_0^\infty (1+\rho^2)^{\frac{s}{2}} D(\rho,u,w) \|{\mathcal H}_k[f_{k}](u)\|\|{\mathcal H}_\ell[g_{\ell}](w)\|uw\du\dw\right)^{\!\!2}\rho\drho \\
& \lesssim \int_0^\infty \left(\int_0^\infty \int_0^\infty (1+u^2+w^2)^{\frac{s}{2}} D(\rho,u,w) \|{\mathcal H}_k[f_{k}](u)\|\|{\mathcal H}_\ell[g_{\ell}](w)\|uw\du\dw\right)^{\!\!2}\rho\drho,
\end{align*}
where we have estimated
$$\rho^2 \leq (u+w)^2 \lesssim u^2 + w^2$$
because $D(\rho,u,w)=0$ if $\rho>u+w$. Using the inequality
$$(1+u^2+w^2)^r \leq (1+u^2)^{r+r_1}(1+w^2)^{-r_1}+(1+w^2)^{r+r_2}(1+u^2)^{-r_2}$$
for $r,r_1,r_2 \geq 0$ with $r=\frac{1}{2}s$, $r_1=\frac{1}{2}(s-s_1)$, $r_2=\frac{1}{2}(s-s_2)$, we thus find that
\begin{equation}
\| f_{k} g_{\ell}\|_{B_{(k+\ell)}^s}^2 \leq \int_0^\infty (I_1+I_2)^2 \,\rho\drho \lesssim \int_0^\infty  I_1^2 \,\rho\drho + \int_0^\infty I_2^2 \,\rho\drho,
\label{eq:alg1}
\end{equation}
where
\begin{align*}
I_1&:=\int_0^\infty\int_0^\infty
D(\rho,u,w)(1+u^2)^{\frac{s_1}{2}}\|{\mathcal H}_k[f_{k}](u)\|(1+w^2)^{\frac{1}{2}(s-s_1)}\|{\mathcal H}_\ell[g_{\ell}](w)\|u w
\du\dw, \\
I_2 &:=\int_0^\infty\int_0^\infty
D(\rho,u,w)(1+u^2)^{\frac{1}{2}(s-s_2)}\|{\mathcal H}_k[f_{k}](u)\|(1+w^2)^{\frac{s_2}{2}}\|{\mathcal H}_\ell[g_{\ell}](w)\|u w
\du\dw.
\end{align*}

Using the Cauchy-Schwarz inequality, one finds that
\begin{align*}
I_1&\leq \left(\int_0^\infty\int_0^\infty D(\rho,u,w)(1+u^2)^{s_1}\|{\mathcal H}_k[f_{k}](u)\|^2 (1+w^2)^{\frac{1}{2}(s-s_1)}\|{\mathcal H}_\ell[g_{\ell}](w)\|uw\du\dw\right)^{\!\!1/2} \\
& \qquad\ \times \left(\int_0^\infty\int_0^\infty D(\rho,u,w)(1+w^2)^{\frac{1}{2}(s-s_1)}\|{\mathcal H}_\ell[g_{\ell}](w)\| uw\du\dw\right)^{\!\!1/2},
\end{align*}
such that
\begin{align}
\int_0^\infty\!\! I_1^2\rho\drho & \leq \int_0^\infty \!\!(1+u^2)^{s_1}\|{\mathcal H}_k[f_{k}](u)\|^2 \,u\du
\left(\int_0^\infty \!\!(1+w^2)^{\frac{1}{2}(s-s_1)}\|{\mathcal H}_\ell[g_{\ell}](w)\|\,w\dw\right)^{\!\!2} \nonumber \\
& \leq \int_0^\infty \!\!(1+u^2)^{s_1}\|{\mathcal H}_k[f_{k}](u)\|^2 \,u\du \int_0^\infty \!\!(1+w^2)^{s_2}\|{\mathcal H}_\ell[g_{\ell}](w)\|^2 \,w\dw
\int_0^\infty (1+w^2)^{s-s_1-s_2}\,w\dw \nonumber \\
& \lesssim \|f_{k}\|_{B_{(k)}^{s_1}}^2 \|g_{\ell}\|_{B_{(\ell)}^{s_2}}^2,  \label{eq:alg2}
\end{align}
where we have used \eqref{eq:unit integrals1}, \eqref{eq:unit integrals2} in the first line. A similar calculation shows that
\begin{equation}
\int_0^\infty  I^2_2\,\rho\drho \lesssim \|f_{k}\|_{B_{(k)}^{s_1}}^2 \|g_{\ell}\|_{B_{(\ell)}^{s_2}}^2. \label{eq:alg3}
\end{equation}
The result follows from inequalities \eqref{eq:alg1}--\eqref{eq:alg3}.
\end{proof}

\section{The Dirichlet boundary-value problem}  \label{DVP}

In this section we consider, as an application of the above theory, the linear boundary-value problem
\begin{align}
       \Delta \hat{u}_k ={}& \hat{f}_k , \qquad 0<|(x,y)|<1, \label{bvp:Laplace1} \\
       \hat{u}_k ={}& \hat{g}_k , \qquad |(x,y)|=1, \label{bvp:Laplace2}
\end{align}
for the mode $k$ function $\hat{u}_k=\e^{\i k \theta}u_k(r,z)$ in the cylindrical domain $B_1({\mathbf 0}) \times {\mathbb R}$,
where $\hat{f}_k=\e^{\i k \theta}f_k(r,z)$ and $\hat{g}_k=\e^{\i k \theta}g_k(z)$ are given mode $k$ functions.
More precisely, we consider the equivalent boundary-value problem
\begin{align}
        (\mathcal{D}_{1-k}\mathcal{D}_{k} + \partial_z^2 )u_k ={}& f_k , \qquad 0<r<1, \label{bvp:rzn1} \\
        u_k ={}& g_k , \qquad r=1, \label{bvp:rzn2}
\end{align}
for $u_k$ in the cylindrical domain $(0,1) \times {\mathbb R}$.
Note that it suffices to consider $k \in {\mathbb N}_0$ (see Remark \ref{rem:radialcoeffs}), which we henceforth assume to be fixed.

\subsection{Weak and strong solutions}
We begin by discussing weak solutions.

\begin{definition}
Suppose that $f_k \in L_1^2((0,1) \times {\mathbb R};{\mathbb C})$ and
$g_k \in H^\frac{1}{2}({\mathbb R}; {\mathbb C})$.
A \underline{weak solution} of \eqref{bvp:rzn1}, \eqref{bvp:rzn2}
is a function $u_k \in H^1_{(k)}((0,1) \times {\mathbb R};{\mathbb C})$ which satisfies
\eqref{bvp:rzn2} in $H^\frac{1}{2}({\mathbb R}; {\mathbb C})$ and
\begin{equation}
\tfrac{1}{2}\langle {\mathcal D}_k{u}_k, {\mathcal D}_k{\phi}_k \rangle_{L_1^2} + \tfrac{1}{2}\langle{\mathcal D}_{-k} {u}_k, {\mathcal D}_{-k}  {\phi}_k \rangle_{L_1^2}
+ \langle \partial_z{u}_k, \partial_z {\phi}_k \rangle_{L_1^2} + \langle {f}_k, {\phi}_k \rangle_{L_1^2}=0
\label{eq:prewkform}
\end{equation}
for all $\phi_k \in H^{1}_{(k)0}((0,1) \times {\mathbb R};{\mathbb C})$.
\end{definition}

\begin{remark} \label{rem:strong soln}
A \underline{strong solution} of \eqref{bvp:rzn1}, \eqref{bvp:rzn2} is a weak solution which lies in $H^2_{(k)}((0,1) \times {\mathbb R};{\mathbb C})$
and hence necessarily satisfies \eqref{bvp:rzn1} in $L_1^2((0,1) \times {\mathbb R};{\mathbb C})$ and \eqref{bvp:rzn2}
in $H^\frac{3}{2}({\mathbb R};{\mathbb C})$.
\end{remark}

\begin{lemma}
For each $f_k \in L_1^2((0,1) \times {\mathbb R};{\mathbb C})$ and
$g_k \in H^\frac{1}{2}({\mathbb R}; {\mathbb C})$ the boundary-value problem
\eqref{bvp:rzn1}, \eqref{bvp:rzn2} admits a unique weak solution $u_k \in H^1_{(k)}((0,1) \times {\mathbb R};{\mathbb C})$,
and this weak solution depends continuously upon $f_k$ and $g_k$.
\end{lemma}
\begin{proof}
Let $h_k \in H^1_{(k)}((0,1) \times {\mathbb R};{\mathbb C})$ be a function with $h_k|_{r=1} = g_k$ in
$H^\frac{1}{2}({\mathbb R}; {\mathbb C})$ and $\|h_k\|_{H_{(k)}^1} \lesssim \|g_k\|_{H^\frac{1}{2}}$
(see Lemma \ref{lem:trace}). We seek
$w_k \in H^1_{(k)0}((0,1) \times {\mathbb R};{\mathbb C})$ such that
\begin{align}
& \tfrac{1}{2}\langle {\mathcal D}_kw_k, {\mathcal D}_k{\phi}_k \rangle_{L_1^2} + \tfrac{1}{2}\langle{\mathcal D}_{-k} w_k, {\mathcal D}_{-k}  {\phi}_k \rangle_{L_1^2}
+ \langle \partial_z w_k, \partial_z {\phi}_k \rangle_{L_1^2} = \nonumber \\
& \qquad \mbox{}-\tfrac{1}{2}\langle {\mathcal D}_k h_k, {\mathcal D}_k{\phi}_k \rangle_{L_1^2} - \tfrac{1}{2}\langle{\mathcal D}_{-k} h_k, {\mathcal D}_{-k}  {\phi}_k \rangle_{L_1^2}
- \langle \partial_z h_k, \partial_z {\phi}_k \rangle_{L_1^2}
 - \langle {f}_k, {\phi}_k \rangle_{L_1^2} \label{eq:wkform}
\end{align}
for all $\phi_k \in H^1_{(k)0} ((0,1) \times {\mathbb R};{\mathbb C})$, so that $u_k=w_k+g_k$ is a weak solution of
\eqref{bvp:rzn1}, \eqref{bvp:rzn2}.

Both sides of \eqref{eq:wkform} are evidently bilinear continuous forms $(H^1_{(k)0} ((0,1) \times {\mathbb R};{\mathbb C}))^2 \rightarrow{\mathbb R}$.
The elementary estimate
$$w(r,z)^2=\left(\int_r ^1 r^{-\frac{1}{2}} w_r r^{\frac{1}{2}} \dr\right)^{\!\!2} \leq - \log r \int_0^1 w_r^2 r \dr,$$
establishes the Poincar\'{e} inequality
$$\int_0^1 w^2 r \dr \leq \frac{1}{4}\int_0^1 w_r^2 r \dr =\frac{1}{16} \int_0^1 ({\mathcal D}_k w + {\mathcal D}_{-k}w)^2 r \dr \leq \frac{1}{8}\int_0^1 \left(({\mathcal D}_k w)^2 + ({\mathcal D}_{-k}w)^2\right)r\dr$$
for $w \in {\mathscr D}_{(k)}([0,1) \times {\mathbb R};{\mathbb C})$ and hence $w \in H^1_{(k)0} ((0,1) \times {\mathbb R};{\mathbb C})$ by density, such that\linebreak
$w \mapsto (\tfrac{1}{2}\|{\mathcal D}_k w\|_{L_1^2}^2 + \tfrac{1}{2}\|{\mathcal D}_{-k} w\|_{L_1^2}^2 + \tfrac{1}{2}\|\partial_z w\|_{L_1^2}^2)^\frac{1}{2}$ is equivalent to the usual
norm for $H^1_{(k)0} ((0,1) \times {\mathbb R};{\mathbb C})$. The left-hand side of \eqref{eq:wkform} is therefore coercive, 
and the existence of a unique function $w \in H^1_{(k)0} ((0,1) \times {\mathbb R};{\mathbb C})$ satisfying \eqref{eq:wkform} for all $\phi_k \in H^1_{(k)0} ((0,1) \times {\mathbb R};{\mathbb C})$
follows from the Lax-Milgram lemma. The corresponding weak solution $u_k$ to \eqref{bvp:rzn1}, \eqref{bvp:rzn2}
is unique since the difference between two weak solutions satisfies \eqref{eq:wkform} with $h_k=f_k=0$ for all
$\phi_k \in H^1_{(k)0} ((0,1) \times {\mathbb R};{\mathbb C})$ and is therefore zero.
\end{proof}

It is possible to obtain an explicit formula for the (weak) solution to \eqref{bvp:rzn1}, \eqref{bvp:rzn2}. First suppose that
$f_k \in {\mathscr S}_{(k)}((0,1) \times {\mathbb R}; {\mathbb C})$, $g_k \in {\mathscr S}({\mathbb R};{\mathbb C})$ and take the 
Fourier transform in $z$, which yields the two-point boundary-value problem
\begin{align*}
        (\mathcal{D}_{1-k}\mathcal{D}_k - |\xi|^2 )\check{u}_k ={}& \check{f}_k, \qquad 0<r<1, \\
        \check{u}_k ={}& \check{g}_k , \qquad r=1,
    \end{align*}
where $\check{u}_k$, $\check{f}_k$ and $\check{g}_k$ denote the Fourier transforms of $u_k$, $f_k$ and $g_k$ respectively.
The solution to
\eqref{bvp:rzn1}, \eqref{bvp:rzn2} is
\begin{equation}
u_k=\mathcal{G}_k(f_k)  + \mathcal{B}_k(g_k), \label{eq:integral solution}
\end{equation}
in which
\begin{equation}\label{defn:G-H}
    \mathcal{G}_k(f_k) := \mathcal{F}^{-1}\left[\int_0^1 G_k(r,\rho)\,\check{f}_k(\rho)\,\rho\,\mathrm{d}\rho\right],\qquad \mathcal{B}_k(g_k) := \mathcal{F}^{-1}\left[\tilde{\mathcal D}_{k}G_k(r,1)\,\check{g}_k\right]
\end{equation}
(we again denote Bessel operators with respect to $r$, and later with respect to $s$, by ${\mathcal D}_k^i$ and with respect to $\rho$ by $\tilde{\mathcal D}_k^i$).
The Green's function $G_k(r,\rho)$ is given by
\begin{equation*}
    G_k(r,\rho) = \begin{cases}
        \displaystyle-I_{k}(|\xi|r) \tilde{K}_k(|\xi|\rho), & 0<r<\rho,\\
        \displaystyle-I_{k}(|\xi|\rho)\tilde{K}_k(|\xi|r), & \rho<r<1,
    \end{cases}
\end{equation*}
where $\tilde{K}_i$ is defined in terms of the modified Bessel functions $I_j$ and $K_j$ of the first and second kind by the formula
\begin{equation*}
    \tilde{K}_{i}(s) := K_{i}(s) - (-1)^{k-i}\frac{K_{k}(|\xi|)}{I_{k}(|\xi|)}I_{i}(s);
\end{equation*}
it satisfies
\begin{equation*}
    \tilde{K}_{-i}(s) = \tilde{K}_{i}(s), \qquad \mathcal{D}_{i}\tilde{K}_{i}(s) = -\tilde{K}_{i-1}(s), \qquad s\left[\tilde{K}_{i}(s)I_{i-1}(s) + I_{i}(s)\tilde{K}_{i-1}(s)\right] = 1.
\end{equation*}

In Section \ref{Estimates} below we establish the following theorem by showing that
$\mathcal{G}_k(f_k)$ and $\mathcal{B}_k(g_k)$ inherit additional regularity of $f_k$ and $g_k$ and observing that a weak solution of
\eqref{bvp:rzn1}, \eqref{bvp:rzn2} which lies in $H_{(k)}^2((0,1) \times {\mathbb R};{\mathbb C})$ is a strong solution (see
Remark \ref{rem:strong soln}).\pagebreak

\begin{theorem} $ $ \label{lo estimates}
\begin{itemize}
\item[(i)]
For each $f_k \in {\mathscr S}_{(k)}((0,1) \times {\mathbb R}; {\mathbb C})$
the function ${\mathcal G}_k(f_k)$ belongs to $H_{(k)}^2((0,1) \times {\mathbb R};{\mathbb C})$ with
$$\|{\mathcal G}_k(f_k)\|_{H^{2}_{(k)}}\lesssim\|f_{k}\|_{L_1^2},$$
such that ${\mathcal G}_k$ extends by density to a bounded linear operator $L_1^2((0,1) \times {\mathbb R};{\mathbb C}) \rightarrow H_{(k)}^2((0,1) \times {\mathbb R};{\mathbb C})$.
\item[(ii)]
For each $g_k \in {\mathscr S}({\mathbb R};{\mathbb C})$ the function ${\mathcal B}_k(g_k)$ belongs to $H_{(k)}^2((0,1) \times {\mathbb R};{\mathbb C})$ with
$$\|{\mathcal B}_k(g_k)\|_{H^{1}_{(k)}} \lesssim \|g_k\|_{H^{\frac{1}{2}}}, \qquad \|{\mathcal B}_k(g_k)\|_{H^{2}_{(k)}} \lesssim \|g_k\|_{H^{\frac{3}{2}}}$$
such that ${\mathcal B}_k$ extends by density to a bounded linear operator $H^{\frac{1}{2}}({\mathbb R};{\mathbb C}) \rightarrow H_{(k)}^{1}((0,1) \times {\mathbb R};{\mathbb C})$ and
$H^{\frac{3}{2}}({\mathbb R};{\mathbb C}) \rightarrow H_{(k)}^2((0,1) \times {\mathbb R};{\mathbb C})$.
\end{itemize}
\end{theorem}

\begin{corollary}
Suppose that $f_k \in L_1^2((0,1) \times {\mathbb R};{\mathbb C})$ and
$g_k \in H^\frac{1}{2}({\mathbb R}; {\mathbb C})$. The weak solution to \eqref{bvp:rzn1}, \eqref{bvp:rzn2} is given by
the formula \eqref{eq:integral solution}.
\end{corollary}
\begin{proof}
First suppose that $f_k \in {\mathscr S}_{(k)}((0,1) \times {\mathbb R}; {\mathbb C})$,
$g_k \in {\mathscr S}({\mathbb R};{\mathbb C})$. By construction
$\mathcal{G}_k(f_k)|_{r=1}=0$ and $\mathcal{B}_k(g_k)|_{r=1}=g_k$, such that $u_k|_{r=1}=g_k$, and substituting
\eqref{eq:integral solution} into the left-hand side of equation \eqref{eq:prewkform} and integrating by parts, we find that
this equation is satisfied for each $\phi_k \in H^1_{(k)0} ((0,1) \times {\mathbb R};{\mathbb C}))$.
The result follows by density (because of Theorem \ref{lo estimates} and the fact that the right-hand side
of \eqref{eq:prewkform} is a bilinear continuous form $(H^1_{(k)0} ((0,1) \times {\mathbb R};{\mathbb C}))^2 \rightarrow{\mathbb R}$).
\end{proof}

Our main result is also deduced from Theorem \ref{lo estimates}.

\begin{theorem}\label{thm:un-est;m}
Suppose that $m \in \mathbb{N}_0$. The function
\begin{equation}
u_k = \mathcal{G}_k(f_k)+ \mathcal{B}_k(g_k) \label{eq:the solution formula}
\end{equation}
\begin{itemize}
\item[(i)]
belongs to $H_{(k)}^{m+1}((0,1) \times {\mathbb R};{\mathbb C})$ with
$$\|u_k\|_{H^{m+1}_{(k)}} \lesssim \|f_{k}\|_{H^{m}_{(k)}} + \|g_k\|_{H^{m+\frac{1}{2}}}$$
whenever $f_k \in {H^{m}_{(k)}}((0,1) \times {\mathbb R};{\mathbb C})$, $g_k \in H^{m+\frac{1}{2}}({\mathbb R};{\mathbb C})$;
\item[(ii)]
belongs to $H_{(k)}^{m+2}((0,1) \times {\mathbb R};{\mathbb C})$ with
$$\|u_k\|_{H^{m+2}_{(k)}}\lesssim\|f_{k}\|_{H^{m}_{(k)}} + \|g_k\|_{H^{m+\frac{3}{2}}}$$
whenever $f_k \in {H^{m}_{(k)}}((0,1) \times {\mathbb R};{\mathbb C})$, $g_k \in H^{m+\frac{3}{2}}({\mathbb R};{\mathbb C})$.
\end{itemize}
\end{theorem}
\begin{proof} We establish these results inductively, noting that by density it suffices to prove the stated estimates for 
$f_k \in {\mathscr S}_{(k)}((0,1) \times {\mathbb R};{\mathbb C})$, $g_k \in {\mathscr S}({\mathbb R};{\mathbb C})$; they
hold for $m=0$ by Theorem \ref{lo estimates}.

Suppose that assertion (i) holds for some $m \in {\mathbb N}_0$.
Differentiating \eqref{eq:the solution formula}, we find that $v_{k}^{0}:=\partial_z u_k$ and $v_{k}^{\pm}:= \mathcal{D}_{\pm k}u_k$ satisfy
\begin{equation*}
v_{k}^{0} = \mathcal{G}_{k}(f_{k}^{0}) + \mathcal{B}_{k}(g_{k}^{0}), \qquad v_{k}^{\pm} = \mathcal{G}_{k\mp1}(f_{k}^{\pm}) + \mathcal{B}_{k\mp1}(g_{k}^{\pm}),
\end{equation*}
where $f_{k}^{0}:= \partial_{z}f_{k}$, $ g_{k}^{0}:= \partial_z g_k$, $f_{k}^{\pm}:= \mathcal{D}_{\pm k}f_{k}$ and
\begin{equation*}
\begin{split}
   g_{k}^{\pm}:={}& \mathcal{F}^{-1}\left[\int_0^1 \mathcal{D}_{\pm k}G_{k}(1,\rho)\check{f}_{k}(\rho)\,\rho\,\mathrm{d}\rho + \mathcal{D}_{\pm k}\tilde{\mathcal D}_kG_{k}(1,1)\check{g}_{k}\right],\\
    ={}& \mathcal{F}^{-1}\left[\int_0^1 \frac{I_k (|\xi|\rho)}{I_k (|\xi|)}\check{f}_{k}(\rho)\,\rho\,\mathrm{d}\rho + \frac{|\xi|I_{k\mp1}(|\xi|)}{I_k (|\xi|)}\check{g}_{k}\right].
\end{split}
\end{equation*}
By the inductive hypothesis
\begin{equation}
\|v_k^0\|_{H^{m+1}_{(k)}} \lesssim \|f_{k}^0\|_{H^{m}_{(k)}} + \|g_k^0\|_{H^{m+\frac{1}{2}}} , \qquad \|v_k^\pm\|_{H^{m+1}_{(k\mp1)}} \lesssim \|f_{k}^\pm\|_{H^{m}_{(k\mp1)}} + \|g_k^\pm\|_{H^{m+\frac{1}{2}}}
\label{eq:indresult}
\end{equation}

Obviously
\begin{equation}
    \|f_{k}^{\pm}\|_{H^{m}_{(k\mp1)}} = \|\mathcal{D}_{\pm k}f_{k}\|_{H^{m}_{(k\mp1)}} \lesssim \|f_{k}\|_{H^{m+1}_{(k)}}, \qquad
    \|f_{k}^{0}\|_{H^{m}_{(k)}} = \|\partial_{z} f_{k}\|_{H^{m}_{(k)}}\lesssim \|f_{k}\|_{H^{m+1}_{(k)}} \label{eq:oldnew1}
\end{equation}
and
\begin{equation}
        \|g_{k}^{0}\|_{H^{m+\frac{1}{2}}} =\|\partial_z g_k \|_{H^{m+\frac{1}{2}}} \leq \|g_{k}\|_{H^{m+\frac{3}{2}}}, \label{eq:oldnew2}
\end{equation}
and furthermore
\begin{align}
        \|g_{k}^{\pm}\|^2_{H^{m+\frac{1}{2}}} 
        \leq{}& \int_{-\infty}^{\infty} (1+|\xi|^2)^{m+\frac{1}{2}}\left(\int_0^1 \frac{I_k (|\xi|\rho)}{I_k (|\xi|)}\check{f}_{k}(\rho)\,\rho\,\mathrm{d}\rho\right)^{\!\!2}\,\mathrm{d}\xi + \int_{-\infty}^{\infty} (1+|\xi|^2)^{m+\frac{1}{2}}\left(\frac{|\xi|I_{k\mp1}(|\xi|)}{I_k (|\xi|)}\right)^{\!\!2} |\check{g}_{k}|^2\,\mathrm{d}\xi\nonumber \\
        \leq{}& \int_{-\infty}^{\infty} (1+|\xi|^2)^{m+\frac{1}{2}}\left(\int_0^1 \frac{I^2_{k}(|\xi|\rho)}{I^2_{k}(|\xi|)}\,\rho\,\mathrm{d}\rho\right)\!\!\!\left(\int_0^1 |\check{f}_{k}(\rho)|^2\,\rho\,\mathrm{d}\rho\right)\,\mathrm{d}\xi\nonumber \\
        {}& \qquad \mbox{}+ \int_{-\infty}^{\infty} (1+|\xi|^2)^{m+\frac{1}{2}}\left(\frac{|\xi|I_{k\mp1}(|\xi|)}{I_k (|\xi|)}\right)^{\!\!2} |\check{g}_{k}|^2\,\mathrm{d}\xi\nonumber \\
        ={}& \int_{-\infty}^{\infty} (1+|\xi|^2)^{m+\frac{1}{2}}\left(\frac{I_k ^2(|\xi|) - I_{k-1}(|\xi|)I_{k+1}(|\xi|)}{2I^2_{k}(|\xi|)}\right)\!\!\!\left(\int_0^1 |\check{f}_{k}(\rho)|^2\,\rho\,\mathrm{d}\rho\right)\,\mathrm{d}\xi \nonumber \\
        & \qquad\mbox{}+ \int_{-\infty}^{\infty} (1+|\xi|^2)^{m+\frac{1}{2}}\left(\frac{|\xi|I_{k\mp1}(|\xi|)}{I_k (|\xi|)}\right)^{\!\!2} |\check{g}_{k}|^2\,\mathrm{d}\xi\nonumber \\
        \lesssim{}& \int_{-\infty}^{\infty} \int_0^1 (1+|\xi|^2)^m|\check{f}_{k}(\rho)|^2\,\rho\,\mathrm{d}\rho\,\mathrm{d}\xi + \int_{-\infty}^{\infty} (1+|\xi|^2)^{m+\frac{3}{2}} |\check{g}_{k}|^2\,\mathrm{d}\xi\nonumber \\
       ={}& \|\partial_z^m f_{k}\|^2_{L_1^2} + \|g_{k}\|^2_{H^{m+\frac{3}{2}}} \nonumber \\
       \leq{}& \|f_k\|_{H_{(k)}^m}^2 + \|g_{k}\|^2_{H^{m+\frac{3}{2}}}. \label{eq:oldnew3}
    \end{align}

It follows that
    \begin{equation*}\begin{split}
         \|u_k\|^2_{H^{m+2}_{(k)}}  &{}= \|u_k\|^2_{H^{m+1}_{(k)}} + \tfrac{1}{2}\|\mathcal{D}_{k}u_k\|^2_{H^{m+1}_{(k-1)}} + \tfrac{1}{2}\|\mathcal{D}_{-k}u_k\|^2_{H^{m+1}_{(k+1)}} + \|\partial_z u_k\|^2_{H^{m+1}_{(k)}}\\
        &{}= \|u_k\|^2_{H^{m+1}_{(k)}} + \tfrac{1}{2}\|v^{+}_k\|^2_{H^{m+1}_{(k-1)}} + \tfrac{1}{2}\|v^-_k\|^2_{H^{m+1}_{(k+1)}} + \|v^0_k\|^2_{H^{m+1}_{(k)}} \\
  &{}\lesssim \|f_k\|^2_{H^{m}_{(k)}} + \|g_k\|^2_{H^{m+\frac{1}{2}}}  + \tfrac{1}{2}\|f^+_k\|^2_{H^{m}_{(k-1)}} + \tfrac{1}{2}\|g^+_k\|^2_{H^{m+\frac{1}{2}}} \\
        & \qquad\mbox{} + \tfrac{1}{2}\|f^-_k\|^2_{H^{m}_{(k+1)}} + \tfrac{1}{2}\|g^-_k\|^2_{H^{m+\frac{1}{2}}} + \|f^0_k\|^2_{H^{m}_{(k)}} + \|g^0_k\|^2_{H^{m+\frac{1}{2}}}\\
         &{} \lesssim \|f_k\|^2_{H^{m+1}_{(k)}} + \|g_k\|^2_{H^{m+\frac{3}{2}}},
          \end{split}
    \end{equation*}
    where the third line follows from the second by \eqref{eq:indresult} and the fourth from the third by \eqref{eq:oldnew1}--\eqref{eq:oldnew3}.
    
Assertion (ii) is proved in a similar fashion.
    \end{proof}

In the following theory we make use of the rescaled modified Struve function of the second kind\linebreak
$\mathbf{M}_{k}: [0,\infty) \rightarrow {\mathbb R}$ given by the formula
$$ \mathbf{M}_{k}(s) := 2^{k-1}\sqrt{\pi}\Gamma(k+\tfrac{1}{2})\left(I_{k}(s) - \mathbf{L}_{k}(s)\right),$$
where $\mathbf{L}_{k}$ is the modified Struve function of the first kind of order $k\in\mathbb{R}$
(see Abramowitz \& Stegun \cite[\S 12.2]{AbramowitzStegun}).\pagebreak

\begin{proposition}
The function $\mathbf{M}_{k}$ satisfies
\begin{itemize}
\item[(i)] the integral identity
\begin{equation*}
        \int_{a}^{b}  Z_k(s)\,s^{k}\ds = \Big[W_{k}[\mathbf{M}_{k},Z_{k}](s)\Big]_{a}^{b}
\end{equation*}
for each $k \in {\mathbb N}_0$, where $Z_k \in \{I_k, K_k, \tilde{K}_k\}$ and $W_k[v_1,v_2](s)=s(v_1 {\mathcal D}_k v_2- v_2 {\mathcal D}_k v_1)(s)$;
\item[(ii)]
the differential identities
$$
        \mathcal{D}_{k} \mathbf{M}_{k}(s) = (2k-1)\mathbf{M}_{k-1}(s), \qquad k \in {\mathbb N},
$$
and 
$$\mathcal{D}_{-k}\mathbf{M}_{k}(s) = \frac{1}{(2k+1)}\left(\mathbf{M}_{k+1}(s) - s^k\right), \qquad k \in {\mathbb N}_0;$$
\item[(iii)]
the estimates $ 0\leq \mathbf{M}_{k}(s) \leq s^{k-1}$, $k \in {\mathbb N}$ and $0\leq \mathbf{M}_{0}(s) \leq 2s^{-1}$
for all $s \in (0,\infty)$ and
\begin{equation*}
        \begin{aligned}
            \mathbf{M}_{k}(s) ={}& \frac{\sqrt{\pi}\Gamma(k+\tfrac{1}{2})}{2\Gamma(k+1)} \,s^k \left(1 + \mathcal{O}(s)\right), & \qquad &s\to0,\\
            \mathbf{M}_{k}(s) ={}& s^{k-1}\left(1 + \mathcal{O}(s^{-2})\right) & \qquad &s\to\infty
        \end{aligned}
    \end{equation*}
for each $k \in {\mathbb N}_0$.
\end{itemize}
The functions $s\mapsto s^{k}\mathbf{M}_{k}(s)$ and $s\mapsto s^{-k}\mathbf{M}_{k}(s)$ are monotone increasing and decreasing over $s\in(0,\infty)$, respectively, for each $k\in\mathbb{N}$. Furthermore $s\mapsto \mathbf{M}_{0}(s)$ is monotone decreasing over $s\in(0,\infty)$.
\end{proposition}
\begin{proof}
Part (i) follows from the calculation
$${\mathcal D}_0 W_{k}[v_1,v_2](s) =  s\left( v_{1} \mathcal{D}_{1-k}\mathcal{D}_{k}v_{2}- v_{2} \mathcal{D}_{1-k}\mathcal{D}_{k}v_{1}\right)(s),$$
such that
$${\mathcal D}_0 W_{k}[\mathbf{M}_{k},Z_{k}](s) = Z_k(s)s^k,$$
because
$$ \mathcal{D}_{1-k}\mathcal{D}_{k} Z_k =   Z_k, \qquad  \mathcal{D}_{1-k}\mathcal{D}_{k} \mathbf{M}_{k}=   \mathbf{M}_k-s^{k-1},$$
while part (ii) follows from the facts that  $\mathcal{D}_k I_k = I_{k-1}$, $\mathcal{D}_{k}\mathbf{L}_{k}= \mathbf{L}_{k-1}$ for $k \in {\mathbb N}$ and 
$$
\mathcal{D}_{-k} I_k = I_{k+1}, \qquad \mathcal{D}_{-k}\mathbf{L}_{k}(s) = \mathbf{L}_{k+1}(s) + \frac{s^k}{2^k \sqrt{\pi}\Gamma(k+\tfrac{3}{2})}
$$
for $k \in {\mathbb N}_0$. Turning to part (iii), we find from the 
integral representations of modified Bessel and Struve functions (Abramowitz \& Stegun \cite[(9.6.18), (12.2.2)]{AbramowitzStegun}) that
$$
 \mathbf{M}_{k}(s) = s^{k}\int_0^{\frac{\pi}{2}} \mathrm{e}^{-s\cos(\theta)}\,\sin^{2n}(\theta)\,\mathrm{d}\theta,
$$
such that  $\mathbf{M}_{k}(s) >0$ and 
\begin{align*}
\mathbf{M}_{k}(s) 
={}& s^{k-1}\int_0^{\frac{\pi}{2}} \frac{\mathrm{d}}{\mathrm{d}\theta}\left(\mathrm{e}^{-s\cos(\theta)}\right)\,\sin^{2k-1}(\theta)\,\mathrm{d}\theta\\
={}& s^{k-1}\left(1 - (2k-1)\int_0^{\frac{\pi}{2}} \mathrm{e}^{-s\cos(\theta)}\cos(\theta)\sin^{2(k-1)}(\theta)\,\mathrm{d}\theta\right)\\
\leq {}& s^{k-1}
\end{align*}
for $k \in {\mathbb N}$, while the identity
$$
    \mathbf{M}_{0}(s) = \frac{1}{3}\left(\mathbf{M}_{2}(s) - s\right) + \frac{2}{s}\mathbf{M}_{1}(s),
$$
shows that $ 0 \leq \mathbf{M}_{0}(s) \leq 2s^{-1}$. The asymptotic estimates are also obtained from the corresponding estimates for $I_n$ and ${\mathbf L}_n$
(see Abramowitz \& Stegun \cite[(9.6.10), (9.7.1), (12.2.1), (12.2.6)]{AbramowitzStegun}).

The final assertion follows from the calculations
    \begin{equation*}
    \begin{split}
        \frac{\mathrm{d}}{\mathrm{d}s} \left(s^{-k}\mathbf{M}_{k}(s)\right) ={}& s^{-k}\mathcal{D}_{-k}\mathbf{M}_{k}(s) = \frac{s^{-k}}{(2k+1)}\big(\mathbf{M}_{k+1}(s) - s^{k}\big)\leq 0, \\
        \frac{\mathrm{d}}{\mathrm{d}s} \left(s^{k}\mathbf{M}_{k}(s)\right) ={}& s^{k}\mathcal{D}_{k}\mathbf{M}_{k}(s) = (2k-1) s^{k}\mathbf{M}_{k-1}(s) \geq 0,
    \end{split}
    \end{equation*}
for $k \in {\mathbb N}$ and
    \begin{equation*}
    \begin{split}
        \frac{\mathrm{d}}{\mathrm{d}s}\mathbf{M}_{0}(s) ={}& \mathbf{M}_{1}(s)-1 \leq 0. \qedhere
    \end{split}
    \end{equation*}
\end{proof}

\begin{remark}
The same argument shows that the functions $s\mapsto s^{k}I_{k}(s)$ and $s\mapsto s^{-k}I_{k}(s)$ are monotone increasing and $s\mapsto s^{k}K_{k}(s)$ and $s\mapsto s^{-k}K_{k}(s)$ are monotone decreasing over $s\in(0,\infty)$
for each $k\in\mathbb{N}_0$.
\end{remark}

Finally, we record the integrals
\begin{equation}
       \int_a^b I_ \ell(|\xi|r)\,r^{ \ell+1}\,\mathrm{d}r= \left[\frac{r^{ \ell+1}}{|\xi|}I_{ \ell+1}(|\xi|r)\right]_a^b, \qquad
       \int_a^b \tilde{K}_ \ell(|\xi|r)\,r^{ \ell+1}\,\mathrm{d}r= \left[-\frac{r^{ \ell+1}}{|\xi|}\tilde{K}_{\ell+1}(|\xi|r)\right]_a^b
       \label{Block1}
\end{equation}
for $\ell \geq 0$,
\begin{align}
 \int_a^b I_ \ell(|\xi|r)\,r^{ \ell}\,\mathrm{d}r ={}&\left[|\xi|^{- \ell}\, r\left(\mathbf{M}_{ \ell}(|\xi|r)I_{ \ell-1}(|\xi|r) - (2 \ell-1)\mathbf{M}_{ \ell-1}(|\xi|r)I_ \ell (|\xi|r)\right)\right]_a^b, \label{Block2I}\\
        \int_a^b \tilde{K}_ \ell(|\xi|r)\,r^{ \ell}\,\mathrm{d}r ={}&\left[-|\xi|^{- \ell}\, r\left(\mathbf{M}_{ \ell}(|\xi|r)\tilde{K}_{ \ell-1}(|\xi|r) + (2 \ell-1)\mathbf{M}_{ \ell-1}(|\xi|r)\tilde{K}_ \ell(|\xi|r)\right)\right]_a^b \label{Block2K}
    \end{align}
for $\ell \geq 1$,
\begin{align}
    \int_a^b I_ \ell(|\xi|r)\,r^{ \ell-2}\,\mathrm{d}r ={}&\left[-r^{ \ell-1}I_ \ell (|\xi|r) + |\xi|^{2- \ell}r\left(\mathbf{M}_{ \ell-1}(|\xi|r)I_{ \ell-2}(|\xi|r) - (2 \ell-3)\mathbf{M}_{ \ell-2}(|\xi|r)I_{ \ell-1}(|\xi|r)\right)\right]_a^b, \label{Block3I}\\
\int_a^b \tilde{K}_ \ell(|\xi|r)\,r^{ \ell-2}\,\mathrm{d}r ={}&\left[-r^{ \ell-1}\tilde{K}_ \ell(|\xi|r) + |\xi|^{2- \ell}r\left(\mathbf{M}_{ \ell-1}(|\xi|r)\tilde{K}_{ \ell-2}(|\xi|r) + (2 \ell-3)\mathbf{M}_{ \ell-2}(|\xi|r)\tilde{K}_{ \ell-1}(|\xi|r)\right)\right]_a^b \label{Block3K}
    \end{align}
for $\ell \geq 2$, and
\begin{align}
\int_a^b I_{\ell-2}&(|\xi|r)\,r^{\ell-2}\,\mathrm{d}r = \nonumber \\
{}&\left[\frac{r^{\ell-1}I_{\ell-2}(|\xi|r)}{2\ell-3} -\frac{|\xi|^{2-\ell}r}{2\ell-3}\left(\mathbf{M}_{\ell-1}(|\xi|r)I_{\ell-2}(|\xi|r) - (2\ell-3)\mathbf{M}_{\ell-2}(|\xi|r)I_{\ell-1}(|\xi|r)\right)\right]_a^b, \label{Block4I}\\
\int_a^b \tilde{K}_{\ell-2}&(|\xi|r)\,r^{\ell-2}\,\mathrm{d}r = \nonumber \\
{}&\left[\frac{r^{\ell-1}\tilde{K}_{\ell-2}(|\xi|r)}{2\ell-3} -\frac{|\xi|^{2-\ell}r}{2\ell-3}\left(\mathbf{M}_{\ell-1}(|\xi|r)\tilde{K}_{\ell-2}(|\xi|r) + (2\ell-3)\mathbf{M}_{\ell-2}(|\xi|r)\tilde{K}_{\ell-1}(|\xi|r)\right)\right]_a^b \label{Block4K}
    \end{align}
for $\ell \geq 2$.

\subsection{Estimates for $\mathcal{G}_k$ and ${\mathcal B}_k$} \label{Estimates}

We begin by establishing Theorem \ref{lo estimates}(i) in three steps (Lemmata \ref{lem:G(f_k)-est}, \ref{lem:DG(f_k)-est} and \ref{lem:D2G(f_k)-est}).
In analogy with the derivatives $\mathcal{D}_{k}f=  r^{-k}\frac{\mathrm{d}}{\mathrm{d}r}(r^{k}f)$, our method is to 
consider integrals of the form
\begin{equation*}
    r^{-\ell}\int_{0}^{1}|f(r,\rho)|\,\rho^{1+\ell}\,\mathrm{d}\rho,\qquad \rho^{-\ell}\int_{0}^{1}|f(r,\rho)|\,r^{1+\ell}\,\mathrm{d}r,
\end{equation*}
for suitably chosen $\ell\in\mathbb{Z}$.

\begin{proposition}\label{prop:G-bound}
For each $q \in \{0,1,2\}$ and $k \in {\mathbb N}_0$ the estimates
    \begin{equation*}
    |\xi|^{q}\rho^{-k}\int_0^1 |G_k(r,\rho)|\,r^{1+k}\,\mathrm{d}r, \ |\xi|^{q} r^{-k}\int_0^1 |G_k(r,\rho)|\,\rho^{1+k}\,\mathrm{d}\rho \lesssim 1
\end{equation*}
hold uniformly over $r,\rho\in[0,1]$ and $\xi \in {\mathbb R}$.
\end{proposition}
\begin{proof}
Using the integrals \eqref{Block1}
we find that
\begin{equation*}
    \begin{split}
        \mathcal{I}_{1}:={}& \rho^{-k}\int_0^1 |G_k(r,\rho)|\,r^{1+k}\,\mathrm{d}r\\
        ={}& \tilde{K}_k(|\xi|\rho)\rho^{-k}\int_0^\rho I_k(|\xi|r)\,r^{1+k}\,\mathrm{d}r + I_k(|\xi|\rho)\rho^{-k}\int_\rho^1 \tilde{K}_k(|\xi|r)\,r^{1+k}\,\mathrm{d}r\\
        ={}& \frac{1}{|\xi|^2} \left(1 - \frac{\rho^{-k}I_k(|\xi|\rho)}{I_{k}(|\xi|)}\right), \\
    \end{split}
\end{equation*}
which implies that $|\xi|^2 \mathcal{I}_{1}\leq 1$. Furthermore, the function $s\mapsto s^{-k}I_k (s)$ is monotone increasing over \linebreak
$s\in(0,\infty)$, such that
\begin{align*}
    \frac{1}{|\xi|^2} \left(1 - \frac{\rho^{-k}I_k(|\xi|\rho)}{I_k (|\xi|)}\right) & \leq \frac{1}{|\xi|^2} \left(1 - \frac{\lim_{s\to0}s^{-k}I_k(s)}{|\xi|^{-k} I_k (|\xi|)}\right) \\
    & = \frac{1}{|\xi|^2} \left(1 - \frac{2^{-k}}{\Gamma(k+1)|\xi|^{-k} I_k (|\xi|)}\right) \\
    & \rightarrow \begin{cases}
        \frac{1}{4(k+1)}, & |\xi|\to0,\\
        0, & |\xi|\to\infty;
    \end{cases}
\end{align*}
it follows that ${\mathcal I}_1 \lesssim 1$ and hence
$
    |\xi|\mathcal{I}_{1} \leq \frac{1}{2}\mathcal{I}_1 + \frac{1}{2}|\xi|^2\mathcal{I}_1 \lesssim 1.
$

The second estimate follows immediately from the symmetry property $G_k(r,\rho) = G_k(\rho,r)$.
\end{proof}

\begin{lemma}\label{lem:G(f_k)-est}
The estimates
\begin{equation*}
    \|\mathcal{G}_k(f_k)\|_{L_1^2},\  \|\partial_z\mathcal{G}_k(f_k)\|_{L_1^2},\  \|\partial_z^2\mathcal{G}_k(f_k)\|_{L_1^2} \lesssim  \,\|f_k\|_{L_1^2}
\end{equation*}
hold for each $f_k \in {\mathscr S}_{(k)}((0,1) \times {\mathbb R};{\mathbb C})$ and $k \in {\mathbb N}_0$.

\end{lemma}
\begin{proof}
For each fixed $q\in\{0,1,2\}$ we find that
\begin{equation*}
    \begin{split}
        \|\partial_z^q \mathcal{G}_k(f_k)\|^2_{L_1^2} ={}& \int_{-\infty}^{\infty}\int_0^1 \left|\int_0^1 |\xi|^q G_k(r,\rho) \check{f}_k(\rho)\,\rho\,\mathrm{d}\rho\right|^2\,r\,\mathrm{d}r\,\mathrm{d}\xi,\\
        \leq{}& \int_{-\infty}^{\infty}\int_0^1 \left(|\xi|^q r^{-k}\int_0^1 |G_k(r,\rho) |\,\rho^{1+k}\,\mathrm{d}\rho\right)\left(|\xi|^q r^{k}\int_0^1 |G_{k}(r,\rho)| |\check{f}_k(\rho)|^2\,\rho^{1-k}\,\mathrm{d}\rho\right)\,r\,\mathrm{d}r\,\mathrm{d}\xi,\\
        \lesssim{}& \int_{-\infty}^{\infty}\int_0^1 \left(|\xi|^q r^{k}\int_0^1 |G_k(r,\rho)| |\check{f}_k(\rho)|^2\,\rho^{1-k}\,\mathrm{d}\rho\right)\,r\,\mathrm{d}r\,\mathrm{d}\xi,\\
        ={}& \int_{-\infty}^{\infty}\int_0^1 \left(|\xi|^q \rho^{-k}\int_0^1 |G_k(r,\rho)|\,r^{1+k}\,\mathrm{d}r \right) |\check{f}_k(\rho)|^2\,\,\rho\,\mathrm{d}\rho\,\mathrm{d}\xi,\\
        \lesssim{}& \int_{-\infty}^{\infty}\int_0^1 |\check{f}_k(\rho)|^2\,\,\rho\,\mathrm{d}\rho\,\mathrm{d}\xi,\\
        \lesssim{}& \|f_k\|_{L_1^2},
    \end{split}
\end{equation*}
where we have used the Cauchy-Schwarz inequality
\begin{equation}
\left|\int_0^1 |\xi|^q G_k(r,\rho) \check{f}_k(\rho)\,\rho\,\mathrm{d}\rho\right|^2
\!\!\leq
\left(|\xi|^q r^{-k}\int_0^1 |G_k(r,\rho) |\,\rho^{1+k}\,\mathrm{d}\rho\right)\!\!\left(|\xi|^q r^{k}\int_0^1 |G_{k}(r,\rho)| |\check{f}_k(\rho)|^2\,\rho^{1-k}\,\mathrm{d}\rho\right), \label{eq:CS}
\end{equation}
Fubini's theorem and Proposition \ref{prop:G-bound}.
\end{proof}

The next step is to estimate
$$
        \mathcal{D}_{\pm k}\mathcal{G}_k(f_k) 
   =  \mathcal{F}^{-1}\left[\int_0^1 \mathcal{D}_{\pm k}G_k(r,\rho)\,\rho\,\check{f}_k(\rho)\,\rho\,\mathrm{d}\rho\right],
$$
where
\begin{equation*}
    \mathcal{D}_{\pm k}G_k(r,\rho) = \begin{cases}
        \displaystyle-|\xi| I_{k\mp1}(|\xi|r) \tilde{K}_k(|\xi|\rho), & 0<r<\rho,\\
        \displaystyle |\xi| I_k (|\xi|\rho)\tilde{K}_{k\mp1}(|\xi|r), & \rho<r<1.\\
    \end{cases}
\end{equation*}

\begin{proposition}\label{prop:DG-bound}$ $
\begin{itemize}
\item[(i)]
For each fixed $q \in \{0,1\}$ and $k \in {\mathbb N}$ the estimates
$$
        |\xi|^{q}\rho^{1-k}\int_0^1 |\mathcal{D}_{k}G_k(r,\rho)|\,r^{k}\,\mathrm{d}r,\  |\xi|^{q}r^{1-k}\int_0^1 |\mathcal{D}_{k}G_k(r,\rho)|\,\rho^{k}\,\mathrm{d}\rho \lesssim 1
$$
hold uniformly over $r,\rho\in[0,1]$ and $\xi \in {\mathbb R}$.
\item[(ii)]
For each fixed $q \in \{0,1\}$ and $k \in {\mathbb N}_0$ the estimates
$$
        |\xi|^{q}\rho^{-k}\int_0^1 |\mathcal{D}_{-k}G_k(r,\rho)|\,r^{1+k}\,\mathrm{d}r,\  |\xi|^{q}r^{-k}\int_0^1 |\mathcal{D}_{-k}G_k(r,\rho)|\,\rho^{1+k}\,\mathrm{d}\rho \lesssim 1
$$
hold uniformly over $r,\rho\in[0,1]$ and $\xi \in {\mathbb R}$.
\end{itemize}
\end{proposition}
\begin{proof}
(i) Using the integrals \eqref{Block2I}, \eqref{Block2K}, we find that\pagebreak
\begin{align*}
        \mathcal{I}_2 :={}& r^{1-k}\int_0^1 |\mathcal{D}_{k}G_{k}(r,\rho)|\,\rho^{k}\,\mathrm{d}\rho\\
        ={}& |\xi|\tilde{K}_{k-1}(|\xi|r) r^{1-k}\int_0^r I_k (|\xi|\rho)\,\rho^{k}\,\mathrm{d}\rho + |\xi| I_{k-1}(|\xi|r)  r^{1-k}\int_r^1 \tilde{K}_k(|\xi|\rho)\,\rho^{k}\,\mathrm{d}\rho\\
        ={}& 2(|\xi|r)^{2-k}\frac{1}{|\xi|}K_{k-1}(|\xi|r)\left(\mathbf{M}_{k}(|\xi|r)I_{k-1}(|\xi|r)-(2k-1)\mathbf{M}_{k-1}(|\xi|r)I_k (|\xi|r)\right) \\
        & \qquad\mbox{}+ 2(|\xi|r)^{2-k}\frac{K_{k}(|\xi|)}{|\xi|I_k (|\xi|)}I_{k-1}(|\xi|r)\left(\mathbf{M}_{k}(|\xi|r)I_{k-1}(|\xi|r)-(2k-1)\mathbf{M}_{k-1}(|\xi|r)I_k (|\xi|r)\right) \\
        & \qquad\mbox{}+ (|\xi|r)^{1-k}\frac{1}{|\xi|}\left((2k-1)\mathbf{M}_{k-1}(|\xi|r)  - \frac{\mathbf{M}_{k}(|\xi|)}{I_k (|\xi|)}I_{k-1}(|\xi|r)\right).
\end{align*}
Applying the inequalities
\begin{equation*}
    \begin{split}
        \frac{K_k(|\xi|)}{I_k (|\xi|)}I_{k-1}(|\xi|r) \leq \frac{K_k(|\xi|)}{I_k (|\xi|)}I_{k-1}(|\xi|) = \frac{1}{I_k (|\xi|)}\left(\frac{1}{|\xi|} - K_{k-1}(|\xi|)I_k (|\xi|)\right)\leq \frac{1}{|\xi|I_k (|\xi|)}\leq \frac{1}{|\xi|rI_k (|\xi|r)}
    \end{split}
\end{equation*}
and 
\begin{equation*}
\begin{split}
(|\xi|r)^{1-k}\left((2k-1)\mathbf{M}_{k-1}(|\xi|r)  - \frac{\mathbf{M}_{k}(|\xi|)}{I_k (|\xi|)}I_{k-1}(|\xi|r)\right) 
\leq{}& \frac{1}{2^{k-1}\Gamma(k)}\left(2^{k-1}\sqrt{\pi}\Gamma(k+\tfrac{1}{2})  - \frac{\mathbf{M}_{k}(|\xi|)}{I_k (|\xi|)}\right),
\end{split}
\end{equation*}
which follow from the facts that $s \to I_k(s)$, $s \to s^{-k}I_k(s)$ are monotone increasing, while $s \mapsto s^{-k}\mathbf{M}_{k}(s)$ is monotone decreasing, yields
$$
        |\xi|^{q}\mathcal{I}_{2} \leq 2F_{q}^1(|\xi|r) + 2F_q^2(|\xi|r) + F_{q}^3(|\xi|),
$$
where
\begin{equation*}
    \begin{split}
    F_{q}^1(s):={}& s^{q+1-k}K_{k-1}(s)\left(\mathbf{M}_{k}(s)I_{k-1}(s)- (2k-1)\mathbf{M}_{k-1}(s)I_k (s)\right), \\
    F_q^2(s):={}& s^{q-k}\,\frac{\mathbf{M}_{k}(s)I_{k-1}(s)- (2k-1)\mathbf{M}_{k-1}(s)I_k (s)}{I_k (s)}, \\
    F_{q}^3(s):={}& s^{q-1}\frac{1}{2^{k-1}\Gamma(k)}\left(2^{k-1}\sqrt{\pi}\Gamma(k+\tfrac{1}{2})  - \frac{\mathbf{M}_{k}(s)}{I_k (s)}\right).\\
    \end{split}
\end{equation*}
Finally, observing that
\begin{equation*}
    \begin{aligned}
     F_0^1(s) \to& \begin{cases}
        0, & s\to0,\\
        0, & s\to\infty,
    \end{cases} 
    &\qquad\qquad 
    F_1^1(s) \to& \begin{cases}
        0, & s\to0\\
        \frac{1}{2}, & s\to\infty,
    \end{cases}   \\
    F_0^2(s) \to& \begin{cases}
        \frac{1}{2k+1}, & s\to0,\\
        0, & s\to\infty,
    \end{cases} 
    &\qquad\qquad 
    F_1^2(s) \to& \begin{cases}
        0, & s\to0\\
        1, & s\to\infty,
    \end{cases}   \\
    F_0^3(s) \to& \begin{cases}
        \frac{2k}{2k+1}, & s\to0,\\
        0, & s\to\infty,
    \end{cases} 
    &\qquad\qquad 
    F_1^3(s) \to& \begin{cases}
        0, & s\to0\\
        \frac{\sqrt{\pi}\Gamma(k+\tfrac{1}{2})}{\Gamma(k)}, & s\to\infty,
    \end{cases}   
    \end{aligned}
\end{equation*}
we find that $F_q^j(s) \lesssim 1$ over $s\in[0,\infty)$ for $q\in\{0,1\}$ and $j\in\{1,2,3\}$, and thus conclude that
$|\xi|^{q}\mathcal{I}_{2} \lesssim1$.

Similarly
\begin{align*}
        \mathcal{I}_3 :={}& \rho^{1-k}\int_0^1 |\mathcal{D}_{k}G_{k}(r,\rho)|\,r^{k}\,\mathrm{d}r\\
        ={}& |\xi|\tilde{K}_k(|\xi|\rho) \rho^{1-k}\int_0^\rho I_{k-1}(|\xi|r)\,r^{k}\,\mathrm{d}r + |\xi| I_k (|\xi|\rho) \rho^{1-k}\int_\rho^1 \tilde{K}_{k-1}(|\xi|r)\,r^{k}\,\mathrm{d}r\\
       ={}& \frac{2}{|\xi|} (|\xi|\rho)\tilde{K}_k(|\xi|\rho) I_k (|\xi|\rho),
\end{align*}
such that
$$|\xi|^q {\mathcal I}_3 \leq 2 F^4_q(|\xi|\rho), \qquad  F_q^4(s) := s^{q} K_{k}(s)I_k (s)$$
for $q\in\{0,1\}$. Because
\begin{equation*}
    \begin{aligned}
     F_0^4(s) \to& \begin{cases}
        \frac{1}{2k}, & s\to0,\\
        0, & s\to\infty,
    \end{cases} 
    &\qquad\qquad 
    F_1^4(s) \to& \begin{cases}
        0, & s\to0,\\
        \frac{1}{2}, & s\to\infty,
    \end{cases}   \\ 
    \end{aligned}
\end{equation*}
we conclude that $F_q^4(s)\lesssim1$ over $s\in[0,\infty)$ for $q\in\{0,1\}$ and hence 
$|\xi|^{q}\mathcal{I}_{3} \lesssim 1$.

(ii) It follows from the integrals \eqref{Block1} that
\begin{equation*}
    \begin{split}
        \mathcal{I}_4 :={}& r^{-k}\int_0^1 |\mathcal{D}_{-k}G_{k}(r,\rho)|\,\rho^{k+1}\,\mathrm{d}\rho\\
        ={}& |\xi|\tilde{K}_{k+1}(|\xi|r) r^{-k}\int_0^r I_k (|\xi|\rho)\,\rho^{k+1}\,\mathrm{d}\rho + |\xi| I_{k+1}(|\xi|r)  r^{-k}\int_r^1 \tilde{K}_k(|\xi|\rho)\,\rho^{k+1}\,\mathrm{d}\rho\\
         ={}& \frac{2}{|\xi|}( |\xi| r) K_{k+1}(|\xi|r) I_{k+1}(|\xi|r) + \frac{I_{k+1}(|\xi|r)}{|\xi|I_{k}(|\xi|)}\left(2( |\xi| r) K_k(|\xi|)I_{k+1}(|\xi|r) - r^{-k}\right)\\
        \leq{}& \frac{2}{|\xi|}( |\xi| r) K_{k+1}(|\xi|r) I_{k+1}(|\xi|r)+\frac{2I_{k+1}(|\xi|)}{|\xi|I_k(|\xi|)}
    \end{split}
\end{equation*}
since
$$(|\xi| r) K_k(|\xi|)I_{k+1}(|\xi|r) \leq |\xi| K_k(|\xi|)I_{k+1}(|\xi|)=1-|\xi|K_{k+1}(|\xi|)I_k(|\xi|),$$
such that 
$$|\xi|^q\mathcal{I}_{4} \lesssim 2 F^4_q(|\xi| r) + F^5_q(|\xi|), \qquad F^5_q(s):=\frac{2I_{k+1}(s)}{sI_k(s)}$$
(where we have replaced $k$ by $k+1$ in the definition of $F_q^4$). Because
\begin{equation*}
    F_0^5(s)\to\begin{cases}
     \frac{1}{k+1}, & s\to0,\\
     0, & s\to\infty,\\
    \end{cases}\qquad\qquad F_1^5(s)\to\begin{cases}
    0, & s\to0,\\
    2, & s\to\infty,
    \end{cases}
\end{equation*}
we conclude that $|\xi|^q\mathcal{I}_{4} \lesssim 1$.

Similarly
\begin{equation*}
    \begin{split}
        \mathcal{I}_5 :={}& \rho^{-k}\int_0^1 |\mathcal{D}_{-k}G_{k}(r,\rho)|\,r^{k+1}\,\mathrm{d}r\\
        ={}& |\xi|\tilde{K}_k(|\xi|\rho) \rho^{-k}\int_0^\rho I_{k+1}(|\xi|r)\,r^{k+1}\,\mathrm{d}r + |\xi| I_k (|\xi|\rho) \rho^{-k}\int_\rho^1 \tilde{K}_{k+1}(|\xi|r)\,r^{k+1}\,\mathrm{d}r\\
              ={}& \frac{2}{|\xi|}\tilde{K}_k(|\xi|\rho) (|\xi|\rho)^{1-k}\left(\mathbf{M}_{k+1}(|\xi|\rho)I_k (|\xi|\rho) - (2k+1)\mathbf{M}_{k}(|\xi|\rho)I_{k+1}(|\xi|\rho)\right) \\
        & \qquad\mbox{}+ \frac{1}{|\xi|}(|\xi|\rho)^{-k} (2k+1)\left(\mathbf{M}_{k}(|\xi|\rho) - \frac{\mathbf{M}_{k}(|\xi|)}{I_k (|\xi|)}I_k (|\xi|\rho)\right)\\
        \leq{}& \frac{2}{|\xi|}K_{k}(|\xi|\rho) (|\xi|\rho)^{1-k}\left(\mathbf{M}_{k+1}(|\xi|\rho)I_k (|\xi|\rho) - (2k+1)\mathbf{M}_{k}(|\xi|\rho)I_{k+1}(|\xi|\rho)\right)\\
        & \qquad\mbox{}+ \frac{1}{|\xi|}\frac{(2k+1)}{2^k\Gamma(k+1)}\left(2^{k-1}\sqrt{\pi}\Gamma(k+\tfrac{1}{2}) - \frac{\mathbf{M}_{k}(|\xi|)}{I_k (|\xi|)}\right),\\
    \end{split}
\end{equation*}
such that
\begin{equation*}
    |\xi|^{q}\mathcal{I}_{5} \leq 2 F_q^1(|\xi|\rho) + \frac{(2k+1)}{2k} F_q^3(|\xi|)
\end{equation*}
(where we have replaced $k$ by $k+1$ in the definition of $F_q^1$). It follows that $|\xi|^{q}\mathcal{I}_{5} \lesssim1$ for $k \in {\mathbb N}$,
and furthermore
$$
   \left.\frac{(2k+1)}{2k} F_q^3(s)\right|_{k=0} = \frac{\pi s^{q-1}}{2}\frac{\mathbf{L}_{0}(s)}{I_{0}(s)},
$$
such that  
   \begin{equation*}
   \left.\frac{(2k+1)}{2k} F_0^3(s)\right|_{k=0} \to \begin{cases}
        1 & s\to0,\\
        0, & s\to\infty,\\
    \end{cases}
    \qquad
     \left.\frac{(2k+1)}{2k} F_1^3(s)\right|_{k=0} \to \begin{cases}
        0 & s\to0,\\
        \frac{\pi}{2}, & s\to\infty,\\
    \end{cases}  
\end{equation*}
and hence $|\xi|^{q}\mathcal{I}_{5} \lesssim1$ for $q \in \{0,1\}$ and $k=0$.
\end{proof}

\begin{lemma}\label{lem:DG(f_k)-est}
The estimates
\begin{equation*}
    \|\mathcal{D}_{k}\mathcal{G}_k(f_k)\|_{L_1^2},\ \|\partial_{z}\mathcal{D}_{k}\mathcal{G}_k(f_k)\|_{L_1^2} \lesssim \|f_k\|_{L_1^2},\qquad k \in {\mathbb N},
 \end{equation*}
 and
 \begin{equation*}   
\|\mathcal{D}_{-k}\mathcal{G}_k(f_k)\|_{L_1^2},\  \|\partial_{z}\mathcal{D}_{-k}\mathcal{G}_k(f_k)\|_{L_1^2} \lesssim \|f_k\|_{L_1^2}\qquad k \in {\mathbb N}_0,
\end{equation*}
hold for each $f_k \in {\mathscr S}_{(k)}((0,1) \times {\mathbb R};{\mathbb C})$.
\end{lemma}
\begin{proof}
For fixed $q\in\{0,1\}$ we find that
$$\|\partial_z^q\mathcal{D}_{k}\mathcal{G}_k(f_k)\|^2_{L_1^2}  \lesssim \|f\|_{L_1^2}, \qquad
\|\partial_z^q\mathcal{D}_{-k}\mathcal{G}_k(f_k)\|^2_{L_1^2}  \lesssim \|f\|_{L_1^2}$$
by the method used in Lemma \ref{lem:G(f_k)-est}, replacing the Cauchy-Schwarz estimate \eqref{eq:CS} with respectively
$$\left||\xi|^{q} \!\!\int_0^1 \!\!\mathcal{D}_{k}G_k(r,\rho) \check{f}_k(\rho)\,\rho\,\mathrm{d}\rho\right|^2
\! \leq\!
\left(|\xi|^{q} r^{1-k}\!\!\int_0^1 |\mathcal{D}_{k}G_k(r,\rho) |\,\rho^{k}\,\mathrm{d}\rho\right)\!\!\left(|\xi|^{q} r^{k-1}\!\!\int_0^1 |\mathcal{D}_{k}G_k(r,\rho)| |\check{f}_k(\rho)|^2\,\rho^{2-k}\,\mathrm{d}\rho\right)
$$
and
$$
\left||\xi|^{q} \!\!\int_0^1 \!\! \mathcal{D}_{-k}G_k(r,\rho) \check{f}_k(\rho)\,\rho\,\mathrm{d}\rho\right|^2
\!\!\leq\!
\left(|\xi|^{q} r^{-k}\!\!\int_0^1 |\mathcal{D}_{-k}G_k(r,\rho) |\,\rho^{1+k}\,\mathrm{d}\rho\right)\!\!\left(|\xi|^{q} r^{k}\!\!\int_0^1 |\mathcal{D}_{-k}G_{k}(r,\rho)| |\check{f}_k(\rho)|^2\,\rho^{1-k}\,\mathrm{d}\rho\right)
$$
and using the estimates given in Proposition \ref{prop:DG-bound}.
\end{proof}

\begin{proposition}\label{prop:D2G-bound}$ $
\begin{itemize}
\item[(i)]
For each fixed $k \geq 2$ the estimates
$$
        \rho^{3-k}\int_0^1 |\mathcal{D}^{2}_{k}G_k(r,\rho)|\,r^{k-2}\,\mathrm{d}r,\ r^{3-k}\int_0^1 |\mathcal{D}^{2}_{k}G_k(r,\rho)|\,\rho^{k-2}\,\mathrm{d}\rho \lesssim 1
$$
hold uniformly over $r,\rho\in[0,1]$ and $\xi \in {\mathbb R}$.
\item[(ii)]
For each fixed $k \in {\mathbb N}_0$ the estimates
$$
        \rho^{1-k}\int_0^1 |\mathcal{D}^{2}_{-k}G_k(r,\rho)|\,r^{k}\,\mathrm{d}r,\ r^{1-k}\int_0^1 |\mathcal{D}^{2}_{-k}G_k(r,\rho)|\,\rho^{k}\,\mathrm{d}\rho \lesssim 1
$$
hold uniformly over $r,\rho\in[0,1]$ and $\xi \in {\mathbb R}$.
\end{itemize}
\end{proposition}
\begin{proof}
(i) This result follows from the calculations
\begin{equation*}
    \begin{split}
        \mathcal{I}_{6}:={}& \rho^{3-k}\int_0^1 |\mathcal{D}^{2}_{k}G_{k}(r,\rho)|\,r^{k-2}\,\mathrm{d}r\\
                ={}& |\xi|^2\tilde{K}_k(|\xi| \rho)  \rho^{3-k}\int_0^\rho I_{k-2} (|\xi| r)\,r^{k-2}\,\mathrm{d}r +  |\xi|^2 I_k(|\xi| \rho)\rho^{3-k}\int_\rho^1 \tilde{K}_{k-2}(|\xi|r)\,r^{k-2}\,\mathrm{d}r\\
        ={}& \frac{1}{(2k-3)}\bigg(2(k-1) + \frac{(|\xi|\rho)^{3-k}}{(2k-1)}\left(\mathbf{M}_{k}(|\xi|\rho) - (|\xi|\rho)^{k-1}\right)\\
        & \hspace{1.75cm}\mbox{}+ I_k (|\xi|\rho) (|\xi|\rho)^{3-k}\left(-\frac{2(k-1)}{|\xi|^{3-k}I_k (|\xi|)} - \frac{1}{I_k (|\xi|)}\left((2k-3)\mathbf{M}_{k-2}(|\xi|) - \frac{2(k-1)}{|\xi|}\mathbf{M}_{k-1}(|\xi|)\right)\!\right)\!\!\bigg)\\
        \leq{}& \frac{2(k-1)}{(2k-3)}\bigg(1 + \frac{(|\xi|\rho)^{3-k}I_k (|\xi|\rho)}{|\xi|^{3-k}I_k (|\xi|)}|\xi|^{2-k}\mathbf{M}_{k-1}(|\xi|)\bigg)\\
        \leq{}& \frac{4(k-1)}{(2k-3)}
    \end{split}
\end{equation*}
and
\begin{equation*}
    \begin{split}
        \mathcal{I}_{7}:={}& r^{3-k}\int_0^1 |\mathcal{D}^{2}_{k}G_{k}(r,\rho)|\,\rho^{k-2}\,\mathrm{d}\rho\\
        ={}& |\xi|^2\tilde{K}_{k-2}(|\xi| r)  r^{3-k}\int_0^r I_k (|\xi|\rho)\,\rho^{k-2}\,\mathrm{d}\rho +  |\xi|^2 I_{k-2}(|\xi| r)r^{3-k}\int_r^1 \tilde{K}_k(|\xi|\rho)\,\rho^{k-2}\,\mathrm{d}\rho\\
       ={}& 2(k-1) - (2k-3)(|\xi| r)^{3-k}\mathbf{M}_{k-2}(|\xi| r) + \frac{(|\xi| r)^{3-k} I_{k-2}(|\xi| r)}{(2k-1)I_k (|\xi|)}\left(\mathbf{M}_{k}(|\xi|) - |\xi|^{k-1}\right)\\
        \leq{}& 2(k-1),
    \end{split}
\end{equation*}
in which integrals \eqref{Block4I}, \eqref{Block4K} and \eqref{Block3I}, \eqref{Block3K} respectively have been used.

(ii) Similarly, this result follows from the calculations
\begin{equation*}
    \begin{split}
    \mathcal{I}_{8}:={}& \rho^{1-k}\int_0^1 |\mathcal{D}^{2}_{-k}G_{k}(r,\rho)|\,r^{k}\,\mathrm{d}r\\
    ={}& |\xi|^2\tilde{K}_k(|\xi|\rho) \rho^{1-k}\int_0^\rho I_{k+2}(|\xi|r) \,r^{k}\,\mathrm{d}r + |\xi|^2I_k (|\xi|\rho) \rho^{1-k}\int_\rho^1 \tilde{K}_{k+2}(|\xi|r)\,r^{k}\,\mathrm{d}r\\
       ={}& 2(k+1) - (2k+1)(|\xi|\rho)^{1-k}\mathbf{M}_{k}(|\xi|\rho) -\frac{2(k+1)(|\xi|\rho)^{1-k}I_k (|\xi|\rho)}{|\xi|^{1-k}I_k (|\xi|)} \\
       & \qquad\mbox{}+ \frac{(2k+1)(|\xi|\rho)^{1-k}I_k (|\xi|\rho)}{|\xi|^{1-k}I_k (|\xi|)}|\xi|^{1-k}\mathbf{M}_{k}(|\xi|)\\
    \leq{}& 2(k+1) + \frac{(2k+1)(|\xi|\rho)^{1-k}I_k (|\xi|\rho)}{|\xi|^{1-k}I_k (|\xi|)}|\xi|^{1-k}\mathbf{M}_{k}(|\xi|)\\
    \leq{}& 4k+3
    \end{split}
\end{equation*}
and
\begin{equation*}
    \begin{split}
        \mathcal{I}_{9}:={}& r^{1-k}\int_0^1 |\mathcal{D}^{2}_{-k}G_{k}(r,\rho)|\,\rho^{k}\,\mathrm{d}\rho\\
        ={}& |\xi|^2\tilde{K}_{k+2}(|\xi|r) r^{1-k}\int_0^r I_k (|\xi|\rho)\,\rho^{k}\,\mathrm{d}\rho + |\xi|^2I_{k+2}(|\xi|r) r^{1-k}\int_r^1 \tilde{K}_k(|\xi|\rho)\,\rho^{k}\,\mathrm{d}\rho\\
          ={}& \frac{1}{2k+1}\bigg(2(k+1) + (|\xi|r)^{1-k}\left(-\frac{2(k+1)}{(|\xi|r)}\mathbf{M}_{k+1}(|\xi|r) + (2k+1)\mathbf{M}_{k}(|\xi|r)\right) \\
          &\hspace{1.6cm}\mbox{}-(|\xi|r)^{1-k} I_{k+2}(|\xi|r) (2k+1)\frac{\mathbf{M}_{k}(|\xi|)}{I_k (|\xi|)}\bigg)\\
        \leq{}& \frac{1}{2k+1}\bigg(2(k+1) + (2k+1)(|\xi|r)^{1-k}\mathbf{M}_{k}(|\xi|r)\bigg)\\
        \leq{}& \frac{4k+3}{2k+1},
    \end{split}
\end{equation*}
in which integrals \eqref{Block3I}, \eqref{Block3K} and \eqref{Block2I}, \eqref{Block2K} respectively have been used.
\end{proof}

\begin{lemma}\label{lem:D2G(f_k)-est}
The estimates
\begin{equation*}
    \|\mathcal{D}_{-k+i}^{2-i}\mathcal{D}^{i}_{k}\mathcal{G}_k(f_k)\|_{L_1^2} \lesssim \|f_k\|_{L_1^2},\qquad 0 \leq i \leq 2,
\end{equation*}    
hold for each $f_k \in {\mathscr S}_{(k)}((0,1) \times {\mathbb R};{\mathbb C})$ and $k \in {\mathbb N}_0$.
\end{lemma}
\begin{proof}
Observe that
\begin{equation*}\begin{split}
    \mathcal{D}^{2-i}_{-k+i}\mathcal{D}^{i}_{k}\mathcal{G}_k(f_k) ={}& \mathcal{F}^{-1}\left[\int_0^1 \mathcal{D}^{2-i}_{-k+i}\mathcal{D}^{i}_{k}G_k(r,\rho)\,\rho\,\check{f}_k(\rho)\,\rho\,\mathrm{d}\rho\right] + f_k.
\end{split}
\end{equation*}

The case $i=1$ is treated by noting that
$$
\mathcal{D}_{1-k}\mathcal{D}_{k}G_k(r,\rho) = \mathcal{D}_{1+k}\mathcal{D}_{-k}G_k(r,\rho) = |\xi|^2 G_k(r,\rho),
$$
such that
$$
        \|\mathcal{D}_{1-k}\mathcal{D}_{k}\mathcal{G}_k(f_k) - f_k\|^2_{L_1^2} \leq \||\xi|^2\mathcal{G}_k(f_k)\|^2_{L_1^2} = \|\partial_z^2 \mathcal{G}_1(f_k)\|^2 \lesssim \|f_k\|_{L_1^2}
$$
for $k \in {\mathbb N}_0$, while the case $i=2$, $k=1$ is treated similarly by noting that
$$
\mathcal{D}^2_{1}G_{1}(r,\rho) = \mathcal{D}_0\mathcal{D}_{1}G_{1}(r,\rho) = |\xi|^2 G_{1}(r,\rho),
$$
such that
$$
 \|\mathcal{D}^2_1\mathcal{G}_1(f_1) - f_1\|^2_{L_1^2} \leq \||\xi|^2 \mathcal{G}_1(f_1)\|_{L_1^2} ^2 = \|\partial_z^2 \mathcal{G}_1(f_1)\|^2 \lesssim \|f_1\|_{L_1^2}.
 $$

It therefore remains to examine $\mathcal{D}^2_{k}G_{k}(r,\rho)$ $(i=2)$ for $k\geq2$ and $\mathcal{D}^2_{-k}G_{k}(r,\rho)$ ($i=0$) for $k\in\mathbb{N}_{0}$. Using the calculation
\begin{equation*}
    \mathcal{D}^{2}_{\pm k}G_k(r,\rho) = \begin{cases}
        \displaystyle-|\xi|^2 I_{k\mp2}(|\xi|r) \tilde{K}_k(|\xi|\rho), & 0<r<\rho,\\
        \displaystyle -|\xi|^2 I_k (|\xi|\rho)\tilde{K}_{k\mp2}(|\xi|r), & \rho<r<1,
    \end{cases}
\end{equation*}
we find that
\begin{equation*}
    \begin{split}
        \|\mathcal{D}^2_{k}\mathcal{G}_k(f_k) - f_k\|^2_{L_1^2} \leq{}& \int_{-\infty}^{\infty}\int_0^1 \left|\int_0^1 \mathcal{D}^2_{k}G_k(r,\rho) \check{f}_k(\rho)\,\rho\,\mathrm{d}\rho\right|^2\,r\,\mathrm{d}r\,\mathrm{d}\xi,\\
        \leq{}& \int_{-\infty}^{\infty}\int_0^1 \left(r^{3-k}\int_0^1 |\mathcal{D}^2_{k}G_k(r,\rho) |\,\rho^{k-2}\,\mathrm{d}\rho\right)\!\!\!\left(r^{k-3}\int_0^1 |\mathcal{D}^2_{k}G_k(r,\rho)| |\check{f}_k(\rho)|^2\,\rho^{4-k}\,\mathrm{d}\rho\right)\,r\,\mathrm{d}r\,\mathrm{d}\xi,\\
       \lesssim{}& \|f_k\|_{L_1^2}
       \end{split}
\end{equation*}
for $k \geq 2$ and
\begin{equation*}
    \begin{split}
        \|\mathcal{D}^2_{-k}\mathcal{G}_k(f_k) - f_k\|^2_{L_1^2} \leq{}& \int_{-\infty}^{\infty}\int_0^1 \left|\int_0^1 \mathcal{D}^2_{-k}G_k(r,\rho) \check{f}_k(\rho)\,\rho\,\mathrm{d}\rho\right|^2\,r\,\mathrm{d}r\,\mathrm{d}\xi,\\
        \leq{}& \int_{-\infty}^{\infty}\int_0^1 \left(r^{1-k}\int_0^1 |\mathcal{D}^2_{-k}G_k(r,\rho) |\,\rho^{k}\,\mathrm{d}\rho\right)\!\!\!\left(r^{k-1}\int_0^1 |\mathcal{D}^2_{-k}G_k(r,\rho)| |\check{f}_k(\rho)|^2\,\rho^{2-k}\,\mathrm{d}\rho\right)\,r\,\mathrm{d}r\,\mathrm{d}\xi,\\
       \lesssim{}& \|f_k\|_{L_1^2}
       \end{split}
\end{equation*}
for $k \in {\mathbb N}_0$ (in both cases the third line follows by the method used in Lemma \ref{lem:G(f_k)-est} and the estimates given in Proposition \ref{prop:D2G-bound}).
\end{proof}

Theorem \ref{lo estimates}(ii) is similarly proved two steps (Lemmata \ref{lem:Hthm1} and \ref{lem:Hthm2}).

\begin{lemma} \label{lem:Hthm1}
The estimates
    \begin{equation*}
        \|\mathcal{B}_k(g)\|_{L_1^2}, \ \|\partial_z \mathcal{B}_k(g)\|_{L_1^2}, \  \|\mathcal{D}_{k} \mathcal{B}_k(g)\|_{L_1^2}, \  \|\mathcal{D}_{-k} \mathcal{B}_k(g)\|_{L_1^2} \lesssim \|g\|_{\mathrm{H}^{\frac{1}{2}}}
    \end{equation*}
hold for each $g \in {\mathscr S}({\mathbb R};{\mathbb C})$ and $k \in {\mathbb N}_0$.
\end{lemma}
\begin{proof}
Observe that
    \begin{equation*}
        \begin{split}
            \|\mathcal{B}_{k}(g)\|^2_{L_1^2} + \|\partial_z\mathcal{B}_{k}(g)\|^2_{L_1^2} ={}& \int_{-\infty}^{\infty}\int_0^1 (1 + |\xi|^2) \left(\frac{I_k(|\xi|r)}{I_k(|\xi|)}\right)^2\,|\check{g}|^2\,r\,\mathrm{d}r\,\mathrm{d}\xi,\\
            ={}& \int_{-\infty}^{\infty} \frac{1}{2}(1 + |\xi|^2) \left(1 - \frac{I_{k+1}(|\xi|)I_{k-1}(|\xi|)}{I_k^2(|\xi|)}\right)\,|\check{g}|^2\,\mathrm{d}\xi \\
            \lesssim{}& \int_{-\infty}^{\infty} \frac{1}{2}(1 + |\xi|^2)^\frac{1}{2}|\check{g}|^2\,\mathrm{d}\xi\\
            ={}& \|g\|_{H^{\frac{1}{2}}}^2
        \end{split}
    \end{equation*}
because
\begin{equation*}
\frac{1}{2}(1 + s^2)^{\frac{1}{2}}\left(1 - \frac{I_{k+1}(s)I_{k-1}(s)}{I_k^2(s)}\right)
\to\begin{cases}
     \frac{1}{2(k+1)}, & s\to0,\\
     \frac{1}{2}, & s\to\infty.
    \end{cases}
\end{equation*}

Furthermore
\begin{equation*}
    \mathcal{D}_{\pm k}\mathcal{B}_{k}(g) = \mathcal{F}^{-1}\left[ |\xi| \frac{I_{k\mp1}(|\xi|r)}{I_k(|\xi|)} \check{g}\right],
\end{equation*}
such that
    \begin{equation*}
        \begin{split}
            \|\mathcal{D}_{\pm k}\mathcal{B}(g)\|^2_{L_1^2} ={}& \int_{-\infty}^{\infty}\int_0^1 |\xi|^2 \left(\frac{I_{k\mp1}(|\xi|r)}{I_k(|\xi|)}\right)^2\,|\check{g}|^2\,r\,\mathrm{d}r\,\mathrm{d}\xi,\\
            ={}& \int_{-\infty}^{\infty}\frac{1}{2} |\xi|^2 \left(\frac{I^2_{k\mp1}(|\xi|) - I_k (|\xi|)I_{k\mp2}(|\xi|)}{I^2_k(|\xi|)}\right)\,|\check{g}|^2\,\mathrm{d}\xi,\\
                        \lesssim{}& \int_{-\infty}^{\infty}(1+|\xi|^2)^{\frac{1}{2}}\,|\check{g}|^2\,\mathrm{d}\xi,\\
            ={}& \|g\|^2_{H^{\frac{1}{2}}}
        \end{split}
    \end{equation*}
because
\begin{equation*}
 \frac{1}{2} s^2(1+s^2)^{-\frac{1}{2}} \left(\frac{I^2_{k-1}(s) - I_k (s)I_{k- 2}(s)}{I^2_k(s)}\right)
 \to\begin{cases}
     2k, & s\to0,\\
     \frac{1}{2}, & s\to\infty
    \end{cases}
\end{equation*}
and
\begin{equation*}
 \frac{1}{2} s^2(1+s^2)^{-\frac{1}{2}} \left(\frac{I^2_{k+1}(s) - I_k (s)I_{k+ 2}(s)}{I^2_k(s)}\right)
 \to\begin{cases}
     0, & s\to0,\\
     \frac{1}{2}, & s\to\infty.
    \end{cases}
\end{equation*}
\end{proof}
\begin{lemma} \label{lem:Hthm2}
The estimates
    \begin{equation*}
        \|\partial_z^2 \mathcal{B}_{k}(g)\|_{L_1^2},\ \|\partial_z \mathcal{D}_{\pm k}\mathcal{B}_{k}(g)\|_{L_1^2},\ \|\mathcal{D}^{2-i}_{-k+i}\mathcal{D}^{i}_{k}\mathcal{B}_{k}(g)\|_{L_1^2} \lesssim \|g\|_{H^{\frac{3}{2}}}, \qquad 0 \leq i \leq 2,
    \end{equation*}
hold for each $g \in {\mathscr S}({\mathbb R};{\mathbb C})$ and $k \in {\mathbb N}_0$.
\end{lemma}
\begin{proof}
Clearly
\begin{align*}
\|\partial_z^2\mathcal{B}_{k}(g)\|^2_{L_1^2} &= \int_{-\infty}^{\infty}\int_0^1 |\xi|^4 \left(\frac{I_k(|\xi|r)}{I_k(|\xi|)}\right)^2\,|\check{g}|^2\,r\,\mathrm{d}r\,\mathrm{d}\xi= \|\partial_z\mathcal{B}_{k}(\partial_z g)\|^2_{L_1^2} \lesssim \|\partial_z g\|_{H^\frac{1}{2}}^2
\leq \|g\|_{H^\frac{3}{2}}^2, \\
            \|\partial_z\mathcal{D}_{\pm k}\mathcal{B}_{k}(g)\|^2_{L_1^2} &= \int_{-\infty}^{\infty}\int_0^1 |\xi|^4 \left(\frac{I_{k\mp 1}(|\xi|r)}{I_k(|\xi|)}\right)^2\,|\check{g}|^2\,r\,\mathrm{d}r\,\mathrm{d}\xi
            = \|\mathcal{D}_{\pm k}\mathcal{B}_{k}(\partial_z g)\|^2_{L_1^2} \lesssim \|\partial_z g\|_{H^\frac{1}{2}}^2
\leq \|g\|_{H^\frac{3}{2}}^2
 \end{align*}
and  
$$
    \|\mathcal{D}_{1- k}\mathcal{D}_{k}\mathcal{B}_{k}(g)\|_{L_1^2}^2 = 
    \int_{-\infty}^{\infty}\int_0^1 |\xi|^4 \left(\frac{I_k(|\xi|r)}{I_k(|\xi|)}\right)^2\,|\check{g}|^2\,r\,\mathrm{d}r\,\mathrm{d}\xi
= \|\partial_z^2\mathcal{B}_{k}(g)\|^2_{L_1^2}
            \lesssim \|g\|^2_{H^{\frac{3}{2}}}.
$$

Furthermore,
    \begin{equation*}
    \mathcal{D}^{2}_{\pm k}\mathcal{B}_{k}(g) = \mathcal{F}^{-1}\left[ |\xi|^2 \frac{I_{k\mp2}(|\xi|r)}{I_k(|\xi|)} \check{g}\right],
\end{equation*}
such that
\begin{equation*}
        \begin{split}
            \|\mathcal{D}_{\pm k}^{2}\mathcal{B}_{k}(g)\|^2_{L_1^2} ={}& \int_{-\infty}^{\infty}\int_0^1 |\xi|^4 \left(\frac{I_{k\mp 2}(|\xi|r)}{I_k(|\xi|)}\right)^2\,|\check{g}|^2\,r\,\mathrm{d}r\,\mathrm{d}\xi,\\
            ={}& \int_{-\infty}^{\infty}\frac{1}{2} |\xi|^4 \left(\frac{I^2_{k\mp2}(|\xi|) - I_{k\mp1}(|\xi|)I_{k\mp 3}(|\xi|)}{I^2_k(|\xi|)}\right)\,|\check{g}|^2\,\mathrm{d}\xi,\\
 \lesssim{}& \int_{-\infty}^{\infty} (1+|\xi|^2)^{\frac{3}{2}}\,|\check{g}|^2\,\mathrm{d}\xi \\
            ={}& \|g\|^2_{H^{\frac{3}{2}}}
        \end{split}
    \end{equation*}
because
    \begin{equation*}
     \frac{1}{2} s^4 (1 + s^2)^{-\frac{3}{2}}\left(\frac{I^2_{k\mp2}(s) - I_{k\mp1}(s)I_{k\mp3}(s)}{I^2_k(s)}\right) 
      \to\begin{cases}
     0, & s\to 0,\\
     \frac{1}{2}, & s\to \infty.
    \end{cases}
    \end{equation*}
    
\end{proof}

\appendix
\section{Young's inequality for radial convolutions}
\numberwithin{equation}{section}

\begin{lemma}[Young's inequality]
Suppose that $X$ is a Banach algebra and that $p$, $q$, $r \in [1,\infty)$ with $\frac{1}{p}+\frac{1}{q}=1+\frac{1}{r}$. The inequality
$$\|f_k \ast g_\ell \|_{L_1^{r}} \leq \|f_k\|_{L_1^{p}} \|g_\ell\|_{L_1^{q}}$$
holds for all  $f_k\in L_1^{p}((0,\infty);X)$ and $g_\ell\in L_1^{q}((0,\infty);X)$, where
$$L_1^p((0,\infty);X) = \left\{f: [0,\infty) \rightarrow X\st 2\pi\int_0^\infty \|f(r)\|^p r \dr< \infty\right\}, \qquad p \geq 1.$$
\end{lemma}

\begin{proof}
Writing
\begin{align*}
D(\rho,&u,w)uw  \|f_k(u)\| \|g_\ell(w)\| \\
&=
\big(D(\rho,u,w)uw  \|f_k(u)\|^p \|g_\ell(w)\|^q\big)^{\frac{1}{r}}\big(D(\rho,u,w)^{\frac{1}{p}}(uw)^\frac{1}{p}  \|f_k(u)\|\big)^{1-\frac{p}{r}}\big(D(\rho,u,w)^{\frac{1}{q}}(uw)^\frac{1}{q} \|g_\ell(w)\|\big)^{1-\frac{q}{r}}
\end{align*}
and observing that
$$\frac{1}{r}+\frac{1}{\dfrac{p}{r-\frac{p}{r}}}+\frac{1}{\dfrac{q}{1-\frac{q}{r}}}=1,$$
we find from the generalised H\"{o}lder inequality that
\begin{align}
\int_0^\infty &\!\!\! \int_0^\infty D(\rho,u,w)  \|f_k(u)\| \|g_\ell(w)\| uw\du \dw \nonumber \\
& \leq \left(\int_0^\infty\!\!\!\int_0^\infty D(\rho,u,w)  \|f_k(u)\|^p \|g_\ell(w)\|^q uw\du \dw)\right)^{\frac{1}{r}}
\left(\int_0^\infty\!\!\!\int_0^\infty D(\rho,u,w)  \|f_k(u)\|^p uw\du \dw)\right)^{\frac{1-\frac{p}{r}}{p}} \nonumber \\
& \qquad\qquad \times
\left(\int_0^\infty\!\!\!\int_0^\infty D(\rho,u,w) \|g_\ell(w)\|^q uw\du \dw)\right)^{\frac{1-\frac{q}{r}}{q}}. \label{First estimate}
\end{align}

Next note that
\begin{align*}
\int_0^\infty\!\!\!\int_0^\infty D(\rho,u,w)  \|f_k(u)\|^p uw \du \dw
& = \int_0^\infty\!\!\!\int_0^\infty D(\rho,u,w) w\dw\,   \|f_k(u)\|^p u\du = \int_0^\infty   \|f_k(u)\|^p u\du,\\
\int_0^\infty\!\!\!\int_0^\infty D(\rho,u,w) \|g_\ell(w)\|^q uw \du \dw
& = \int_0^\infty\!\!\!\int_0^\infty D(\rho,u,w) u \du\,  \|g_\ell(w)\|^q w\dw = \int_0^\infty  \|g_\ell(w)\|^q w\dw.
\end{align*}
Inserting these formulae into \eqref{First estimate}, raising both sides to the power $r$, multiplying by $\rho$ and integrating yields
\begin{align}
\int_0^\infty\!\! &\bigg(\int_0^\infty \!\!\! \int_0^\infty D(\rho,u,w)  \|f_k(u)\| \|g_\ell(w)\| \du \dw\bigg)^r \rho \drho \nonumber \\
& \leq \int_0^\infty\!\!\! \int_0^\infty\!\!\! \int_0^\infty D(\rho,u,w)  \|f_k(u)\|^p \|g_\ell(w)\|^q \rho u w\du\dw\drho\,
\bigg(\int_0^\infty   \|f_k(u)\|^p u\du\bigg)^{\!\!\frac{r-p}{p}}
\bigg(\int_0^\infty  \|g_\ell(w)\|^q w\dw\bigg)^{\!\!\frac{r-q}{q}}\!\!\!\!\!. \label{Second estimate}
\end{align}
Furthermore
\begin{align*}
\int_0^\infty\!\!\! \int_0^\infty\!\!\! \int_0^\infty D(\rho,u,w)  \|f_k(u)\|^p \|g_\ell(w)\|^q \rho u w\du\dw\drho
& = \int_0^\infty\!\!\! \int_0^\infty\!\!\! \int_0^\infty D(\rho,u,w)\rho\drho\,   \|f_k(u)\|^p \|g_\ell(w)\|^quw\du\dw \\
& = \int_0^\infty   \|f_k(u)\|^p u\du \int_0^\infty  \|g_\ell(w)\|^q w\dw,
\end{align*}
and inserting this formula into \eqref{Second estimate} yields
\begin{equation}
\int_0^\infty \bigg(\int_0^\infty \!\!\! \int_0^\infty D(\rho,u,w)  \|f_k(u)\| \|g_\ell(w)\| uw \du \dw\bigg)^r \rho \drho
\leq \bigg( \int_0^\infty   \|f_k(u)\|^p u \du\bigg)^{\frac{r}{p}}\int_0^\infty  \|g_\ell(w)\|^q w\dw\bigg)^{\frac{r}{q}}. \label{Third estimate}
\end{equation}

The result follows from \eqref{Third estimate} and
\begin{equation*}
 \|(f_k \ast g_\ell)(\rho)\|
\leq 2\pi\int_0^\infty \!\!\! \int_0^\infty D(\rho,u,w)  \|f_k(u)\| \|g_\ell(w)\| uw\du \dw. \qedhere
\end{equation*}
\end{proof}

\begin{remark}
Similar arguments yield the corresponding result in the case that one or two of $p$, $q$, $r$ are equal to $\infty$.
\end{remark}\pagebreak

\end{document}